\documentclass[11pt, leqno]{article}
\pdfoutput=1

\usepackage{chngcntr}
\usepackage{amssymb}
\usepackage{amsmath}
\usepackage{enumerate}
\usepackage{geometry}
\usepackage{bm}
\usepackage{appendix}
\usepackage[english]{babel}
\usepackage[utf8]{inputenc}
\usepackage{amsthm}
\RequirePackage[colorlinks,linkcolor=blue, citecolor=blue,urlcolor=blue]{hyperref}
\usepackage{tikz}
\usepackage{mathtools}
\usepackage{xspace}
\usepackage{enumitem}
\usepackage{rotating} 
\usepackage[all]{xy}

\RequirePackage{natbib}

\usepackage{subcaption}

\numberwithin{equation}{section}
\allowdisplaybreaks

\newtheorem{Theorem}{Theorem}[section]

\newtheorem{Lemma}[Theorem]{Lemma}
\newtheorem{lma}[Theorem]{Lemma}

\newtheorem{Proposition}[Theorem]{Proposition}
\newtheorem{prop}[Theorem]{Proposition}

\newtheorem{Definition}[Theorem]{Definition}
\newtheorem{defn}[Theorem]{Definition}

\theoremstyle{remark}
\newtheorem{Remark}[Theorem]{Remark}

\newtheorem{innercustomass}{Assumption}
\newenvironment{customass}[1]
  {\renewcommand\theinnercustomass{#1} \innercustomass }
  {\endinnercustomass}

\newcommand{\MAFE}{{\text{MAFE}}}
\newcommand{\MSFE}{{\text{MSFE}}}

\newcommand{\bH}{{\boldsymbol H}}
\renewcommand{\b}{\boldsymbol}
\newcommand{\ones}{{\b 1}}
\newcommand{\bchi}{{\b \chi}}
\newcommand{\bxi}{{\b \xi}}
\newcommand{\bmu}{{\b \mu}}

\newcommand{\ba}{{\b a}}
\newcommand{\bB}{{\b B}}
\newcommand{\bb}{{\b b}}
\newcommand{\bX}{{\b X}}
\newcommand{\bx}{{\b x}}
\newcommand{\bY}{{\b Y}}
\newcommand{\by}{{\b y}}
\newcommand{\bu}{{\b u}}
\newcommand{\bU}{{\b U}}

\newcommand{\bv}{{\b v}}
\newcommand{\bW}{{\b W}}

\newcommand{\bz}{{\b z}}
\newcommand{\bP}{{\b P}}
\newcommand{\bp}{{\b p}}
\newcommand{\bF}{{\b F}}
\newcommand{\bQ}{{\b Q}}
\newcommand{\bD}{{\b D}}
\newcommand{\be}{{\b e}}
\newcommand{\bI}{{\b I}}
\newcommand{\bR}{{\b R}}
\newcommand{\bM}{{\b M}}
\newcommand{\tildeR}{\tilde{\b R}}
\newcommand{\bLambda}{{\b \Lambda}}
\newcommand{\hatbu}[1]{\hat{\bu}^{(#1)}}
\newcommand{\hatbU}[1]{\hat{\bU}^{(#1)}}
\newcommand{\tildebu}[1]{\tilde{\bu}^{(#1)}}
\newcommand{\tildebU}[1]{\widetilde{\bU}^{(#1)}}
\newcommand{\tildeB}[1]{\widetilde{\bB}^{(#1)} }
\newcommand{\tildeb}[1]{\tilde{\bb}^{(#1)} }
\newcommand{\barB}[1]{\overline{\bB}^{(#1)} }

\newcommand{\defeq}{ \vcentcolon = }
\newcommand{\eqdef}{=\vcentcolon}
\newcommand{\IC}{\mathrm{IC}}

\DeclareMathOperator*{\argmin}{arg\,min}

\newcommand{\tp}{\mathsf{T}}
\newcommand{\adjoint}{\star}

\DeclareMathOperator*{\tensor}{\otimes}  
\newcommand{\abs}[1]{\left\lvert #1 \right\rvert}
\renewcommand{\sc}[1]{\left\langle #1 \right\rangle}

\newcommand{\KL}{{Karhunen--Lo\`{e}ve }}

\newcommand{\LL}[1]{{L^2\left( #1, \mathbb{R} \right)}\xspace}
\newcommand{\LLP}{L^2\!\left( \Omega \right)\xspace}
\newcommand{\LLPH}[1]{L^2_{#1}\left( \Omega \right)\xspace}
\newcommand{\LLR}{{\LL{[0,1]}}}

\DeclareMathOperator{\pp}{ {\mathbb{P} }}
 \DeclareMathOperator{\ee}{ \:{\mathbb{E} }}
\newcommand{\eee}[1]{{\mathbb{E}\left[ #1 \right]}}
\DeclareMathOperator{\cov}{ {\mathrm{cov}}}

\newcommand{\convp}{\stackrel{{\rm P}}{\longrightarrow}}
\newcommand{\convP}{\stackrel{{\rm P}}{\longrightarrow}}

\newcommand{\identity}{\mathrm{Id}}
\newcommand{\canonicalProj}{\mathcal{P}}
\newcommand{\singleProj}{\mathsf{P}}

\newcommand{\operator}[1]{\mathsf{L}( #1 )} 
\newcommand{\matrixOf}[1]{\mathsf{Mat}( #1 )} 

\DeclareMathOperator{\row}{\mathrm{row}}
\DeclareMathOperator{\col}{\mathrm{col}}
\DeclareMathOperator{\trace}{{\mathrm{Tr}}}

\DeclareMathOperator{\proj}{\mathrm{proj}}
\DeclareMathOperator{\vspan}{\mathrm{span}}

\newcommand{\schatten}{\mathcal{S}}
\newcommand{\tc}[1]{\schatten_1(#1)}
\newcommand{\bounded}[1]{\mathcal{L}(#1)}
\newcommand{\HS}[1]{\schatten_2(#1)}
\newcommand{\compact}[1]{\schatten_\infty(#1)}

\newcommand{\hnorm}[1]{{\left\lVert #1 \right\rVert}}   
\newcommand{\rpnorm}[1]{{\left| #1 \right|}}   
\newcommand{\snorm}[1]{{\left\vert\kern-0.2ex\left\vert\kern-0.2ex\left\vert #1
  \right\vert\kern-0.2ex\right\vert\kern-0.2ex\right\vert}}
\newcommand{\ssnorm}[1]{{\vert\kern-0.2ex\vert\kern-0.2ex\vert #1
  \vert\kern-0.2ex\vert\kern-0.2ex\vert}} 
\newcommand{\opnorm}[1]{\snorm{#1}_\infty}
\newcommand{\tracenorm}[1]{\snorm{#1}_1}

\newcommand{\Fnorm}[1]{\snorm{#1}_2} 

\newcommand{\Lipnorm}[1]{\hnorm{#1}_\mathrm{Lip}} 

\newcommand{\vfi}{{\varphi}}
\newcommand{\vep}{{\varepsilon}}

\newcommand{\bbR}{\mathbb{R}}
\newcommand{\bbZ}{\mathbb{Z}}

\newcommand{\bbN}{\mathbb{N}}
\newcommand{\LipschitzSpace}{\mathsf{\Lambda}_1}

\newcommand{\image}{\mathrm{Im}}

\newcommand{\thesupplement}{Appendices} 

\usepackage{authblk}

\begin{document}

\title{Factor Models for High-Dimensional Functional Time Series}
\author[1]{Shahin Tavakoli\footnote{address for correspondence: s.tavakoli@warwick.ac.uk}}
\author[2]{Gilles Nisol}
\author[2]{Marc Hallin}

\affil[1]{Department of Statistics, University of Warwick, UK}
\affil[2]{ECARES and D\' epartement de Math\' ematique\\ Universit\' e libre de
Bruxelles, Belgium}

\maketitle

\begin{abstract}
    In this paper, we set up the theoretical foundations for a high-dimensional
    functional factor model approach in the analysis of large  cross-sections (panels) of
    functional time series (FTS).  We first establish a representation result
    stating that, under mild assumptions on the covariance operator of the cross-section,
     we can represent each FTS as the sum of a common
    component driven by scalar factors loaded via functional loadings,
    and a mildly cross-correlated idio\-syncratic component.
    Our model and theory are developed in a general Hilbert space setting that allows for mixed 
    panels of functional and scalar time series.   We then turn to the
    identification of the number of factors, and the
    estimation of the factors, their loadings, and the common components. 
    We provide a family of information criteria for identifying the number of
    factors, and prove their consistency.
  We provide average error bounds for the estimators of the factors, loadings, and common component; our
    results encompass   the scalar case, for which they reproduce and extend,
    under weaker conditions, well-established similar results.
     Under slightly stronger assumptions, we also provide uniform
    bounds for the estimators of factors, loadings, and common component, thus
    extending existing scalar results.  
    Our consistency results in the asymptotic regime where the number $N$ of
    series  and the number~$T$ of time observations  diverge thus extend to the functional context 
    the ``blessing of dimensionality'' that explains the success of factor
    models in the analysis  of high-dimensional (scalar) time series.
    We provide numerical illustrations that
    corroborate the convergence rates predicted by the theory, and provide
    finer understanding of the interplay between $N$ and $T$ for estimation
    purposes. We conclude with an application to forecasting mortality curves,
    where we demonstrate that our approach outperforms existing methods. 
\end{abstract}



\section{Introduction}
Throughout the last decades, researchers have been dealing with datasets of increasing size
and complexity. In particular, Functional Data Analysis \citep[
see
e.g.][]{ramsay:silverman:2005,ferraty:vieu:2006,horvath:2012:book,hsing:eubank:2015,wang:2015review}
has received
much interest and, in view of its relevance in a number of applications,  has gained fast-growing popularity. In Functional Data Analysis, the observations
are taking values in some functional space, usually some Hilbert space~$H$---often,
in practice,  the space $\LLR$ of real-valued squared-integrable functions. When an ordered sequence of
functional observations exhibits serial dependence, we enter the realm of
\emph{Functional Time Series} (FTS) \citep{hormann:kokoszka:2010,hormann:kokoszka:2012}. Many standard univariate and low-dimensional multivariate 
time-series  methods have been adapted to this functional setting, either using a
time-domain approach
\citep{kokoszka:reimherr:2013norm,kokoszka:reimherr:2013det,hormann:horvath:reeder:2013,aue2014dependent,Aue:2015:prediction,horvath2016adaptive,Aue2017:GARCH,Gorecki:2018:testing,bucher2019detecting,gao:shang:2018},
a frequency domain approach under stationarity assumptions
\citep{panaretos:tavakoli:2013SPA,panaretos:tavakoli:2013AOS,hormann:kidzinski:hallin:2015,tavakoli2016detecting,hormann:kokoszka:nisol:2018,rubin2020sparsely,fang2020new}
or under local stationarity assumptions
\citep{van2017nonparametric,vandelft:2018locally,vandelft2021similarity}.

Parallel to this development of functional time series  analysis, data in high
dimensions \citep[e.g.][]{Buhlmann:VanDeGeer:2011,fan:2013:jrssb-discussion} have become
pervasive in data sciences and related disciplines where, under the name of Big Data,
they constitute one of the most active subjects of contemporary statistical research. 

This contribution stands at the intersection of these two strands of literature,
cumulating the challenges of function-valued observations and those of high
dimension. Datasets, in this context, consist of large collections of $N$ scalar or
functional time series---equivalently, scalar or functional time series in high dimension (from
fifty, say, to several hundreds)---observed over a period of time $T$. Typical
examples are continuous-time  series of concentrations for a large number of
pollutants, or/and collected daily over a large number of sites, daily series of  returns
observed at high intraday frequency  for a large collection of stocks, or intraday
energy consumption curves\footnote{available, for instance, at~\texttt{data.london.gov.uk/dataset/smartmeter-energy-use-data-in-} \texttt{london-households}}, to name only
a few. Not all component series in the dataset are required to be function-valued, though, and it is very important for practice that mixed panels\footnote{We are adopting here the convenient econometric terminology: a {\it panel} is the finite $N\times T$ realization of a double-indexed stochastic process  $\{x_{it}\vert i\in \mathbb{N}, t\in \mathbb{Z}\}$, where $i$ is a cross-sectional index and $t$ stands for time---equivalently, an $N\times T$ array the columns of which constitute the finite realization $({\bx}_{1}, \ldots, {\bx}_{T})$ of some $N$-dimensional time series ${\bx}_t=(x_{1t}, \ldots,  x_{Nt})^{\tp}$,  with functional or scalar components.} of scalar and functional series can be considered as well. 
In order to model such datasets, we develop a class of \emph{high-dimensional functional
factor models} 
inspired by the factor model approaches developed, mostly,
in time series econometrics, which have proven effective, flexible, and quite
efficient in the scalar case.

Since the econometric literature may be unfamilar to much of the readership of this journal, a general presentation of  factor model methods for high-dimensional time series is in order.  The basic idea of the approach  consists of decomposing the observation into the sum of two mutually orthogonal unobserved components: the {\it common component}, a reduced-rank process driven by a small number of {\it common shocks} which  account for most of the cross-correlations, and the {\it idiosyncratic component} which, consequently, 
 is only mildly cross-correlated. 
Early instances  of this can be traced back to the
pioneering contributions by  \citet{geweke:1977}, \citet{sargent:sims:1977},
\citet{chamberlain:1983}, and \citet{chamberlain:rothschild:1983}. The factor models
considered by  Geweke, Sargent, and Sims are {\it exact}, that is, involve mutually
orthogonal (all leads, all lags) idiosyncratic components; the common shocks then account for all cross-covariances, a most restrictive 
assumption that cannot be expected to hold in practice.  \citet{chamberlain:1983} and
\citet{chamberlain:rothschild:1983} are relaxing this exactness assumption into an
assumption of {\it mildly} cross-correlated idiosyncratics (the so-called
``approximate factor models'').\linebreak  
Finite-$N$ non-identifiability is the price to be paid for
that relaxation; the resulting models, however, remain asymptotically (as $N$ tends to
infinity) identified, which is perfectly in line with the spirit of high-dimensional
asymptotics. 

This idea of an approximate factor model in the analysis of high-dimensional time series has been picked up,  and popularized in the early 2000s, mostly by
\citet{stock:watson:2002a,stock:watson:2002b,bai:ng:2002,bai2003inferential}, and \citet{FHLR:2000}, followed by many others.  The most common  definition of ``mildly cross-correlated,'' hence of ``idiosyncratic,'' is   {\it ``a (second-order stationary) process $\{\xi_{it}\, \vert\, i\in\mathbb{N},\, t\in\mathbb{Z}\}$ is called {\em idiosyncratic} if the  variance of any sequence of normed (sum of squared coefficients equal to one) linear combination of the present, past, and future  values of~$\{\xi_{it}\}$, $i=1,\ldots,N$ is bounded as $N\to\infty$.''}
With this definition, \cite{hallin:lippi:2013} show that, under very general regularity assumptions,   the decomposition into common and idiosyncratic constitutes a representation result  rather than a statistical model---meaning that   (at population level) it exists and is fully identified---and coincides with the so-called {\it general(ized) dynamic factor model} of \cite{FHLR:2000}. 

{Let us give some intuition for common and idiosyncratic components for scalar data. In finance and economics, the common component is the co-movement of all the series, which are highly cross-sectionally correlated, and the idiosyncratic is thus literally the individual movement of each component that has smaller correlations with others. For financial data, low common-to-idiosyncratic variance ratios are typical for each series. However, not matter how small the values of these ratios, the contribution of the idiosyncratics, as $N \to \infty$, is eventually surpassed by that of the common component due to the pervasiveness of the factors (co-movements).}

Most existing factor models then can be obtained by imposing further restrictions on this representation result: assuming a finite-dimensional common space yields the  models proposed by  \citet{stock:watson:2002a,stock:watson:2002b} and \cite{bai:ng:2002}, additional sparse covariance assumptions characterize the approaches by \cite{fan2011,fan:2013:jrssb-discussion}, and  the dimension-reduction approach by  \cite{Lam2011}  and \cite{lam:yao:2012} is based on the assumption of  i.i.d.~idio\-syncratic components.
For further references on  dynamic factor models, we refer to the survey paper by \cite{BNg08} or   the monographs by \cite{SW11} and  \cite{Hallinetal2020}; for the state-of-the-art on general dynamic factor models, see \citet{FHLZ:2015,FHLZ:2017}.

Factor models for \emph{univariate} functional data have already been studied: these are based on models which consider functional factors with scalar loadings \citep{hays:shen:huang:2012,kokoszka:miao:zhang:2015,young:2016,bardsley2017change,kokoszka2017testing,horvath2020time}, scalar factors with functional loadings \citep{kokoszka2018dynamic}, functional factors with loadings that are operators \citep{tian:2020}, or scalar factors and loadings in a smoothing model \citep{gao2021factor}.

For \emph{multivariate} functional data (moderate fixed dimension), the only factor models   we are aware of were introduced by \citet{castellanos2015multivariate}, who consider a multivariate Gaussian   factor model with functional factors and scalar loadings, by \citet{kowal2017bayesian}, with scalar factors and functional loadings within a dynamic linear model, and by \citet{white2020multivariate}   with functional factors and scalar loadings within a multi-level model. 
There exists, also, a substantial literature on modelling multivariate functional data, see for instance \citet{happ_greven_multivariateFPCA,gorecki2018selected,wong2019partially,li2020fast,zhang2020dynamic,qiu2021two}. 
Except for  \cite{wong2019partially}, all these contributions, however, are limited to fixed-dimension vectors of   functional data defined on the same domain.

The increasing availability of \emph{high-dimensional FTS} data, requiring high-dimensional asymptotics under which   the dimension of the functional objects is   diverging with the sample size, recently has   generated much interest: \citet{zhou2020statistical} consider statistical inference for sums of high-dimensional FTS, \citet{fang2020new} explore  new perspectives on dependence for high-dimentional FTS, \citet{chang2020autocovariance} study a learning framework based on autocovariances, and \citet{hu2021sparse} investigate a sparse functional principal component analysis in high dimensions under Gaussian assumptions. Again, all  functional components involved, in these papers, are required to take values in the same functional space.

On the other hand, the \emph{factor model} approach in the statistical analysis of high-dimensional FTS remains largely unexplored.
\citet{gao:shang:2018} are using 
factor model techniques in the forecasting of high-dimensional FTS. They adopt a heuristic two-stage approach 
combining a truncated-PCA dimension reduction\footnote{conducted via an eigendecomposition of the long-run covariance operator, as opposed to the lag-$0$ covariance operator considered here.}  and separate scalar factor-model analyses of the resulting (scalar) panels of   scores. 

Our approach,   inspired by the time-series econometrics literature, is radically different. Indeed, we start (Section~\ref{sub:representation_theorem}) with a stochastic representation result  establishing the existence, under a functional version of the traditional   assumptions on dynamic panels, of a factor model representation of the observations with scalar  factors and functional loadings.   Theorems~\ref{thm:equiv_factormodel_eigenvalues}
and~\ref{thm:factormodel_uniqueness} are  relating the existence and the form of 
a high-dimensional functional factor model representation to the asymptotic behavior   of the eigenvalues of the panel covariance operator of the observations; they actually are the functional counterparts of    
classical results by \cite{chamberlain:rothschild:1983} and \cite{chamberlain:1983}.  Contrary to  \citet{gao:shang:2018}, our   approach, thus, is principled: the existence and the form of our factor model representation    are neither an assumption nor  the description of a presupposed data-generating process but follow from a mathematical result. It avoids the   loss of information related with Gao et~al.'s   truncation of the PCA decomposition and their separate analysis of interrelated panels of principal component scores. 
Its   scalar factors and   functional loadings, moreover, allow for  the analysis of panels comprising   time series of different types, functional and scalar.   

Drawing inspiration from
\citet{stock:watson:2002a},  \cite{stock:watson:2002b}, and \cite{bai:ng:2002},
we then develop an estimation theory
for the unobserved factors, loadings, and common components 
(Theorems~\ref{theorem:factors_average_bound},
\ref{theorem:factors_average_bound_up_to_sign}, \ref{theorem:loadings_average_bound}, and
\ref{theorem:common_component_average_bound}). This is  laying the theoretical
foundations for an analysis of high-dimensional FTS. Since our methods can deal with mixed  panels of scalar and functional
time series, our results encompass and extend, sometimes under weaker assumptions, the classical ones by 
\cite{chamberlain:rothschild:1983,bai:ng:2002,stock:watson:2002a,fan:2013:jrssb-discussion}. 

Our paper also derives a forecasting method for high-dimensional FTS. Through a panel of Japanese mortality curves, we show the superior accuracy of our method compared to alternative forecasting methods.

The paper is organized as follows.  
Section~\ref{sec:model_and_representation} is devoted to our representation result: roughly, a factor-model representation, with $r$ scalar factors and functional loadings, exists if and only if  the number of diverging (as $N\to\infty$) eigenvalues of the  covariance operator of the $N$-dimensional observation is~$r$. 
Section~\ref{sec:estimation} is dealing with estimation procedures: Section~\ref{sub:estimation_of_factors_loadings_and_common_component} introduces our 
estimators of the factors, loadings, and common component through a least squares
approach and Section~\ref{sub:inferring_the_number_of_factors} 
provides a family of criteria for  the identification of the number $r$ of
factors. Section~\ref{sec:theoretical_results}  is about the consistency properties of the procedures described in Section~\ref{sec:estimation}: consistency  (Section~\ref{sub:consistent_estimation_of_the_number_of_factors}) of the identification method of Section~\ref{sub:inferring_the_number_of_factors}   and consistency   rate results for the estimators described in Section~\ref{sub:estimation_of_factors_loadings_and_common_component}. The latter take the form of average (Section~\ref{sub:average_error_bounds}) and uniform (Section~~\ref{sub:uniform_error_bounds}) error bounds, respectively. In
Section~\ref{s:simu}, we conduct some numerical experiments, and provide the forecasting
application in
Section~\ref{sec:forecasting_mortality_curves}. We conclude in
Section~\ref{sec:discussion} with a discussion. Technical results and all the
proofs, as well as additional simulation results, are contained in the
\thesupplement.

\section{A representation Theorem}
\label{sec:model_and_representation}


\subsection{Notation}
\label{sub:notation_for_model}

Throughout, we denote by    
\begin{equation}
    \label{eq:calX_NT}
    {\cal X}_{N,T} \defeq \{x_{it},\  i=1,\ldots,N, \; t=1,\ldots,T\}\quad N\in\mathbb{N},\ T\in\mathbb{N}
\end{equation}
an observed $N\times T$ panel of time
series, where the random ele\-ments~$\{ x_{it} : t \geq 1 \}$ take values in a real separable
Hilbert space   $H_{i}$ equipped with the inner product $\sc{\cdot, \cdot}_{H_i}$ and norm~$\hnorm{\cdot}_{H_i} \defeq   \sc{\cdot, \cdot}_{H_i}^{1/2}$, $i=1,\ldots,N$.
These series can be of different types. A case of interest is the one for which some series are scalar~($H_i=\mathbb{R}$)
and some others are square-integrable functions from some closed real interval to $\mathbb R$, that is, $H_i=L^2([a_i, b_i], \bbR),$ where~$a_i, b_i \in \bbR$ with $a_i < b_i$.\footnote{Even though each interval $[a_i, b_i]$ could   be changed to $[0,1]$, we keep this notation to emphasize that curves for different values of~$i$ are fundamentally of different nature, and hence cannot,  in general,  be added up.} The typical inner-product on~$L^2([a_i, b_i], \bbR)$ is $\sc{f,g}_{H_i} \defeq \int_{a_i}^{b_i} f(\tau) g(\tau) d\tau$ for~$f,g \in L^2([a_i, b_i], \bbR)$.\footnote{Practitioners could also use other norms, such as  $\sc{f,g}_{H_i} \defeq \int_{a_i}^{b_i} f(\tau) g(\tau) w(\tau) d\tau$, where~$w(\tau) > 0$ for all~$\tau \in [a_i, b_i]$. Our setup can also handle series with values in multidimensional spaces, such as 2-dimensional or~3-dimensional images, since these can be viewed as functions in $L^2([0,1]^d, \bbR)$, $d \geq 1$.}
While the results could be derived for these special cases only, the generality of our approach allows the practitioner to choose the Hilbert spaces $H_1,\ldots, H_N$ to suit their specific application\footnote{For instance, a different inner product could be chosen for the spaces $L^2([a_i, b_i], \bbR)$, or $H_i$ could be chosen to be a reproducing kernel Hilbert space.}.

We tacitly assume that all $x_{it}$'s are random elements
with mean zero and finite second-order moments defined on some common
probability space $(\Omega,\mathcal{F},\pp)$; we also assume that~${\cal
X}_{N,T}$ constitutes the finite realization of some 
second-order $t$-stationary\footnote{In the case $H_i = L^{2}([a_i, b_i], \bbR)$, this means that each $x_{it}$ is a random function $\tau \mapsto x_{it}(\tau), \tau \in [a_i, b_i]$, that $\ee x_{it}(\tau)$ does not depend on $t$,  and that $\eee{x_{it}(\tau_1) x_{i't'}(\tau_2)}$ depends on $t,t'$ only through $t - t'$ for\linebreak  all~$i, i' \geq 1$,   $\tau_1 \in [a_i, b_i]$, $\tau_2 \in [a_{i'}, b_{i'}]$.} double-indexed pro\-cess~${\cal X}\defeq\{x_{it},\  i\in\mathbb{N}, \; t\in\mathbb{Z}\}$. 
{The mean zero assumption is not necessary, but simplifies the exposition---See Section~\ref{sec:results_no_mean_zero} of the \thesupplement\ for more general, non zero-mean results. In practice, each observed series $\{ x_{it}\}$ is centered at its empirical mean at the start of the statistical analysis.}

Define $\bH_N \defeq H_{1} {\scriptstyle\bigoplus} H_{2}
{\scriptstyle\bigoplus} \cdots {\scriptstyle\bigoplus} H_{N}$ as the
(orthogonal) direct sum of the Hilbert \linebreak spaces~$H_1,\ldots, H_N$. The elements of
$\bH_N$ are of the form $ \b v \defeq ( v_1, v_2, \ldots , v_N)^\tp $ \linebreak  or~$ \b w
\defeq (w_1, w_2 , \ldots , w_N)^\tp$, where $v_i, w_i \in H_i, i=1, \ldots, N$, and $\cdot ^\tp$ denotes transposition. 
The $i$-th component of an element of $\bH_N$ is denoted by~$(\cdot)_i$; for instance,~$(\bv)_i$ stands for~$v_i$. 
The space $\bH_{N}$,  naturally
equipped with the inner product 
\[ 
    \sc{ \b v, \b w}_{\bH_N} \defeq \sum_{i=1}^{N} \sc { v_{i}, w_{i} }_{H_i},
\]
is a real separable Hilbert space.  When no confusion is
possible, we write
$\sc{\cdot, \cdot}$ for~$\sc{\cdot, \cdot}_{\bH_N}$
 and  $\hnorm{\cdot} \defeq  { \sc{\cdot, \cdot}}^{1/2}$ for the
resulting norm.  Let $\bounded{H_1, H_2}$ denote the space of bounded (linear)
operators from $H_1$ to $H_2$, and use the shorthand notation~$\bounded{H}$ for
$\bounded{H, H}$. Denote the operator norm of $V \in \bounded{H_1, H_2}$ by
\[ 
    \opnorm{V} \defeq \sup_{x \in H_1, x \neq 0} \hnorm{Vx}_{H_2}/\hnorm{x}_{H_1},
\]
and write $V^\adjoint$ for the adjoint of $V$, which satisfies $\sc{V u_1, u_2}_{H_2} = \sc{u_1,
V^\adjoint u_2}_{H_1}$ for all~$u_1 \in~\!H_1$ and~$u_2 \in H_2$. In particular, we have 
\[
    \opnorm{V} = \opnorm{V^\adjoint} = \opnorm{V^\adjoint V}^{1/2},
\]
see \citet{hsing:eubank:2015}.
We denote by $\Fnorm{V}$ the {Hilbert--Schmidt norm} of an operator~$V \in \compact{H_1, H_2}$ 
defined by $\Fnorm{V} \defeq \big[\sum_{j \geq 1} \hnorm{Ve_i}_{H_2}^2\big]^{1/2}$ where $\{e_i\} \subset H_1$ is an orthonormal basis of $H_1$, the sum might be finite or infinite, and is independent of the choice of the basis. If $V \in \bbR^{p \times q}$, then the Hilbert--Schmidt norm is often called the Frobenius norm.
For $u_1 \in H_1, u_2 \in H_2$, define $u_1 \tensor u_2 \in \bounded{H_2, H_1}$ by $(u_1 \tensor u_2)(v_2) = \sc{u_2,v_2}_{H_2} u_1$ for all~$v_2 \in H_2$. If $\alpha \in \bbR, u \in H$, we define $u \alpha \defeq \alpha u$.
We shall denote by $\bbR^{p \times q}$ the space of $p \times q$ real matrices, and will identify this space with the space of mappings $\bounded{\bbR^q, \bbR^p}$. Notice that for $\b A \in \bbR^{p \times q}$, $\b A^\tp = \b A^\adjoint$. We write $\bbN$ for the set of positive integers, that is,~$\bbN \defeq \{1,2, \ldots \}$, and denote the Euclidean norm in $\bbR^p$ by $\rpnorm{\cdot}$.

When working with the panel data,  let~$\bx_{t}\defeq (x_{1t}, x_{2t},\ldots, x_{Nt})^{\tp}~\!\in~\!\b H_{N}$ where, in order to keep the presentation simple, the dependence   on $N$ 
does not explicitly appear.  
We denote by $\lambda_{N,1}^{x}, \lambda_{N,2}^{x},\ldots $ the eigenvalues,  in decreasing order of magnitude,  of~$\bx_t$'s
covariance operator $\eee{ (\bx_t - \ee \bx_t) \tensor (\bx_t - \ee \bx_t)} \in \bounded{\bH_N}$,
which maps $\by \in \bH_N$ to
\[
    \eee{\sc{\bx_t - \ee \bx_t, \by} (\bx_t - \ee \bx_t)} \in \bH_N,
\]
\citep[][Definition~7.2.3]{hsing:eubank:2015}.
In view of stationarity,  these eigenvalues do not
depend on $t$.
 Unless otherwise mentioned, convergence of sequences of random variables
is in  mean square. 

\subsection{A functional factor model}


The basic idea in all factor-model approaches to the analysis of high-dimensional
time series  consists in decomposing the observation $x_{it}$ into the sum~$\chi_{it}
+ \xi_{it}$ of two unobservable and mutually orthogonal components, the  {\it common}
component $\chi_{it}$ and the {\it idiosyncratic} one $\xi_{it}$. The various factor models
that are found  in the literature only differ in the way~$\chi_{it}$ and~$\xi_{it}$
are characterized. The characterization we are adopting here is inspired from
\citet{forni:lippi:2001}.

\begin{Definition}
    \label{def:static-factor-model} The functional zero-mean second-order stationary stochastic process
    \[
        {\cal X} \defeq \{x_{it}:\,  i\in~\!\mathbb{N}, \; t\in~\!\mathbb{Z}\}
    \]
    admits a   (high-dimensional)    {\em functional factor model representation with $r \in \bbN$
    factors} 
if there exist {\em loadings}  
    $b_{il} \in H_i$, $i\in\bbN$, $l = 1,\ldots, r$,  $H_i$-valued
    processes $\{\xi_{it}:\ t\in\bbZ \}$, $i \in \bbN$,   and a real $r$-dimensional second-order stationary  process~$\{\bu_{t}=(u_{1t},
    \ldots, u_{rt} )^\tp :\ t\in\bbZ \}$, co-stationary with $\mathcal X$, such that
    \begin{equation} \label{eq:r-factor-model} 
        x_{it} = \chi_{it} + \xi_{it} = b_{i1} u_{1t} + b_{i2} u_{2t} + \cdots + b_{ir} u_{rt} + \xi_{it}, \qquad  i \in\bbN, \:\: t \in \bbZ
    \end{equation}
    ($\chi_{it} $ and $\xi_{it}$ unobservable) holds with
    \begin{enumerate} 
        \item[(i)]  $\ee{\bu}_t={\b 0}$ and $\eee{{\bu}_t{\bu}^\tp_t\,}= \boldsymbol{\Sigma}_{\boldsymbol{u}} \in \bbR^{r \times r}$ is positive definite;
        \item[(ii)] 
            \label{it:property-uncorrelated}
            $\eee{u_{lt}\xi_{it}} =  0$ and $\eee{\xi_{it} }=0$  for all $t\in\bbZ$, $l=1,\ldots,r$, and $i\in\bbN$; 
        \item[(iii)] denoting by $\lambda^{\xi}_{N,l}$ the $l$-th  (in decreasing order of magnitude) eigenvalue of the covariance operator $\eee{ \bxi_t \tensor \bxi_t}$   of $\b \xi_{t} \! \defeq  (\xi_{1t}, \ldots, \xi_{Nt})^\tp$, 
           $\lambda^{\xi}_{1}\!  \defeq  \sup_N \lambda^\xi_{N,1}~\!\!<~\!\infty$; 
    \item[(iv)] denoting by $\lambda^{\chi}_{N,l}$ the $l$-th  (in decreasing order of magnitude) eigenvalue of the covariance operator $\eee{ \bchi_t \tensor \bchi_t}$   of $\b \chi_{t} \! \defeq  (\chi_{1t}, \ldots, \chi_{Nt})^\tp\!\!$,    $\lambda^{\chi}_{r}  \defeq \sup_N \lambda^\chi_{N,r}\!~\!=~\!\infty$.
    \end{enumerate}
    The~$r$ scalar random variables~$u_{lt}$ are called {\em factors}; the $b_{il}$'s are the  {\em
        (functional) factor loadings};~$\chi_{it}$ and $\xi_{it}$ are called $x_{it}$'s   \emph{common} and  \emph{idiosyncratic component}, respectively. 
\end{Definition}
  
Notice that the number $r \in \bbN$ is crucial in this definition, which calls for some remarks and comments. 
\begin{enumerate}
    \item[(a)] In the terminology of \citet{hallin:lippi:2013},
         equation~\eqref{eq:r-factor-model}, where the factors are loaded contemporaneously,
        is  a {\em static functional factor representation}, as opposed to the {\em
        general dynamic factor representation}, where the~$b_{ij}$'s   are
        linear one-sided square-summable  filters of the form $b_{ij}(L)= \sum_{k=0}^\infty b_{ijk}L^k$ ($L$ the lag operator), and the~$u_{jt}$'s are  mutually orthogonal  second-order white noises (the {\em
        common shocks}) satisfying~$\eee{u_{lt} \xi_{is}} = 0$ for all~$t,s \in \bbZ$ and $l=1,\ldots, r$. The static $r$-factor
        model where the~$r$ factors are driven by~$q\leq r$  shocks  is   a particular case of the  general dynamic one    with~$q\leq r$ common shocks. When the idiosyncratic processes~$\{\xi_{it}:\ t~\!\in~\!\bbZ\}$ themselves are mutually orthogonal at all leads
        and lags, static and general dynamic factor models are called {\em exact};
        with this assumption relaxed into (iv) above, they sometimes are called
        {\em approximate}. In the sequel, what we call {\em factor models} all are
        {\em  approximate static factor models}. 
    \item[(b)] The functional factor representation  \eqref{eq:r-factor-model} also can be written, with obvious notation~${\b\chi}_{t}$ and~${\b\xi}_{t}$, in vector form
        \[
            \bx_{t} = \b \chi_{t} + \b \xi_{t} = \b B_N \bu_t + \b\xi_{t},
        \]
        where $\bB_N \in \bounded{\bbR^r, \bH_N}$ is the linear operator mapping the $l$-th canonical basis vector of~$\bbR^r$ to $(b_{1l}, b_{2l}, \ldots, b_{Nl})^\tp \in \bH_N$.
    \item[(c)] Condition (iii) essentially requires that  cross-correlations among
        the components of the process~$\{\b \xi_{t}:\ t\in\bbZ \}$ are not pervasive as
        $N\to\infty$; this is the characterization of idiosyncrasy. A sufficient assumption 
         for condition (iii) to hold is
        \[
            \sum_{j = 1}^\infty \opnorm{ \eee{\xi_{it} \tensor \xi_{jt}} } < M < \infty, \quad
            \forall i=1, 2, \ldots
        \]
        and $\ee \hnorm{\bxi_t}^2 < \infty$ for all $N \geq 1$, 
        see Lemma~\ref{lma:bai_Ng_rth_largest_eigenvalue_of_covariance_diverges} in
        the \thesupplement.
    \item[(d)] Condition  (iv) requires pervasiveness,  as $N\to\infty$, of
        (instantaneous) correlations among the components of~$\{\b \chi_{t}:\
        t\in\bbZ \}$; this is the characterization of commonness---the  factor
        loadings~$\b B_N$ should be such that factors (common shocks) are loaded again and
        again  as $N\to\infty$. 
        A sufficient condition for this is that, as $T\to\infty$,  
        \[
            \sum_{i=1}^T \bu_t \bu_t^\tp/T \to \b \Sigma_\bu  \quad \text{and} \quad \bB_N^\adjoint \bB_N/N = \left(N^{-1} \sum_{i=1}^N \sc{b_{il}, b_{ik}}_{H_i} \right)_{lk}\!\! \!\rightarrow~\!\b \Sigma_{\b B},
        \]
        where $\b \Sigma_\bu \in \bbR^{r \times r}$  and $\b \Sigma_{\bB} \in \bbR^{r \times r}$ are  positive definite, and where we have identified the ope\-rator~$\bB_N^\adjoint \bB_N/N \in \bounded{\bbR^r}$ with the corresponding $r \times r$ real matrix. 
    \item[(e)]  It follows from Lemma~\ref{lma:eigenvalue_inequalities}  in
        the \thesupplement\ that if $\cal X$ has a
         functional factor representation with $r$ factors, then  
        $\lambda^{x}_{r+1} \defeq \sup_N \lambda^x_{r+1} < \infty$. This in turn implies
        that the number~$r$  of factors is uniquely defined.
    \item[(f)] The factor loadings and the factors are only jointly identifiable 
        since, for any 
        invertible matrix $\b Q \in \bbR^{r \times r}$,
      $ \b B_N \bu_t = (\b B_N \b Q^{-1})(\b Q \bu_t),$
        so that $\b v_t = \b Q \bu_t$ provides  the same decomposition of $\cal X$ into common plus idiosyncratic  as~\eqref{eq:r-factor-model}.
    \item[(g)] It is often assumed that $\{\bu_{t}: \ t\in\mathbb{Z}\}$ is an 
        $r$-dimensional autoregressive (VAR) process driven by~$q\leq r$ white noises
        \citep{Amengual2007}, but this is not required here. 
\end{enumerate}

\subsection{Representation Theorem}%
\label{sub:representation_theorem}

The following results shows that, under  adequate condiions on its covariance operator, a process  $\cal X$ admits a  functional   factor model representation  (in the sense of Definition~\ref {def:static-factor-model})  with $r \in \bbN$ factors. These conditions are  functional versions of the conditions considered in the scalar case, and involve the eigenvalues~$\lambda_{N,j}^x$ of the covariance operator of the observations~$\bx_{t}$---while Definition~\ref {def:static-factor-model} involves the eigenvalues $\lambda_{N,j}^\chi$ and $\lambda_{N,j}^\xi$ of the covariance operators of   unobserved common and idiosyncratic components.  Moreover, when $\cal X$ admits a functional   factor model representation of the form \eqref{eq:r-factor-model}, its  decomposition   into a common and an 
idiosyncratic component is unique.

Let $\lambda_{j}^{x} \defeq \lim_{N \to \infty} \lambda_{N,j}^{x}= \sup_{N} \lambda_{N, j}^{x}$. This limit (which is possibly infinite) exists, as~$\lambda_{N,j}^{x}$ is monotone increasing with $N$. 
\begin{Theorem}
    \label{thm:equiv_factormodel_eigenvalues} 
    The process $\cal X$ admits a (high-dimensional) functional~factor model
    representation with $r$ factors if and only if
    $\lambda_{r}^x = \infty$ and  $\lambda_{r+1}^x < \infty$.
\end{Theorem}

The following result tells us that the common component $\chi_{it}$ is asymptotically
identifiable, and provides its expression in terms of an $\LLP$ projection.
\begin{Theorem}
    \label{thm:factormodel_uniqueness}
    Let $\cal X$ admit  (in the sense of Definition~\ref {def:static-factor-model}) the   functional  factor model representa\-tion~$
    x_{it} = \chi_{it} + \xi_{it}$, $i \in \bbN$, $t \in \bbZ,$ with $r$ factors.
    Then (see the Appendix, Section~\ref{sec:orthogonal_projections} for a formal definition of $ \proj_{H_i}$),
    \[
        \chi_{it} = \proj_{H_i}(x_{it} | \mathcal D_t), \quad \forall i \in \bbN, t \in \bbZ,
    \]
    where 
    \[
        \mathcal D_t  \defeq  \left\{ \mathfrak{p} \in \LLP | \: \mathfrak{p} = \lim_{N \rightarrow \infty}
            \sc{\ba_N, \bx_{t} }_{\bH_N},   \ba_N \in \b H_N, \hnorm{\ba_N}_{\bH_N}
        \stackrel{N \rightarrow \infty}{\longrightarrow}  0 \right\} \subset \LLP
    \]
    (here ``$\lim_{N \to \infty}$'' is the limit in quadratic mean).
    The common and the idiosyncratic parts of the factor model representation thus are unique, and asymptotically identified. 
\end{Theorem}
To clarify the definition of $\mathcal D_t$, let us give an example of $\ba_N \in \bH_N$ such $\hnorm{\ba_N}_{\bH_N} \to~\!0$. 
Choosing $v_i \in H_i$,   $i=1,\ldots, N$,  such that $\hnorm{v_i}_{H_i} = 1$, set~$\ba_N = (v_1/{N}, \ldots, v_N/{N}) \in \bH_N$, and notice that~$\hnorm{\ba_N}_{\bH_N} \to 0$ as $N \to \infty$. Now, assuming that \eqref{eq:r-factor-model} holds, 
$$\sc{\bx_t, \ba_N}_{\bH_N} = \sc{\bB_N \bu_t, \ba_N}_{\bH_N} + \sc{\bxi_t, \ba_N}_{\bH_N},$$
 and straightforward calculations show that $\eee{\sc{\bxi_t, \ba_N}_{\bH_N}^2} \leq \lambda^\xi_{N,1} \hnorm{\ba_N}_{\bH_N}^2 \to 0$ as $N \to~\!\infty$. Therefore, 
\[
    \lim_{N \to \infty} \sc{\bx_t, \ba_N}_{\bH_N} = \lim_{N \to \infty} \sc{\bB_N \bu_t, \ba_N}_{\bH_N} = \lim_{N \to \infty} \sc{\bu_t, \bB_N^\adjoint \ba_N}_{\bH_N},
\]
provided the limits exist. This shows that the condition $\hnorm{\ba_N}_{\bH_N} \to 0$ is sufficient for the idiosyncratic component to vanish asymptotically from $\sc{\bx_t, \ba_N}_{\bH_N}$.

The proofs of Theorems~\ref{thm:equiv_factormodel_eigenvalues}
and~\ref{thm:factormodel_uniqueness} are provided in the \thesupplement, Section~\ref{s:proof_rep}; they are inspired from \citet{forni:lippi:2001}---see also \citet{chamberlain:1983} and \citet{chamberlain:rothschild:1983}. 
Notice, however, that, unlike these references, our results do not require the minimal eigenvalue of
the covariance of~$\bx_t$ to be bounded from below---an unreasonable assumption in the functional setting.

The logical importance of Theorems~\ref{thm:equiv_factormodel_eigenvalues}
and~\ref{thm:factormodel_uniqueness}, which   establish the status of~\eqref{eq:r-factor-model} as a canonical\footnote{ Actually, canonical up to an identification constraint disentangling $b_{ik}$ and $u_{kt}$, $k=1,\ldots,r$ since only the product $b_{i1}u_{1t} + \cdots + b_{ir} u_{rt}$ is identified. That identification constraint, which boils down to the (arbitrary) choice of a basis in the space of factors, is quite irrelevant, though, in this discussion. } stochastic representation of ${\cal X}$, should not be underestimated.   That representation indeed needs not be a preassumed data-generating process for $\cal X$. The latter could take various forms, involving, for instance, a scalar matrix loading of functional shocks, yielding an alternative representation (2.1$^\prime$), say,  of ${\cal X}$. Provided that the second-order structure of $\cal X$ satisfies the assumptions of Theorems~\ref{thm:equiv_factormodel_eigenvalues}
and~\ref{thm:factormodel_uniqueness}, the two representations \eqref{eq:r-factor-model} and  (2.1$^\prime$) then may coexist; being simpler, \eqref{eq:r-factor-model} then is by far preferable for statistical inference, although~(2.1$^\prime$) might be more advantageous from the point of view of practical interpretation. 

It should be stressed, however, that a representation as (2.1$^\prime$), which involves functional factors, requires that all Hilbert spaces $H_i$ coincide, thus precluding the analysis of mixed panels of functional and scalar series.

\section{Estimation}
\label{sec:estimation}

Assuming that a functional factor model with $r \in \bbN$ factors holds for~$\mathcal
X$, we shall estimate the factors $\bu_t$ and loading operator $\bB_N$ via a least squares criterion. Section~\ref{sub:estimation_of_factors_loadings_and_common_component} describes the estimation method; see Sections~\ref {sub:average_error_bounds} and \ref{sub:uniform_error_bounds} for consistency results. 


\subsection{Estimation of Factors, Loadings and Common Component}%
\label{sub:estimation_of_factors_loadings_and_common_component}

Least squares criteria are often used for estimating factors
\citep{bai:ng:2002,fan:2013:jrssb-discussion}, but other methods are available
as well \citep{Forni1998,FHLR:2000}, see the discussion in Section~\ref{sub:proof_computing_proposition} of the \thesupplement.

Since the actual number~$r$ of factors is unknown, we shall devise estimators under the assumption that a  functional factor representation of the form \eqref{eq:r-factor-model} holds with unspecified number $r \in \{1,2,\ldots \}$ of factors; Section~\ref{sub:inferring_the_number_of_factors} proposes  a class of criteria for estimating the actual value of~$r$. 

For given $k$, 
our estimators $\bB_N^{(k)}$ of the factor loadings  and $\bU_T^{(k)} = ( \bu_{1}^{(k)},\ldots,  \bu_{T}^{(k)}) $ of  the factor values  are the minimizers, over $\bounded{\bbR^k,
\bH_N}$ and $\bbR^{k \times T}$, respectively, of 
\begin{equation}
    \label{eq:least_squares_fit_criterion}    
   P\left( \bB_N^{(k)},  \bU_T^{(k)}\right) \defeq \sum_t \hnorm{\bx_t -  \bB_N^{(k)}  \bu_{t}^{(k)}}^2.
\end{equation}
The following result gives a method for solving this minimization problem. The proof is provided  in Section~\ref{sub:proof_computing_proposition} of the \thesupplement, where the reader can also find the algorithm in the special case $H_i = \LLR$ for all $i \geq 1$.
\begin{Proposition}
    \label{prop:computing_loadings_and_factors}
    The minimum in \eqref{eq:least_squares_fit_criterion} is obtained by setting $\bB_N^{(k)} \defeq \tildeB{k}_N$ \linebreak
     and~$
    \bU_T^{(k)} \defeq \tildebU{k}_T$, where $\tildeB{k}_N$ and $\tildebU{k}_T$ are computed as follows:
    \begin{enumerate}
        \item compute $\bF \in \bbR^{T \times T}$, defined by $(\bF)_{st} \defeq \sc{\bx_s, \bx_t}/N$;
        \item compute the leading $k$ eigenvalue/eigenvector pairs $(\tilde \lambda_l, \tilde {\b f_l})$ of $\bF$, and set $$\hat \lambda_l \defeq T^{-1/2} \tilde \lambda_l \in \bbR,\ \  \hat {\b f}_l \defeq T^{1/2} \tilde {\b f}_l / \rpnorm{\tilde{\b f_l}} \in \bbR^T;$$
        \item compute $\hat \be_l \defeq \hat \lambda_l^{-1/2} T^{-1} \sum_{t=1}^T (\hat {\b f}_{l})_t \bx_t \in \bH_N$;
        \item set $\tildebU{k}_T \defeq (\hat {\b f}_1, \ldots, \hat {\b f}_k)^\tp \in \bbR^{k \times T}$ and define $\tildeB{k}_N$ as the operator in $\bounded{\bbR^k, \bH_N}$ mapping the $l$-th canonical basis vector of $\bbR^k$ to ${\hat \lambda_l}^{1/2} \hat \be_l$, $l = 1,\ldots, k$.
    \end{enumerate}
\end{Proposition}
The scaling in 2.\ is such that the $\hat \lambda_i$s are $O_{\rm
P}(1)$---see Lemma~\ref{lma:largest_sample_eigenvalue_bounded}  in the
\thesupplement. The estimator of the common component $\chi_{it}$ is then the $i$th component  
\begin{equation}
    \label{eq:hat_chi_it}
    \hat \chi_{it}^{(k)} \defeq \left(\tildeB{k}_N \tildebu{k}_t\right)_i \in H_i,
\end{equation}
of the vector $\tildeB{k}_N \tildebu{k}_t \in \bH_N$. 
Our approach requires computing inner products such as~$\sc{\bx_s, \bx_t} = \sum_{i = 1}^N \sc{x_{is}, x_{it}}_{H_i}$: any preferred method for computing $\sc{x_{is},\  x_{it}}_{H_i}$ can be used, see the discussion at the end of Section~\ref{sub:proof_computing_proposition} of the \thesupplement.  

{In the scalar case ($H_i = \bbR$), our method corresponds to performing a PCA of the data~$\bX_{NT} \in \bbR^{N \times T}$ by first computing the PC scores via an eigendecomposition\linebreak of~$\bX_{NT}^\tp \bX_{NT} \in \bbR^{T \times T}$.
Our approach is different from most existing multivariate functional principal component analysis \citep[MFPCA; e.g.][]{ramsay:silverman:2005,berrendero2011principal,chiou2014multivariate,jacques2014model} in that it accommodates the ``multivariate'' parts to be in different Hilbert spaces.
Our method shares similarities with the MFPCA for functional data with values in different domains of \citet{happ_greven_multivariateFPCA}, but is nevertheless 
different.} Indeed, we do not rely on intermediate \KL projections for
each $\{ x_{it} : t \in \bbZ\}$. Using a \KL (or FPCA) projection for each
$x_{it}$ separately, prior to conducting the global PCA, is
\emph{not} a good idea in our setting, as there is no guarantee that the common component
will be picked by the individual \KL projections which, actually, might well  remove all   common components (see Section~\ref{s:simu} below and
Section~\ref{sec:advantage_current_approach} in the \thesupplement\ for
examples). 
  The connection between our MFPCA and our functional factor models is similar to the one between PCA and multivariate factor models in the high-dimensional context: (1) asymptotically, (MF)PCA recovers the common component, and (2) in finite samples, (MF)PCA can be used to estimate factor and loadings, up to an invertible transformation.

\subsection{Identifying the number of factors}%
\label{sub:inferring_the_number_of_factors}

We have given in
Section~\ref{sub:estimation_of_factors_loadings_and_common_component}
estimators of the factors, loadings and common component if  
 a functional factor model with given number~$k$ of factors holds.  
In practice, that number is unknown and identifying it is a challenging problem. 
In the scalar context, a  rich literature  exists on the subject, based on  information criteria \citep{bai:ng:2002, Amengual2007, hallin:liska:2007, bai2007determining, li2017identifying, WM17}, testing procedures \citep{pan2008modelling, onatski:2009:numberOfDynamicFactors, Kapetanios2010, lam:yao:2012, AHNH13},   or various other strategies \citep{Onatski2010,  lam:yao:2012, Lippietal2016, fan2019estimating}, to quote only a   few. 

All these references, of course, are limited to  scalar factor models. Extending them all to the functional context is beyond the scope of this paper, which is focused on the estimation of loadings and factors. Therefore, we are concentrating on the extension of the classical \cite{bai:ng:2002} information theoretic criterion. When combined with the tuning device proposed in \cite{Alessi2010}\footnote{For clarity of exposition, we postpone to Section~\ref{sec:simulation_number_of_factors}  a rapid description of  that tuning method, which goes back to \cite{hallin:liska:2007}.    Its consistency  straightforwardly follows from the consistency of the corresponding  ``un-tuned''   method, i.e. holds under the results of Section~\ref{sub:consistent_estimation_of_the_number_of_factors}.}, that criterion has proved quite efficient in the scalar case. 

In this section, we   describe a general class of information
criteria extending the  \cite{bai:ng:2002} method and providing a consistent estimation of the actual number of factors in the functional context, hence also in the classical scalar context; in that particular case, our consistency results hold under    assumptions that are weaker than the existing ones \citep[e.g., the assumptions in][]{bai:ng:2002}---see
Remark~\ref{rmk:consistency_theorem}.

The information criteria that we will use for identifying the number of factors are 
of the form
\begin{equation}
    \label{eq:IC_defn_theoretical}
    \IC(k) \defeq V(k, \tildebU{k}_T) + k\, g(N,T),
\end{equation}
which is the sum of a goodness-of-fit measure
\begin{equation}
    \label{eq:defn_V}
    V(k,\tildebU{k}_T) \defeq \min_{{\bB_N^{(k)} \in \bounded{\bbR^k, \bH_N}}} \frac{1}{NT} \sum_{t=1}^{T} \hnorm{ \bx_{t}
    - {\bB_N^{(k)}} \tildebu{k}_{t} }^{2}
\end{equation}
(recall that $\tildebu{k}_t \in \bbR^k$ is the $t$-th column of $\tildebU{k}_T$ defined in Proposition~\ref{prop:computing_loadings_and_factors}) and   a penalty term $k \, g(N,T)$.   
For a given set  $\{ x_{it} : i = 1,\ldots, N; t=1,\ldots, T\}$ of observations, we  estimate the number of factors by
\begin{equation}
    \label{eq:hat_r_theoretical}
    \hat r \defeq \argmin_{k=1,\ldots, k_{\mathrm{max}}} \IC(k),
\end{equation}
where $k_\textrm{max} < \infty$ is some pre-specified maximal number of factors.  Section~\ref{sub:consistent_estimation_of_the_number_of_factors} provides some consistency results for $\hat r $.

\section{Theoretical Results}%
\label{sec:theoretical_results}

Consistency results for $\hat r $  and the estimators described in Section~\ref{sub:estimation_of_factors_loadings_and_common_component} require  regularity assumptions, which   are functional versions of standard assumptions in scalar factor models \citep[see, e.g.,][]{bai:ng:2002}.
\begin{customass}{A} \label{assumption:a}
    There exists $\{ \bu_t \} \subset \bbR^r$ and $ \{\bxi_{t}\} \subset \bH_N$  that are mean zero second-order co-stationary, with $\eee{ u_{lt} \xi_{it}} = 0$ for all $l=1,\ldots, r$ and   $i \geq 1$, such that~\eqref{eq:r-factor-model} holds. The covariance matrix  $\b\Sigma_{\bu} \defeq \eee{\bu_t \bu_t^\tp} \in \bbR^{r \times r}$ is positive definite and $T^{-1} \sum_{t=1}^T \bu_t \bu_t^\tp \convP \b \Sigma_{\bu}$ as $T \rightarrow \infty$.
\end{customass}
\begin{customass}{B} \label{assumption:b}
    $ N^{-1} \sum_{i=1}^N \sc{b_{il}, b_{ik}}_{H_i}  \to \left( \b \Sigma_{\bB}\right)_{lk}$ as $N \to \infty$ for some $r\times r$ positive definite matrix $\b \Sigma_{\bB} \in \bbR^{r \times r}$ and all $1\leq l,\, k\leq r$.
\end{customass}

\begin{customass}{C} \label{assumption:c}
    Let $\nu_{N}(h) \defeq N^{-1}\eee{ \sc{\bxi_t, \bxi_{t-h} } }$. There exists a constant $M < \infty$ such that
    \begin{description}[labelindent=1cm]
        \item[(C1)]  $  \sum_{h \in \bbZ}  \abs{\nu_N(h)} \leq M$ for all $N \geq 1$, and
        \item[(C2)] $\ee \abs{ \sqrt{N}\left(  N^{-1}  \sc{\bxi_{t}, \bxi_{s}}   - \nu_N(t-s) \right) }^2 < M$ for all $N, t,s \geq 1$.
    \end{description}
\end{customass}

\begin{customass}{D} \label{assumption:Bn_xi}
There exists a constant $M < \infty$ such that  
        $\hnorm{b_{il}} < M$   for all~$ i\in~\!\mathbb{N}$ and~$ l=1,\ldots, r$, and  
       $
            \sum_{j = 1}^\infty \opnorm{ \eee{\xi_{it} \tensor \xi_{jt}} } < M$ for all  
             $i\in\mathbb{N}$.
\end{customass}

Assumption~\ref{assumption:a} has some basic requirements about the model (factors and
idiosyncratics are co-stationary and  uncorrelated at lag zero), and  
 the~factors. It assumes, in particular, that~$\Fnorm{\bU_T} = O_{\rm P}(\sqrt{T})$. 
It also implies a weak
law of large numbers for $\{ \bu_t \bu_t^\tp \}$, which holds under various dependence
assumptions on $\{\bu_t \}$, see e.g.\ \citet{brillinger:2001,Bradley2005,dedecker2007weak}.
{The mean zero assumption is not necessary here, but   simplifies the exposition of the results, see Section~\ref{sec:results_no_mean_zero} in the \thesupplement.}

Assumption~\ref{assumption:b} deals with the factor loadings, 
and implies, in particular,\linebreak that~$\sum_{i=1}^N \sum_{l=1}^r \hnorm{b_{il}}^2$ is of order $N$.
Intuitively, it means (pervasiveness)  that the factors are loaded again and again as the cross-section increases.

Assumptions~\ref{assumption:a} and \ref{assumption:b} together imply that the first $r$ eigenvalues of the covariance operator of $\bchi_t$ diverge at rate $N$. 
They could be weakened
by assuming that the $r$ largest eigenvalues of the $r \times r$ matrix
$\left(N^{-1} \sum_{i=1}^N \sc{b_{il}, b_{ik}}_{H_i} \right)_{l,k=1, \ldots, r}$
and $\bU_T \bU_T^\tp/T$ are bounded away from infinity and zero, see e.g.\
\citet{fan:2013:jrssb-discussion}.

Assumption~\ref{assumption:c} is an assumption on the idiosyncratic terms: (C1) 
limits the total variance and  lagged 
cross-covariances of the idiosyncratic component;  
(C2) implies a uniform rate
of convergence in the law of large numbers for~$\{ \sc{ \bxi_{t}, \bxi_s}/N : N \geq 1\}$.
In particular, (C2) implicitly limits  the cross-sectional and lagged 
correlations of the  idiosyncratic components, and is sharply weaker than the corresponding assumption of \citet{bai:ng:2002}, see Remark~\ref{rmk:consistency_theorem} below.

Assumption~\ref{assumption:Bn_xi} limits the cross-sectional correlation of the
idiosyncratic components, and bounds the norm of the loadings. It implies that 
$\sum_{t=1}^T \rpnorm{\bB_N^\adjoint \bxi_t}^2$
is $ O_{\rm P}\left({NT}\right)$ as $N,T \to~\infty$ (see
Lemma~\ref{lma:frobenius_Norm_of_Bn_xi}  in the \thesupplement) and could be
replaced by this weaker condition in the proofs of
Theorems~\ref{theorem:factors_average_bound},
\ref{theorem:factors_average_bound_up_to_sign}, \ref{theorem:loadings_average_bound} and
\ref{theorem:common_component_average_bound}. 

Note that Assumptions~\ref{assumption:a}, \ref{assumption:b}, \ref{assumption:c}, and
\ref{assumption:Bn_xi} imply that the largest $r$ eigenvalues of~$\cov(\bx_t)$
diverge while the $(r+1)$th one remains bounded
(Lemma~\ref{lma:bai_Ng_rth_largest_eigenvalue_of_covariance_diverges} in the
\thesupplement), hence the common and idiosyncratic components are
asymptotically identified (Theorem~\ref{thm:equiv_factormodel_eigenvalues}).

\subsection{Consistent Estimation of the Number of Factors}%
\label{sub:consistent_estimation_of_the_number_of_factors}

The following result gives sufficient conditions on the penalty term $g(N,T)$ in \eqref{eq:IC_defn_theoretical} for 
 the  estimated number of factors~\eqref{eq:hat_r_theoretical} to be consistent as both $N$ and $T$ go to infinity.

\begin{Theorem}
    \label{thm:consistency_number_of_factors} Let $C_{N,T} \defeq \min\{\sqrt{N}, \sqrt T\} $ and consider $\hat r$ as defined in~\eqref{eq:hat_r_theoretical}. 
    Under Assumptions~\ref{assumption:a}, \ref{assumption:b}, \ref{assumption:c}, and~\ref{assumption:Bn_xi}, 
    if $g(N,T)$ is such that
    \[
    g(N,T) \rightarrow 0\quad\text{and}\quad C_{N,T} g(N,T) \rightarrow~\!\infty\quad\text{as}\quad C_{N,T} \rightarrow~\!\infty,
    \] 
   then
    $
        \pp( \hat r = r ) \longrightarrow 1$ 
as  $N$ and $T$ go to infinity.
\end{Theorem}

\begin{Remark}
    \label{rmk:consistency_number_of_factors}
    \begin{enumerate}
        \item[(i)]        
            This result tells us that the penalty function $g(N,T)$ must converge to zero  
            slow enough   that $C_{N,T}g(N,T)$ diverges. \citet{bai:ng:2002} obtain a
            similar result, but they have the less stringent condition on
            the rate of decay of $g(N,T)$, namely that $C_{N,T}^2 g(N,T) \rightarrow \infty$.
            However, they require stronger assumptions on the idiosyncratic components (for
            instance, they require that $\ee \hnorm{\xi_{it}}^7 < \infty$), see the erratum of
            \citet{bai:ng:2002}. Our stronger condition on the rate of decay of $g(N,T)$
            \citep[which is consistent with the results of][]{Amengual2007} allows for  
            much weaker conditions on the idiosyncratic component. The rate of
            decay of $g(N,T)$ could be changed to match the one of \citet{bai:ng:2002} provided
            the correlation of the idiosyncratic component is such that the operator norm of the $T \times T$ real matrix~$(\sc{ \bxi_s, \bxi_t })_{s,t=1,\ldots, T}$ is 
            $O_{\rm P}(NT/C_{N,T}^{2})$. Stronger conditions on $g(N,T)$ are preferable to stronger conditions on the idiosyncratic, though, since   $g(N,T)$ is under the statistician's control, whereas the distribution of the idiosyncratic is not.

        \item[(ii)] In the special case $H_i = \bbR$ for all $i \geq 1$, Theorem~\ref{thm:consistency_number_of_factors} looks similar to Theorem~2 in \citet{bai:ng:2002}. However, our conditions are 
           weaker, as we neither require~$\ee \hnorm{\bu_t}^4 <~\!\infty$ nor~$\ee
            \hnorm{\xi_{it}}^8 < M < \infty$ (an assumption that is unlikely to hold in
            most equity return series). 
            We also are weakening their assumption
            \[
                \ee \abs{ \sqrt{N} \left( \sc{\bxi_t,  \bxi_s}/N - \nu_N(t-s)
                \right) }^4  < M < \infty, \quad \forall s,t,N \geq 1,
            \]
            on idiosyncratic cross-covariances into 
            $\ee \abs{ \sqrt{N} \left(\sc{ \bxi_t, \bxi_s }/N - \nu_N(t-s) \right)}^2 < M$. 
            The main tools that allow us to derive results under weaker
            assumptions are H\"older inequalities between Schatten norms
             of compositions of operators (see Section~\ref{sec:schatten_Norms} in the \thesupplement), whereas classical results mainly use the
            Cauchy--Schwartz inequality.
        \item[(iii)] Estimation of the number of factors in the scalar case has been extensively studied: let us compare our assumptions with the assumptions made in some of the consistency results available in the literature.  The assumptions in~\citet{Amengual2007}   are similar to ours, except for their Assumptions (A7) and (A9) which we do not impose. \citet{Onatski2010} considers estimation of the factors in a setting allowing for ``weak'' factors, using the empirical distribution of the eigenvalues---his assumptions are not directly comparable to ours. \citet{lam:yao:2012} make  the (strong) assumption that the idiosyncratics are white noise, and their model  therefore is not comparable to ours. \citet{li2017identifying} assume the factors are independent linear processes, and that the idiosyncratics are independent, cross-sectionally as well as serially, which is quite a restrictive assumption, and require $N,T \to \infty$ with $N/T \to \alpha > 0$.
        \item[(iv)] As already mentioned, in practice, we recommend combining the method considered in Theorem~\ref{sub:inferring_the_number_of_factors} with the tuning device proposed in \cite{hallin:liska:2007} and \citet{alessi:barigozzi:capasso:2009}. A brief description of that tuning device is provided in Section~\ref{sec:simulation_number_of_factors}; its consistency readily follows from the consistency of the ``un-tuned" version. 
    \end{enumerate}
\end{Remark}

\subsection{Average Error Bounds}%
\label{sub:average_error_bounds}

In this section, we  provide average error bounds on the estimated factors, loadings, and common component. We  restrict  to the case $k=r$ (correct identification of the number of factors), in which case our results imply that the first $r$ estimated factors and loadings are consistent. The misspecified case $k \neq r$ is discussed in Section~\ref{s:simulation_r_misspecified} of the \thesupplement.

The first result of this section (Theorem~\ref{theorem:factors_average_bound}, see below) 
tells us, essentially, that~$\tildebU{r}_T$ consistently estimates 
the true factors $\bU_T$. Since the true factors are only identified up to an invertible linear
transformation, however, consistency here is about the convergence of the row space
spanned by $\tildebU{r}_T$ to the row space spanned by $\bU_T$. The discrepancy between these row spaces  can be measured by 
\[
    \delta_{N,T} \defeq \min_{\bR \in \bbR^{r \times r}}
    \Fnorm{ \tildebU{r}_T - \bR \bU_T}/{\sqrt{T}}.
\]
This $\delta_{N,T}$ is the rescaled Hilbert--Schmidt norm of the residual of the least squares fit of the rows of
$\tildebU{r}_T$ onto the row space of $\bU_T$, where the dependence on
$N$ and $T$ is made explicit.  
The $T^{-1/2}$ rescaling is needed because $\Fnorm{\tildebU{r}_T}^2 = rT$.

We  now can state one of the main results of this section.
\begin{Theorem}  \label{theorem:factors_average_bound}
    Under Assumptions~\ref{assumption:a}, \ref{assumption:b}, \ref{assumption:c}, and 
    \ref{assumption:Bn_xi}, letting $C_{N,T} \defeq \min \{\sqrt{N},\sqrt{T} \}$, 
    \begin{align*}
        \delta_{N,T} & =  O_{\rm P}(C_{N,T}^{-1})\qquad\text{as $N$ and $T$ tend to infinity.}
    \end{align*}
\end{Theorem}
This result, proved in Section~\ref{sec:proofs_average_bounds} of \thesupplement, essentially means that the factors are   consistently estimated (as both $N$ and $T$, hence $C_{N,T}$, go to infinity). 
Note in particular that $\delta_{N,T}  \equiv \delta_{N,T}(\tildebU{r}_T, \bU_T)$ is not
symmetric in $\tildebU{r}_T, \bU_T$, and
hence is not a metric. Nevertheless, small values of $\delta_{N,T}$ imply 
that the row space of the estimated factors is close to the row space of the true factors.
By classical least squares theory, we have
\[
    \delta_{N,T} = \Fnorm{ (\b I_T - \bP_{\bU_T})
    (\tildebU{r}_T)^\tp/\sqrt{T} },
\]
where $\bP_{\bU_T}$ is the orthogonal projection onto the column space of $\bU_T^\tp$. This formula
will be useful in Section~\ref{s:simu}.

Under additional constraints on the factor loadings and the factors and adequate 
additional assumptions, it is
possible to show that the estimated factors~$\tildebU{r}_T$ converge exactly (up to a
sign) to the true
factors~$\bU_T$ \citep{stock:watson:2002a}. Recalling $\b \Sigma_{\bB}$ and $\b \Sigma_\bu$ from Assumptions~\ref{assumption:a} and~\ref{assumption:b}, we need, for instance, 
\begin{customass}{E} \label{assumption:identification_factors}
    All the eigenvalues of  $\b \Sigma_{\bB} \b \Sigma_\bu$ are distinct.
\end{customass}
Under Assumptions~\ref{assumption:a}, \ref{assumption:b}, and
\ref{assumption:identification_factors},  for $N$ and $T$ large enough, we can choose the loadings
and factors such that $\bU_T \bU_T^\tp/T = {\bf I}_r$, and  $\bB_N^\adjoint \bB_N/N \in \bbR^{r \times r}$ is diagonal with
distinct positive diagonal entries whose gaps remain bounded from below, as~$N,T \to~\!\infty$. This new assumption, which we can make (under Assumptions~\ref{assumption:a}, \ref{assumption:b}, and
\ref{assumption:identification_factors}) without loss of generality, allows us to show that the factors are consistently estimated
 up to a sign.
\begin{Theorem} \label{theorem:factors_average_bound_up_to_sign}
    Let Assumptions~\ref{assumption:a}, \ref{assumption:b},
    \ref{assumption:c}, 
    \ref{assumption:Bn_xi}, and \ref{assumption:identification_factors} hold; assume furthermore  
      that, for~$N$ and~$T$
    large enough, $\bU_T \bU_T^\tp/T = {\bf I}_r$ and
    $\bB_N^\adjoint \bB_N/N \in \bbR^{r \times r}$ is diagonal with distinct decreasing entries.
    Then, there exists an $r\times r$ diagonal matrix~$\b R_{NT} \in \bbR^{r \times r}$ (depending on~$N$ and~$T$) with  entries $\pm
    1$ such that
    \begin{equation}\label{th44}
        \Fnorm{ \tildebU{r}_T- \b R_{NT}\bU_T} / \sqrt{T} = O_{\rm P}(C_{N,T}^{-1})\quad\text{as $N,\, T \to \infty$.}
    \end{equation}
    \end{Theorem}
This result is proved in Section~\ref{sec:proofs_average_bounds} of the \thesupplement. Note that \eqref{th44} does not assume any particular dependence between $N$ and~$T$: the order of the estimation error depends only on~$\min\{ N, T \}$. A couple of remarks are in order.
\begin{Remark}
    \label{rmk:consistency_theorem}
    \begin{enumerate}
        \item[(i)] As mentioned earlier, Assumptions~\ref{assumption:a}, \ref{assumption:b}, \ref{assumption:c}, 
            and \ref{assumption:Bn_xi} imply that the common and idiosyncratic components 
            are asymptotically identified, see Lemma~\ref{lma:bai_Ng_rth_largest_eigenvalue_of_covariance_diverges} and
            Theorem~\ref{thm:equiv_factormodel_eigenvalues}. The extra assumptions
          for consistent estimation of the factors' row space (and for the
            loadings and common component, see
            Theorems~\ref{theorem:loadings_average_bound} and
                \ref{theorem:common_component_average_bound} below) are needed because the
                covariance $\cov(\bx_t)$ is unknown, and its first $r$ eigenvectors must be estimated. 
        \item[(ii)]    Notice, in particular, that
            Theorem~\ref{theorem:factors_average_bound} holds for the case $H_i
            = \bbR$ for all~$i$, where it coincides with Theorem~1 of
            \citet{bai:ng:2002}. Our assumptions, however, are weaker, as discussed in Remark~\ref{rmk:consistency_number_of_factors}.
            \citet{stock:watson:2002a} also have consistency results for the
            factors, which are similar to ours in the special case~$H_i = \bbR$ for all $i$.
        \item[(iii)] Note that we could replace the $N^{-1}$ rate in Assumption~\ref{assumption:b} with 
            $N^{-\alpha_0}$,  $\alpha_0 \in (0,1)$, in which case we have \emph{weak}
            (or \emph{semi-weak}) factors
            \citep{Onatski2010,chudik:pesaran:tosetti:2011,lam:yao:2012}, which would affect the
            rate of convergence in Theorems~\ref{theorem:factors_average_bound},
            \ref{theorem:factors_average_bound_up_to_sign},
            \ref{theorem:loadings_average_bound}, and
            \ref{theorem:common_component_average_bound}; see also
            \citet{Boivin2006}.
        \item[(iv)] We do not make any Gaussian assumptions and, unlike
            \citet{lam:yao:2012}, we do not assume that the idiosyncratic component
            is white noise. 
        \item[(v)]\citet{bai:ng:2002} allow for limited correlation between
            the factors and the idiosyncratic components. This, which betrays their interpretation of the factor model representation~\eqref{eq:r-factor-model} as a data-generating process, is only an
            illusory increase of generality, since     it is tantamount to artificially   transferring to the unobserved 
            idiosyncratic part of the contribution of the factors. Such a transfer is pointless, since it cannot bring any inferential improvement: the resulting model 
             is observationally equivalent to  the original one and the estimation method, anyway, (asymptotically) will reverse the transfer between common and idiosyncratic. 
        \item[(vi)] 
            The results could be extended to conditionally heteroscedastic
            common shocks and idiosyncratic components, as frequently assumed
            in the scalar case \citep[see,
            e.g.,][]{alessi2009estimation, barigozzi2019general, hallin2019forecasting}. This, which
            would come at the cost of additional identifiability constraints,
            is left for further research.   
    \end{enumerate}
\end{Remark}

In order to obtain average error bounds for the loadings, let us consider the following condition.
\begin{customass}{F($\alpha$)} \label{assumption:u_xi}
    Letting $C_{N,T} \defeq \min\{ \sqrt{N}, \sqrt{T} \}$, 
    \[
        \Fnorm{ \frac{1}{T} \sum_{t=1}^T \bu_t \tensor \bxi_t }^2 = O_{\rm P}\left(N C_{N,T}^{-(1+\alpha)}\right)
\qquad\text{ for some   $\alpha \in [0, 1]$. }    \]
\end{customass}
Assumption~\ref{assumption:u_xi} can be viewed as providing a concentration bound, and
implicitly imposes limits on the lagged dependency of the series $\{ \bu_t \tensor \bxi_t \}$. 
Notice that Assumptions~\ref{assumption:a} and~\ref{assumption:c} jointly imply Assumption~\ref{assumption:u_xi} for $\alpha = 0$
(see Lemma~\ref{lma:Frobenius_Norm_of_xi}  in the \thesupplement), so that
$\alpha = 0$ corresponds to the absence of restrictions on these cross-dependencies; 
$\alpha = 1$ is the strongest case of this assumptions, and corresponds to the
weakest cross-dependency between factors and idiosyncratics (within $\alpha \in [0,1]$): it is implied by
the following stronger (but more easily interpretable) conditions (see Lemma~\ref{lma:norm_bu_bxi}  in
the \thesupplement):
\begin{enumerate}
    \item[(i)] Assumption~\ref{assumption:a} holds, and $\eee{ \sc{\bxi_t, \bxi_s}  u_{lt}u_{ls}} = \eee{ \sc{\bxi_t, \bxi_s}} \eee{ u_{lt}u_{ls}}$ for all  $s,t \in \bbZ$ and all~$l = 1,\ldots, r$;
    \item[(ii)]  there exists some $M < \infty$ such that $\sum_{h \in \bbZ} \abs{\nu_N(h)} < M,$  for all $N \geq 1$.
\end{enumerate}
Note that Assumption~\ref{assumption:u_xi} with $\alpha=1$ is still less stringent
than Assumption~D in \citet{bai:ng:2002}.

The next result of this section  deals with the consistent  estimation of the
factor loadings. Let $\barB{r}_N \in \bounded{\bbR^r, \bH_N}$ be the operator
mapping, for , $l = 1,\ldots, k$, the $l$-th canonical basis vector of $\bbR^r$ to $\hat \be_l$ where $\hat \be_l$ is defined in Proposition~\ref{prop:computing_loadings_and_factors}. 
Notice that $\barB{r}_N$ has the same range as $\tildeB{r}_N$.
Similarly to the factors, the factor loading operator $\bB_N$ is only
identified up to an invertible transformation, and we therefore measure the 
consistency of $\barB{r}_N$ by quantifying the discrepancy between the range of
$\barB{r}_N$ and the range of $\bB_N$ through
\begin{equation}
    \label{eq:distance_tildeBn_Bn}
    \vep_{N,T} \defeq \min_{\bR \in \bbR^{r \times r}}
    \Fnorm{ \barB{r}_N  - \bB_N \bR}/\sqrt{N}.
\end{equation}
To better understand this expression, let us consider the case~$\bH_N = \bbR^N$ where~$\barB{r}_N$ and~$\bB_N$ are elements of  $\bbR^{N \times r}$. In this case, $\vep_{N,T}$ is the
Hilbert--Schmidt norm of the residual of the columns of $\barB{r}_N$ projected
onto the column space of $\bB_N$, which depends on both~$N$ and~$T$. 
A similar interpretation holds for the general case. 
The $\sqrt{N}$ renormalization is needed as~$\Fnorm{\barB{r}_N}^2 = r N$. We then
have, for the factor loadings, the following consistency result (proved in
Section~\ref{sec:proofs_average_bounds} of the \thesupplement).
\begin{Theorem}  \label{theorem:loadings_average_bound}
    Under Assumptions~\ref{assumption:a}, \ref{assumption:b}, \ref{assumption:c},
    \ref{assumption:Bn_xi}, and \ref{assumption:u_xi},
    \begin{align*}
        \vep_{N,T} & =  O_{\rm P}\left(C_{N,T}^{-\frac{1+\alpha}{2}}\right),
    \end{align*}
    where we recall $C_{N,T} \defeq \min \{\sqrt{N},\sqrt{T} \}$.
\end{Theorem}
 The rate of convergence for the loadings thus  crucially depends on the value of~$\alpha \in~\![0,1]$ in
 Assumption~\ref{assumption:u_xi}. The larger $\alpha$ (in particular, the weaker the lagged dependencies in~$\{ \bu_t \tensor \bxi_t \}$), the better the rate.  Unless $\alpha = 1$, that rate is slower than for the estimation of the
factors. As in Theorem~\ref{theorem:factors_average_bound_up_to_sign},  it could be shown that, under additional identification
constraints, the loadings can be estimated consistently up to a
sign. Details are left to the reader. 

We can  turn now to the estimation of the common
component $\{\bchi_t\}$ itself. Recall the definition of $\hat \chi_{it}^{(r)}$ from \eqref{eq:hat_chi_it}.
Using Theorems~\ref{theorem:factors_average_bound} and
\ref{theorem:loadings_average_bound}, we obtain the following result, the proof of which is in Section~\ref{sec:proofs_average_bounds} of the \thesupplement.
\begin{Theorem} \label{theorem:common_component_average_bound}
    Under Assumptions~\ref{assumption:a}, \ref{assumption:b}, \ref{assumption:c},
    \ref{assumption:Bn_xi}, and \ref{assumption:u_xi}, 
        \[
            \frac{1}{{NT}}{ \sum_{i=1}^N \sum_{t=1}^T \hnorm{ \chi_{it} - \hat \chi_{it}^{(r)} }_{H_i}^2 } =
            O_{\rm P}\left(C_{N,T}^{-(1+\alpha)} \right),\qquad \alpha \in [0,1].
    \]
\end{Theorem}
The  ${NT}$ renormalization is used because $\sum_{i=1}^N  \sum_{t=1}^T \hnorm{\chi_{it}}_{H_i}^2$ is of order ${NT}$. 
Again, the rate of convergence depends on $\alpha$, which quantifies the strength of
lagged dependencies in~$\{ \bu_t \tensor \bxi_t \}$.  \bigskip

\subsection{Uniform Error Bounds}%
\label{sub:uniform_error_bounds}

In this section, we provide uniform error bounds for the estimators of factors, in the spirit of \citet[][Theorem~4 and Corollary~1]{fan:2013:jrssb-discussion}. The proofs of the results are in Section~\ref{sec:proofs_uniform_bounds} of the \thesupplement. Our first result gives uniform bounds for the estimated factors. For this, we need the following regularity assumptions.
\begin{customass}{G($\kappa$)} \label{assumption:uniform_bound_factors}
    For some integer $\kappa \geq 1$, there exists $M_2 < \infty$ such that
    \[\max\left(
        \ee \abs{ \sqrt{N}\left(  N^{-1}  \sc{\bxi_{t}, \bxi_{s}}   - \nu_N(t-s) \right) }^{2 \kappa},\ \ee \rpnorm{N^{-1/2} \bB_N^\adjoint \bxi_t}^{2 \kappa},\  \ee \rpnorm{\bu_t}^{2 \kappa}
        \right) < M_2
    \]
    for all $N,\, t,\, s \geq 1$. 
\end{customass}
Let 
\begin{equation}
    \label{eq:defn_tildeR_paper}
    \tilde \bR \defeq \hat \bLambda^{-1} \tildebU{r}_T \bU_T^\adjoint \bB_N^\adjoint \bB_N/(NT) \in \bbR^{r \times r},
\end{equation}
where $\hat \bLambda \in \bbR^{r \times r}$ is the diagonal   matrix with diagonal elements~$\hat \lambda_1,\ldots, \hat \lambda_r$.
\begin{Theorem}
    \label{thm:unif_bound_factors}
    Under Assumptions~\ref{assumption:a}, \ref{assumption:b}, \ref{assumption:c}, \ref{assumption:Bn_xi}, and~\ref{assumption:uniform_bound_factors}, 
    \[
        \max_{t=1,\ldots, T} \rpnorm{ \tilde \bu_t - \tilde \bR \bu_t } = O_{\rm P}\left( \max\left\{ \frac{1}{\sqrt{T}}, \frac{T^{1/(2 \kappa)}}{\sqrt{N}} \right\}\right). 
    \]
\end{Theorem}
We see that if $N,T \to \infty$ with $T = o(N^ \kappa)$, the estimator of the factors is uniformly consistent. 
For $ \kappa=2$ and the  special case $\bH_N = \bbR^N$, we recover exactly the same rate of convergence as \citet[][Theorem~4]{fan:2013:jrssb-discussion}. Our Assumption~\ref{assumption:uniform_bound_factors}, for~$\kappa=2$, is similar to their Assumptions~4. Some of our other assumptions, however, are weaker: we do not need   lower bounds on the variance of the idiosyncratics, do not assume that the factors and idiosyncratics are sub-Gaussian, but impose instead the weaker assumption $\ee \rpnorm{\bu_t}^4 < \infty$). We also do not need to assume exponential $\alpha$-mixing for the factors and idiosyncratics, but assume instead that $T^{-1} \sum_{t=1}^T \bu_t \bu_t^\tp \convp \b \Sigma_\bu$. 
Notice moreover that our bound is more general than the one in \citet{fan:2013:jrssb-discussion}, since it shows the dependency on the rate at which $T$ can diverge with respect to $N$ as a function of the number of moments in Assumption~\ref{assumption:uniform_bound_factors}.

We also provide some uniform error bounds on the estimated loadings and common component in Section~\ref{sec:uniform_bounds_loadings_commonComponent} of the \thesupplement.

\section{Numerical Experiments}\label{s:simu}

In this section, we assess the finite-sample performance of our estimators on simulated panels $\{ x_{it}: 1 \leq i \leq N, 1 \leq t \leq T\}$.

Panels of size $(N=100)\times (T=200)$   were generated as follows from a functional factor
model with three factors. All functional time series in the panel   are represented in an
orthonormal basis of dimension 7, with basis functions  $ \vfi_{1},   \ldots,  \vfi_{7}$;
the particular choice of the orthonormal functions $\{ \vfi_j \}$ has no influence on the results.
Each of the three factors is independently generated from a Gaussian AR(1) process
with coefficient~$a_{k}$ and variance~$1- a_{k}^{2}$,~$k=1,2,3$. These coefficients
are picked at random from a uniform on $(-1,1)$ at the beginning of the simulations
and kept fixed across the 500 replications. The $a_k$s are then rescaled so
that the operator norm of the companion matrix of the three-dimensional VAR process
$\bu_{t}$ is $0.8$.

The factor loading operator $\bB_N\! \in\! \bounded{\bbR^3\!, \bH_N}$ was chosen such that $(u_{1t}, u_{2t}, u_{3t})^\tp\!\! \in~\!\!\bbR^3$ is mapped to the vector in $\bH_N$ whose $i$-th entry is $\sum_{l=1}^3 u_{lt} \tilde b_{il} \vfi_l$ where~$\tilde b_{il} \in\! \bbR$.
This implies that each common component $\chi_{it}$ always belongs to the space spanned by 
the first three basis functions~$\vfi_l$, $l=1,2,3$.
The coefficient matrix $\tilde{\bf B} = (\tilde b_{il}) \in \bbR^{N \times 3}$ 
uniquely defines the loading operator $\bB_N$.
Those $3 N$ coefficients were generated as follows:  
first pick a value at random
from a uniform over~$[0,1]^{3N}$,  then rescale each fixed-$i$ triple  to have unit Euclidean
norm. This rescaling implies that the total variance of each common component (for
each $i$) is equal to 1;  $\tilde{\bf B}$ is kept fixed across replications.

The idiosyncratic components belong to the space spanned by $\vfi_1,\ldots,\vfi_7$; their
coeffi\-cients~$(\sc{\xi_{it},  \vfi_{j}})_j$ were generated from  
\[
   (\sc{\xi_{it},  \vfi_{1}},\ldots,\sc{\xi_{it},  \vfi_{7}} )^\tp
   \stackrel{\text{i.i.d.}}{\sim} \mathcal{N}(\b 0, c \cdot \b E/\trace(\b E)), \; i=1,\ldots,N, \;
    t=1,\ldots,T.
\]
($\b E$ a $7\times 7$ real matrix). Since the total variance of each common component  is one, the constant $c $ 
is the
relative amount of idiosyncratic noise:  $c=1$ means equal common and idiosyncratic variances, while  
 larger values of $c$ makes estimation of factors, loadings, and
common components more difficult. 
We considered  four Data Generating Processes (DGPs): 
\begin{enumerate}
    \item[]\texttt{DGP1}: $c=1$, $\b E = \mathrm{diag}(1, 2^{-2}, 3^{-2},\ldots,
        7^{-2})$,
    \item[]\texttt{DGP2}: $c=1$, $\b E = \mathrm{diag}(7^{-2}, 6^{-2}, \ldots, 1)$,
    \item[]\texttt{DGP3}: $c=7$, $\b E = \mathrm{diag}(1, 2^{-2}, 3^{-2},\ldots,
        7^{-2})$,
   \item[]\texttt{DGP4}: $c=7$, $\b E = \mathrm{diag}(7^{-2}, 6^{-2}, \ldots, 1)$.
\end{enumerate}

In \texttt{DGP1} and  \texttt{DGP3}, we have chosen to  align 
  the largest idiosyncratic variances with the span of the factor loadings (spanned by~$\{\varphi_1,\varphi_2,
\varphi_3\}$).  In this case,
$\tildeB{3}_N \tildebU{3}_T$
is picking  the three 
common shocks, but also the idiosyncratic components (which have large
variances). On the contrary, in \texttt{DGP2}, \texttt{DGP4}, we have chosen the directions of  
largest idiosyncratic variance to be orthogonal to the span of the factor loadings.

The variance of
the idiosyncratic components (for each $i$) is equal to $1$ for \texttt{DGP1} and  \texttt{DGP2},
  equal to $7$ for \texttt{DGP3} and  \texttt{DGP4}; the latter thus  are more
difficult. 
In particular, while one might feel that performing a \KL truncation for each
separate FTS (each~$i$) is a good idea, this  actually would perform quite
poorly in \texttt{DGP4}, where the first (population) eigenfunctions (for each
$i$) are \emph{exactly} orthogonal to the common component. This phenomenon is corroborated by our application to forecasting, see Section~\ref{sec:forecasting_mortality_curves}.

\subsection{Estimation of the factors, loadings, and common component}%
\label{sub:estimation_of_the_factors_loadings_and_common_component}

We only  consider here the average bounds  of Section~\ref{sub:average_error_bounds}.  
For $N=10,25,50, 100$ and~$T=50, 100, 200$, we have considered the subpanels of the first $N$ and $T$ observation from the ``large'' $100\times 200$ panel. 
For each replication and each choice of $N$ and $T$, we estimated the factors and
factor loadings using the least squares approach described in Proposition~\ref{prop:computing_loadings_and_factors} from  the $N \times T$  panel, assuming that the number of factors is known to be $3$. We have then computed the approximation error $\delta_{N,T}^2$ for the factors,~$\vep_{N,T}^2$ for the 
loadings, and $\phi_{N,T} $ for the common component  (see Section~\ref{sub:average_error_bounds}), with 
\begin{equation}
    \label{eq:phiNT}
    \phi_{N,T} \defeq (NT)^{-1} \sum_{i=1}^N \sum_{t=1}^T \hnorm{ \chi_{it} - \hat{ \chi}^{(3)}_{it}}_{H_i}^2.
\end{equation}

The results, averaged over the 500 replications, are shown in
Figures~\ref{fig:DGP1}--\ref{fig:DGP2} and
Figures~\ref{fig:DGP3}--\ref{fig:DGP4} below, for \texttt{DGP1}, \texttt{DGP2},
\texttt{DGP3}, and \texttt{DGP4}, respectively.  A careful inspection of these
figures allows one to infer whether the asymptotic regime predicted by the
theoretical results (see Section~\ref{sub:average_error_bounds}) has been
reached.  We will give a detailed description of this for \texttt{DGP1},
Figure~\ref{fig:DGP1}.

Looking at the left plot in Figure~\ref{fig:DGP1}, 
the local slope $a$ of the curve $\log_2(N) \mapsto
\log_2\delta_{N,T}^2$ for fixed $T$ tells us that the error rate
is $N^{a}$, for fixed $T$. Here, $a \approx -1$; hence,  the error rates for
the factors is about~$N^{-1}$ for each  $T$.
For $N$ fixed, the spacings $b$ between $\log_2\delta_{N,T}^2$ from~$T=~\!50$ to~$100$ indicate that the error rate is $T^{b}$ for $N$ fixed. Since $0 \leq - b <
0.25$, the error rate for fixed $N$ is less than $T^{-0.25}$.
 The simulation results    give
us insight into which of the terms~$T^{-1}$ or~$N^{-1}$ is dominant, and
for the factors in \texttt{DGP1}, the dominant term in $N^{-1}$ for~$T \in [50,
200]$.
We do expect to see an error rate  $T^{-1}$
for large fixed $N$ large, 
and simulations (with $N=1000$, not shown here) confirm that this is indeed the case.
The middle plot of Figure~\ref{fig:DGP1} shows the error rate for the loadings. Since the factors and the idiosyncratic component are independent in our simulations, we expect to have the same~$O_{\rm P}(\max(T^{-1}, N^{-1}))$ error rates as for the factors. For the larger values of $N$, it is clear that the dominant term is $T^{-1}$. Smaller values of $N$ actually exhibit a transition from a $N^{-1}$   to a~$T^{-1}$ regime: the spacings $b$ between the lines become more uniform and are  close to~$-1$ as~$N$ increases, and the slope $a$  decreases in magnitude as $N$ increases, and seems to converge to zero.  The right sub-figure shows the error rates for the common component, for which we expect, in this setting,  the same $O_{\rm P}(\max(T^{-1}, N^{-1}))$ error rates as for the loadings.  Inspection reveals similar effects as for the factor loadings: for~$T=200$, the error rate is close to $N^{-1}$ for small values of~$N$. For $N = 100$, it is almost $T^{-1}$ for small~$T$ values. Figure~\ref{fig:DGP2} (\texttt{DGP2}) can be interpreted in a similar fashion. 

\begin{figure}[htbp]
    \centering
    \begin{subfigure}[b]{1\textwidth}
        \includegraphics[width=\linewidth]{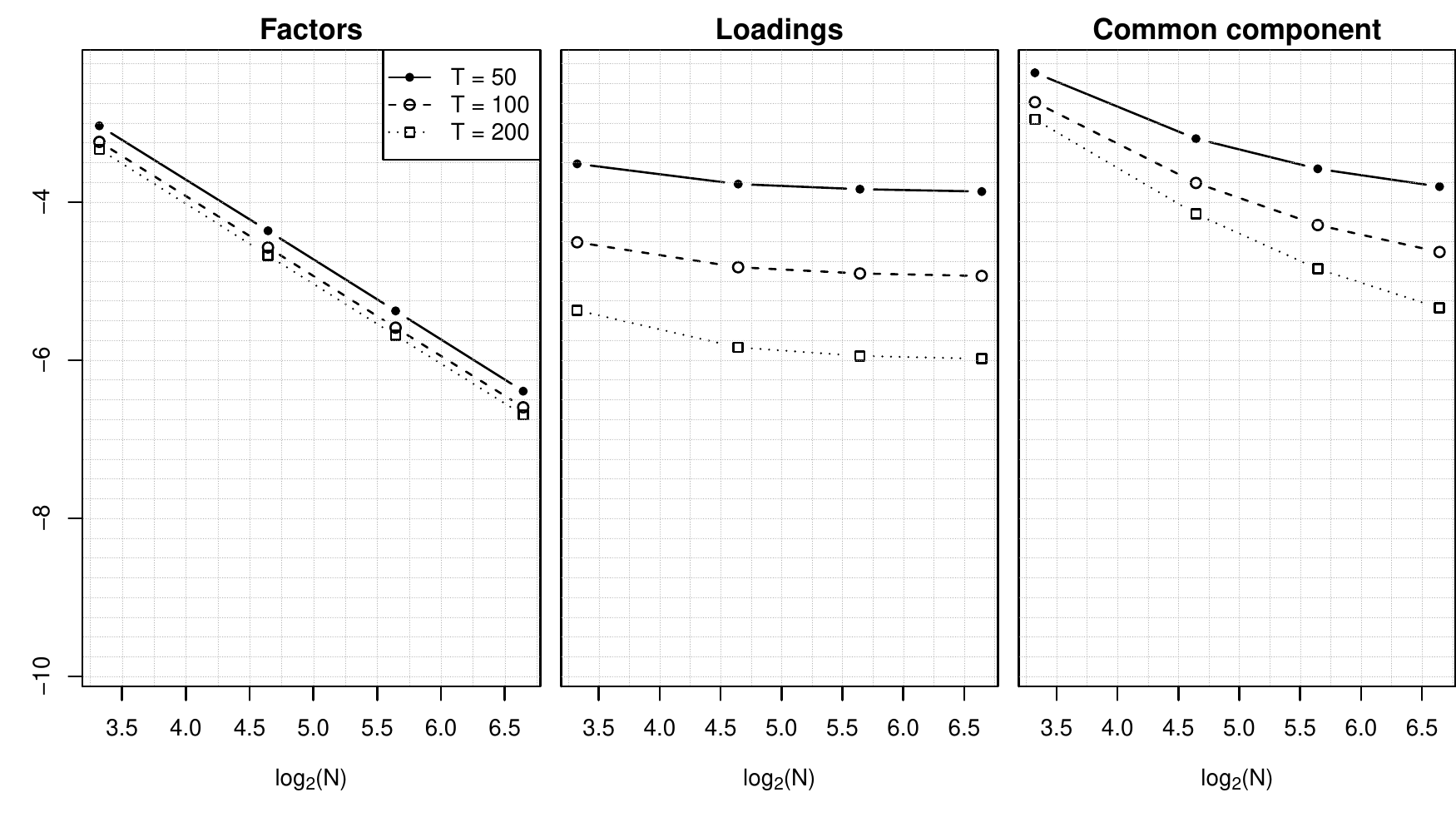}
        \caption{Simulation scenario \texttt{DGP1}.}
        \label{fig:DGP1}
    \end{subfigure}
    \vskip 11pt
    \begin{subfigure}[b]{1\textwidth}
        \includegraphics[width=\linewidth]{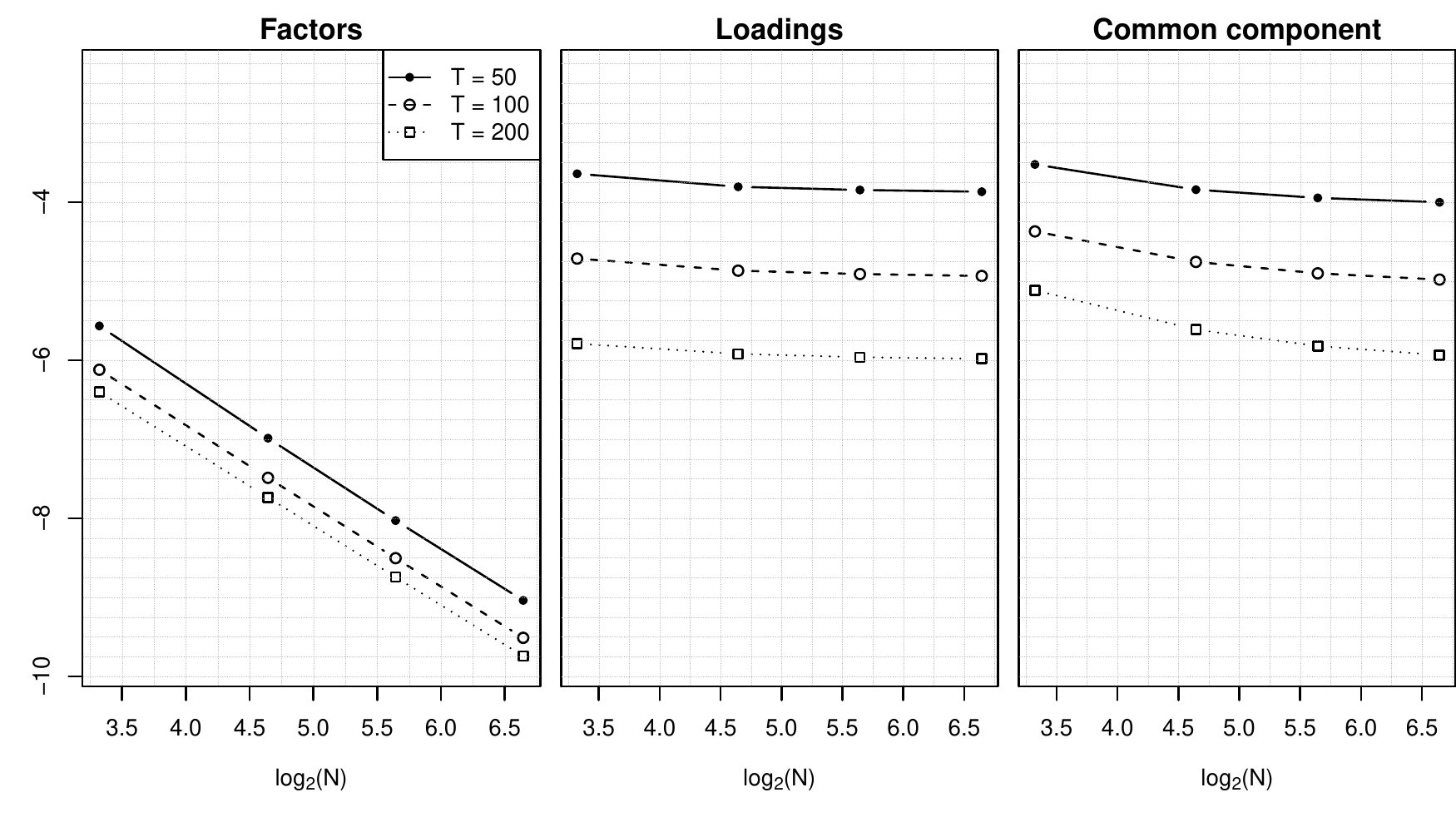}
        \caption{Simulation scenario \texttt{DGP2}.}
        \label{fig:DGP2}
    \end{subfigure}
    \caption{Estimations errors (in $\log_2$ scale) for \texttt{ DGP1 } (subfigure
        (a)) and \texttt{ DGP2 }  (subfigure (b)). For each subfigure, we provide the
        estimation error for the factors ($\log_2 \delta_{N,T}^2$, left), loadings
        ($\log_2 \vep_{N,T}^2$, middle), and common component ($\log_2 \phi_{N,T}$,
        right, $\phi_{N,T}$ defined in~\eqref{eq:phiNT}) as functions of $\log _2N$.
        The   scales on the  vertical axes are the same. Each curve corresponds to
        one value of $T \in \{50, 100, 200\}$, sampled for~$N \in \{10, 25, 50,
        100\}$.   
    }
    \label{fig:estimation-errors-DGP1-2. }
\end{figure}

%

\begin{figure}[htbp]
    \centering
    \begin{subfigure}[b]{1\textwidth}
        \includegraphics[width=\linewidth]{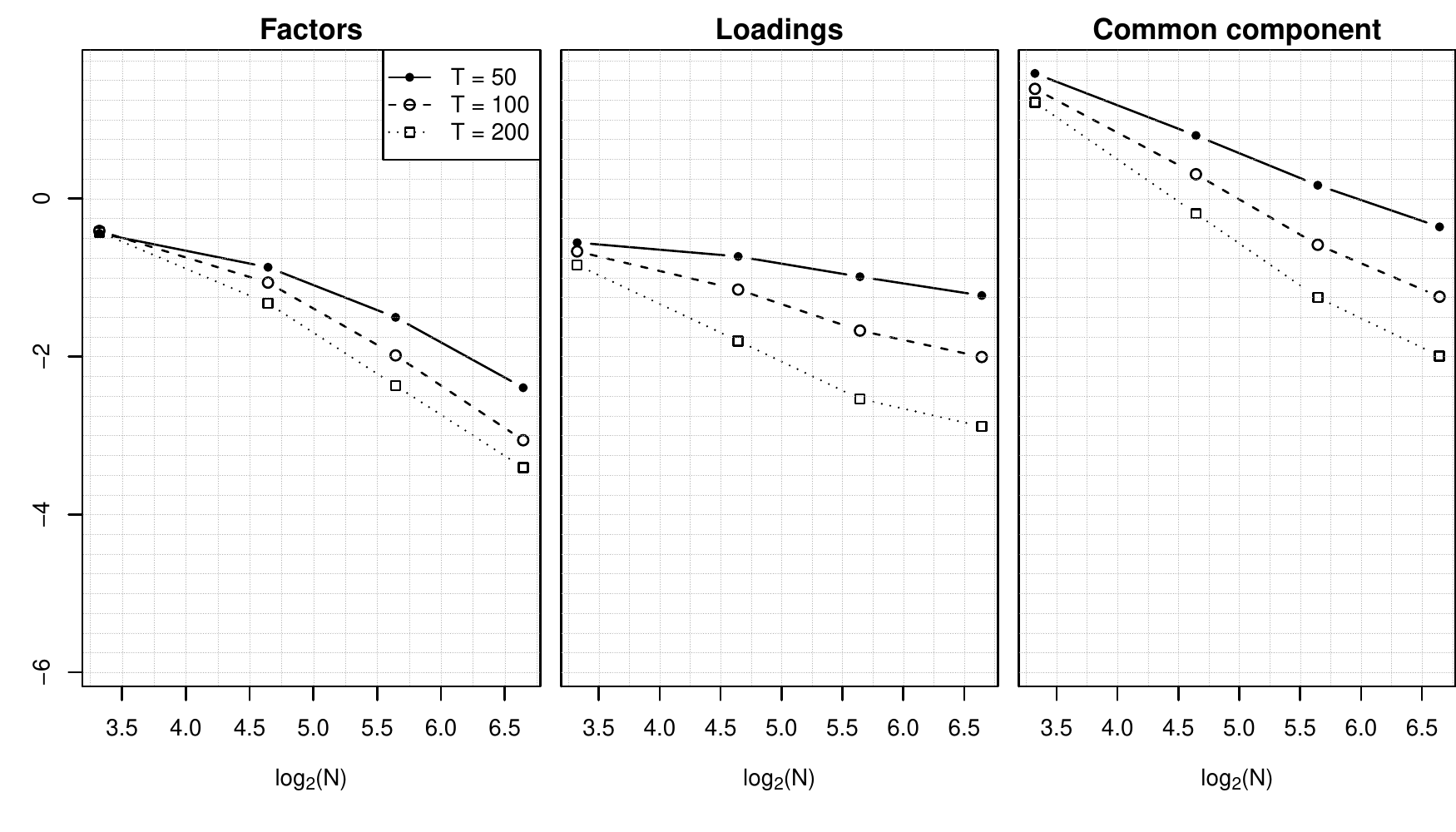}
        \caption{Simulation scenario \texttt{DGP3}.}
        \label{fig:DGP3}
    \end{subfigure}
    \vskip 11pt
    \begin{subfigure}[b]{1\textwidth}
        \includegraphics[width=\linewidth]{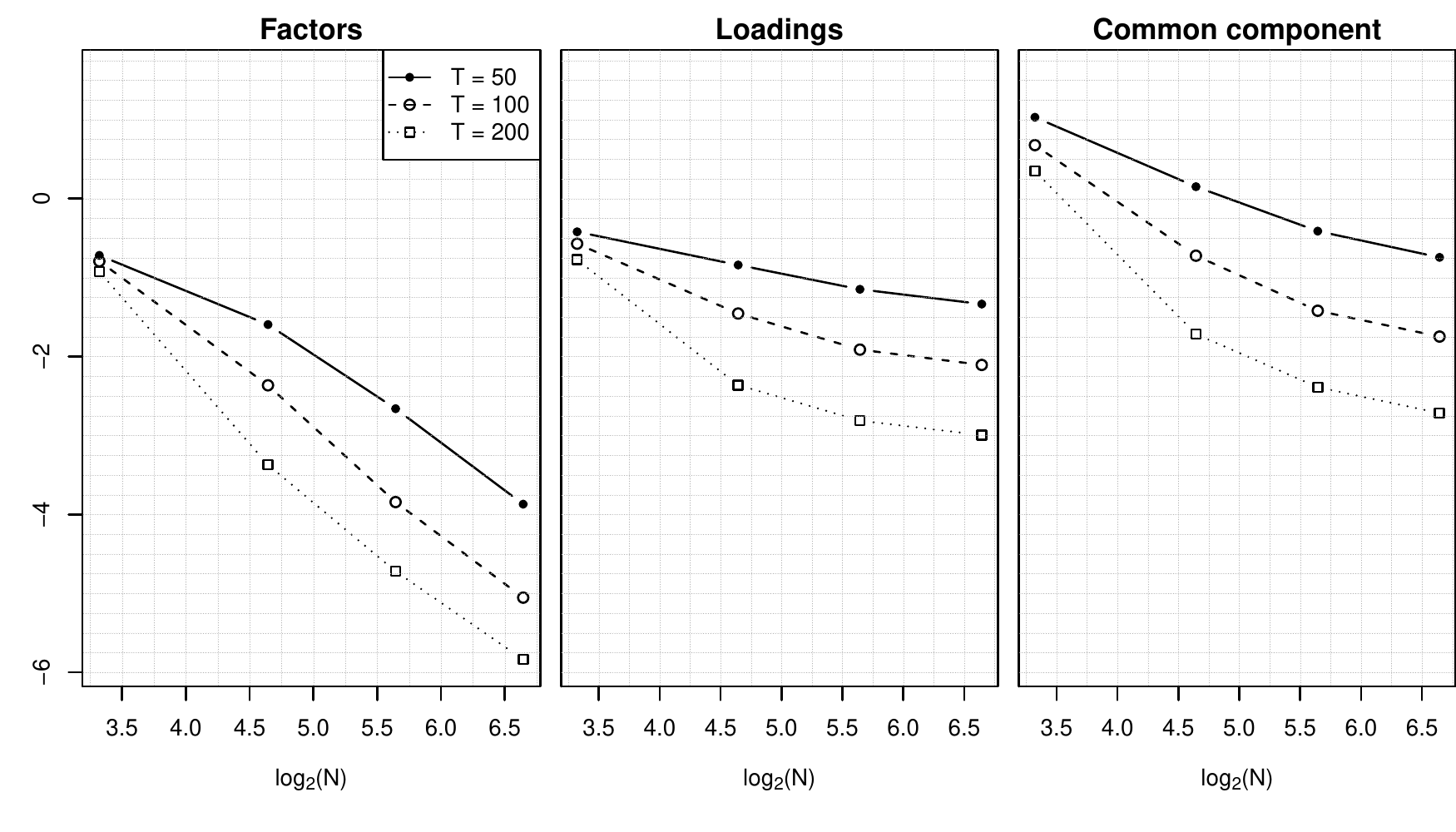}
        \caption{Simulation scenario \texttt{DGP4}.}
        \label{fig:DGP4}
    \end{subfigure}
    \caption{Estimations errors (in $\log_2$ scale) for \texttt{ DGP3 } (subfigure
        (a)) and \texttt{ DGP4 }  (subfigure (b)). For each subfigure, we have the
        estimation error for the factors ($\log_2 \delta_{N,T}^2$, left),
        loadings ($\log_2 \vep_{N,T}^2$, middle), and common component ($\log_2 \phi_{N,T}$,
        right, $\phi_{N,T}$ defined in~\eqref{eq:phiNT}) as functions of $\log _2N$. The   scales of the  vertical axes are the same. Each curve
        corresponds to one value of $T \in \{50, 100, 200\}$, sampled for  $N \in
    \{10, 25, 50, 100\}$. }
    \label{fig:estimation-errors-DGP3-4}
\end{figure}

Results for \texttt{DGP3} and  \texttt{DGP4} are shown in
Figure~\ref{fig:estimation-errors-DGP3-4}. A comparison between \texttt{DGP3}, \texttt{DGP4} and  \texttt{DGP1}, \texttt{DGP2}  is interesting because they only differ by  the scale of the idiosyncratic components. We see that the errors are much higher for \texttt{DGP3}, \texttt{DGP4} 
than for \texttt{DGP1}, \texttt{DGP2}, as expected. Notice that the dominant term for the
factors is no longer  of order $N^{-1}$ over all values of~$N,T$: it seems
to kick in for $N \in [25,100]$ in \texttt{DGP3}, \texttt{DGP4}, but looks slightly higher than $N^{-1}$ for $N \in [10, 25]$ and $T \in [100,200]$ in \texttt{DGP4} (noticeably so for~$T=200$).
A similar phenomenon occurs for the loadings and common component in \texttt{DGP4}, for $N \in [10,25]$
and~$T=200$. These rates do not contradict the theoretical results of
Section~\ref{sub:average_error_bounds}, which hold for~$N,T \to \infty$, so \texttt{DGP4}, in
particular, indicates that the values of $N$ considered there are too small for the
asymptotics to have kicked in, and prompts further theoretical investigations about the
estimation error rates  in finite samples.

\subsection{Estimation of the number of factors}
\label{sec:simulation_number_of_factors}

Following the results of Section~\ref{sub:inferring_the_number_of_factors},
we suggest a data-driven identification of the number  $r$ of factors. Recall the definition of ${\cal X}_{N,T}$ in \eqref{eq:calX_NT}.
The  criteria that we will use for identifying the number of factors are combining the information-theoretic criteria  described in Section~\ref{sub:inferring_the_number_of_factors} with the tuning and permutation-enhanced method of \citet{alessi:barigozzi:capasso:2009} (\texttt{ABC} below). That method considers  criteria 
of the form
\begin{equation}
    \label{eq:IC_defn_with_scaling}
    \IC(c,{\cal X}_{N,T},k) \defeq V(k, \tildebU{k}_T) + c\, k\, g(N,T),
\end{equation}
which is the same as \eqref{eq:IC_defn_theoretical}, but with a tuning $c > 0$ of the penalty
term. 
 Indeed, if a penalty function $g(N, T)$ satisfies the assumptions of
Theorem~\ref{thm:consistency_number_of_factors}, then $c\, g(N,T)$ also satisfies 
these assumptions. The choice of  the tuning parameter $c$ itself is data-driven, which  allows the 
information criterion to be calibrated for each dataset. 

 In the \texttt{ABC} methodology, one estimates $r$ for a grid  $C=\{c_{i}: \; i=1,\ldots,I\}$ of tuning parameter values. For every point in the grid, we shuffle the cross-sectional order several times  then create sub-panels of sizes $\{(N_{j},T_{j}), \, j=1,\ldots,J\}$ with $(N_{J},T_{J})=(N,T)$ by selecting the top-left part of the panel. For every $i$ and $j$, we repeat these procedures $P$ times on random permutations of the cross-sectional order. For every $i=1,\ldots,I$, $j=1,\ldots,J$, and  every permutation $p=1,\ldots,P$, we compute  
\begin{equation}
    \label{eq:hat_r_data-driven}
    \hat r(c_{i},j,p) \defeq \argmin_{k=1,\ldots, k_{\mathrm{max}}} \IC(c_{i},{\cal X}_{N_{j},T_{j}}^{(p)} ,k).
\end{equation}
Define now the sample $\hat r_{i,p}\defeq \{ \hat r(c_{i},j,p): j=1,\ldots,J\}$. One finally selects the optimal scaling $\hat{c}_{i,p}$ as the middle value of the second plateau of the function $c_{i} \rightarrow \mathrm{Var}\big[\hat r_{i,p}\big]$ and  the number $\hat r$ of factors  as the median of $\{ \hat r(\hat{c}_{i,p},J,p): p=1,\ldots, P\}$. The consistency of~$\hat r$ readily follows from that of the ``un-tuned'' method---see  
 \cite{hallin:liska:2007} for details. 

For our simulation, we have adopted $k_{\mathrm{max}} \defeq 10$, $P \defeq 5$, $N_j \defeq \left\lfloor
\frac{4N}{5}+N(j-1)/45\right\rfloor$  $(j=1,\ldots, 10)$,  and $T_{j} \defeq T$,  with a grid   $C = \{0,0.05,0.1,\ldots,10\}$ for the values of $c$ and 
the following adaptations of the \texttt{IC1} and \texttt{IC2} penalty functions of
\cite{bai:ng:2002}:
\begin{align*}
g_{\texttt{IC1a}}(N,T) &\defeq  \sqrt\frac{N+T}{NT} \log\bigg(\frac{NT}{N+T}\bigg),  \\
g_{\texttt{IC2a}}(N,T) &\defeq  \sqrt\frac{N+T}{NT} \log C_{N,T}^2 .
\end{align*}

In order to analyze the performance of the corresponding identification procedures, we have
generated data based on \texttt{DGP1} for $N \in \{10, 25, 50, 100\}$ and $T \in \{25, 50, 100, 200\}$.
 Among 100 replications, we have counted the number of overestimations and
underestimations of $r$. The results are displayed in Table~\ref{t:estimation_r}.
We see that the number of factors is almost perfectly estimated if $N \geq 50$ or $T \geq 100$, or if $N \geq 25$ and $T \geq 50$. Furthermore, that number is almost never overestimated. For   smaller panel sizes ($N=10, T=25$---but this is really very small) the criteria give a correct estimate about half of the time, with the other half yielding an underestimation.
%
\begin{table}[ht]
\centering
\small
\begin{tabular}{l|rrrr||rrrr}
            & \multicolumn{4}{c||}{\texttt{IC1a}} & \multicolumn{4}{c}{\texttt{IC2a}} \\
 & $T\!=\! 25$ & $T\!=\! 50$ & $T\!=\! 100$ & $T\!=\! 200$ & $T\!=\! 25$ & $T\!=\! 50$ & $T\!=\! 100$ & $T\!=\! 200$ \\ 
  \hline
  $N=10$ &  52 (0)&  17 (0)&   1 (0)&   0 (0)&  54 (0)&  20 (0)&   1 (0)&   0 (0)\\ 
  $N=25$ &  13 (0)&   0 (0)&   0 (0)&   0 (0)&  15 (0)&   0 (0)&   0 (0)&   0 (0)\\ 
  $N=50$ &   2 (1)&   0 (0)&   0 (0)&   0 (0)&   2 (0)&   0 (0)&   0 (0)&   0 (0)\\ 
  $N=100$ &   0 (1)&   0 (0)&   0 (0)&   0 (0)  &   0 (1)&   0 (0)&   0 (0)&   1 (0)\\ 
\end{tabular}
\caption{Number of underestimation (overestimation) of the number of factors using \texttt{IC1a}, \texttt{IC2a} among 100 replications of \texttt{DGP1}. }
\label{t:estimation_r}
\end{table}

\subsection{Estimation with Misspecified Number of Factors}
\label{sub:estimation_misspecified}

The influence on the estimation of the common component  of a misspecification of the number of factors  is investigated  in Section~\ref{s:simulation_r_misspecified} of the \thesupplement.  The estimation error  appears to be  smallest when the estimated number of factor is correctly identified, provided that $N$ and $T$ are ``large enough,'' where ``large enough'' depends on the ratio between   total    idiosyncratic  and    total   common variances; see  Section~\ref{s:simulation_r_misspecified}  in the \thesupplement\ for  a detailed discussion.

\section{Forecasting of mortality curves}
\label{sec:forecasting_mortality_curves}

In this section, we show how the theory developed can be used to accurately forecast a high-dimensional vector of FTS. We consider the \citet{japonese_dataset} dataset  that contains age-specific and gender-specific mortality rates for the 47 Japanese prefectures from~1975  through 2016. 
For every prefecture and every year, the mortality rates are grouped by gender (Male or Female) and by age (from 0 to~109). People aged 110+ are grouped together in a single rate value. \citet{gao:shang:2018} used this specific dataset to show that their method outperforms existing FTS methods. We will use the same dataset to show that our method in turn outperforms   \citet{gao:shang:2018}. 

Due to sparse observations at old ages, we have grouped   the observations related to the age groups of 95 or more by averaging the mortality rates and   denote 
  by $x_{it}(\tau)$, $i=1,\ldots,47$; $t=1,\ldots, 42$; $\tau = 0,1,\ldots,95$ 
  the log mortality rate of people aged $\tau$ living in prefecture~$i$ during year~$1974 + t$.  The dataset contains many missing values which we replaced with the rates of the previous age group. We conducted separate analyses for males and females, making our forecasting results  comparable to those in \citet{gao:shang:2018}.

The three $h$-step ahead forecasting methods to be compared are\vspace{-2mm} 
\begin{enumerate}
    \item[(GSY)] The method of \citet{gao:shang:2018},  implemented via the dedicated functions of their~R package \texttt{ftsa} \citep{ftsaArticle}, with the smoothing technique and the parameter values they suggested for this dataset.
    \item[(CF)]   Componentwise forecasting. For each of the $N$ functional time series in the panel, we fitted a functional model independently of the other functional time series. As it is common in functional time series, we obtain the forecast by decomposing every function on its principal components, forecasting every functional principal component score separately and recomposing the functional forecast based on the score forecasts \citep[cf. for instance][]{hyndman:ullah:2007,hyndman2009forecasting}. For this exercise, we have used six functional principal components and ARMA models to predict the scores.
    \item[(TNH)]  Our~$h$-step ahead forecasting method, which we now describe.
\end{enumerate}

The smoothing technique we have used differs from the one considered by \citet{gao:shang:2018} since the latter results in a 96-points representation of the curves while our method is computationally more efficient with data represented in a functional basis. Thus, in order to implement our method, we transform the regularly-spaced data points into curves represented on a 9-dimensional B-splines basis. In our 
 notation, 
  $H_{i} \defeq L
^{2}([0,\, 95], \bbR)$, $i=1,\ldots,47$. Our~$h$-step ahead forecasting algorithm, based on the data $\{ x_{it} : i=1,\ldots, 47;\, t=1,\ldots, T\}$ with~$T \leq 42$, works as follows.

\begin{enumerate}
    \item Compute $\hat r $  using the \textrm{permutation-enhanced} \texttt{ABC} methodology with the same parameters as in Section~\ref{sec:simulation_number_of_factors} and  the  \texttt{IC2a} penalization function. 
    \item For $i = 1,\ldots, N$, compute the $i$th sample mean $\hat \mu_i \defeq \sum_{t=1}^T x_{it} / T$ and the  centered panel~$y_{it} \defeq x_{it} - \hat{\mu}_{i}$.
    \item Compute the estimated factors $\{ \tildebu{\hat r}_t : t=1,\ldots, T\}$ and factor loadings $\tildeB{\hat r}_N$ using Proposition~\ref{prop:computing_loadings_and_factors}, and compute the common component 
 $\hat{\chi}_{it}^{(\hat r )}$  using \eqref{eq:hat_chi_it}; compute the esti\-ma\-tors~$\hat{\xi}^{(\hat r ) }_{it} \defeq y_{it}-\hat{\chi}_{it}^{(\hat r )}$, $i=1,\ldots,N$ and $t=1,\ldots, T$  of the idiosyncratic components.
    \item For each $l \in \{ 1,\ldots, \hat r \}$, select the ARIMA model that best models the $l$-th\linebreak factor~$\{ u_{l,t}^{(\hat r )} : t=1,\ldots, T\}$ according to the BIC criterion; forecast $\tildebu{\hat r}_{T+h}$ based on these fitted models, using past values $\{u_{l,t}^{(\hat r )} : t=1,\ldots, T\}$.
    \item For every $i \in \{1,\ldots, N\}$, pick the ARIMA model that best models the idiosyncratic component $\{ \hat{\xi}^{(\hat r ) }_{it} : t = 1, \ldots, T \}$ according to the BIC criterion, and forecast $\hat{\xi}^{(\hat r )}_{i,T+h}$ based on the fitted model, using past values $ \hat{\xi}^{(\hat r ) }_{it}$ up to~$t=T$.
    \item The final forecast is
     $\hat{x}_{i,T+h} \defeq \hat{\mu}_{i} +  \tilde{b}^{(\hat r  )}_{i} \tildebu{\hat r}_{T+h} + \hat{\xi}^{(\hat r) }_{i,T+h},$
    where $\tilde{b}^{(\hat r  )}_{i} \defeq \singleProj_i \tildeB{\hat r}_N$, \linebreak and $\singleProj_i: \bH_N \to H_i$ maps $(v_1,\ldots, v_N)^\tp \in \bH_N$ to $v_i \in H_i$.
\end{enumerate}

In order to properly assess the forecasting accuracy of each method, we   split the dataset into a training set and a test set. The training set consists of the observations   from 1975 to~$1974 + \Delta$, which translates into $t=1, \ldots, \Delta$. The test set consists of the observations for  the years  $1974 + \Delta+1$ and subsequent  ($t \geq \Delta+1$). For every rolling window $t=1,\ldots, \Delta$, 
we   re-estimate the parameters of our model and those from \citet{gao:shang:2018}. 

To ensure that our forecasts imply positive mortality rates, we have applied the log transform to the mortality curves. Log transforms are  quite classical in that  area of statistical demography   \citep[cf.][]{hyndman:ullah:2007,hyndman2009forecasting,he:huang:yang}. \citet{gao:shang:2018} have also used the log curves for estimating their model but, in the last stage of their forecasting analysis, they compute the exponentials of their  forecasted log curves and base their metrics on it. However, the female mortality rate for $95$+  is~$1645$ times larger than the mortality rate for age $10$ and $345$ times larger than the mortality rate at age 40 (averaged across time and the Japanese regions). Similar results hold for male mortality curves. This fact just illustrates that the scale of the mortality rates differs greatly from one age to another. Hence, in the analysis of \citet{gao:shang:2018}, the contribution of the young ages to the averaged errors is negligible. Assessing the quality of a model's predictions based on a metric that disregards most of the predictions may be misleading.    In accordance with the demographic practice, we thus decided to establish our comparisons on a metric  based on the log curves. For the sake of completeness, we also present the results based on  the methodology of \citet{gao:shang:2018}   in Section~\ref{sec:additional_application} of the \thesupplement.

Comparisons are conducted using the mean absolute forecasting error ($\MAFE$) defined as
\begin{equation*}
\MAFE(h) \defeq \frac{1}{47 \times 89 \times (17-h)} \sum_{i=1}^{47} \sum_{\Delta=26}^{42-h} \sum_{\tau=0}^{89} \Big\vert{ \hat{x}_{i,\Delta+h}(\tau /89)-x_{i,\Delta+h}(\tau /89)}\Big\vert, 
\end{equation*}
and the mean squared forecasting error (MSFE) defined as
\begin{equation*}
    \MSFE(h) \defeq \frac{1}{47 \times 89 \times (17-h)} \sum_{i=1}^{47} \sum_{\Delta=26}^{42-h} \sum_{\tau =0}^{89} \Big [\hat{x}_{i,\Delta+h}(\tau/89) - x_{i,\Delta+h}(\tau/89 ) \Big ]^{2}.
\end{equation*}
Notice that $\Delta=26$ corresponds to the year $2000$, and that $\Delta=42-h$ corresponds to the year $2016-h$.
These forecasting errors are given in Table~\ref{t:mortality_forecasting}.
The results indicate that our   method  markedly outperforms the other methods for all $h$,    for both measures of performance ($\MAFE$ and $\MSFE$), an for both the male and female panels. This is not surprising, since our method does not lose cross-sectional information by separately estimating  factor models for the various principal components of the cross-section as the GSY method does; nor does it ignore the interrelations between the $N$ functional time series as CF does.

\begin{table}[ht]
    \centering
    \scriptsize
    \begin{tabular}{l||rrr|rrr||rrr|rrr|}
            & \multicolumn{6}{c||}{Female} & \multicolumn{6}{c|}{Male} \\
            & \multicolumn{3}{c|}{MAFE} & \multicolumn{3}{c||}{MSFE} & \multicolumn{3}{c|}{MAFE} & \multicolumn{3}{c|}{MSFE}  \\
            & GSY & CF & TNH & GSY & CF& TNH & GSY& CF & TNH & GSY & CF & TNH  \\
            \hline
        $h=1$  & 296 & 286 & \textbf{250} & 190 & 166 & \textbf{143} & 268 & 232 & \textbf{221} & 167 & 124 & \textbf{122} \\ 
        $h=2$  & 295 & 294 & \textbf{252} & 187 & 171 & \textbf{145} & 271 & 243 & \textbf{224} & 171 & 131 & \textbf{124} \\ 
        $h=3$  & 294 & 301 & \textbf{254} & 190 & 176 & \textbf{148} & 270 & 252 & \textbf{227} & 170 & 136 & \textbf{126} \\
        $h=4$  &300 & 305 & \textbf{258} & 195 & 178 & \textbf{152} & 274 & 259 & \textbf{230} & 177 & 141 & \textbf{129} \\ 
        $h=5$  & 295 & 308 & \textbf{259} & 190 & 179 & \textbf{154} & 270 & 268 & \textbf{233} & 169 & 146 & \textbf{131} \\ 
        $h=6$  & 295 & 313 & \textbf{259} & 194 & 181 & \textbf{156} & 271 & 278 & \textbf{235} & 169 & 152 & \textbf{134} \\ 
        $h=7$  & 302 & 321 & \textbf{263} & 200 & 187 & \textbf{161} & 266 & 289 & \textbf{240} & 164 & 160 & \textbf{138} \\ 
        $h=8$  & 298 & 329 & \textbf{269} & 192 & 193 & \textbf{167} & 266 & 302 & \textbf{245} & 161 & 168 & \textbf{142} \\
        $h=9$  & 303 & 339 & \textbf{275} & 203 & 199 & \textbf{172} & 277 & 315 & \textbf{251} & 169 & 178 & \textbf{148} \\
        $h=10$ & 308 & 347 & \textbf{280} & 209 & 205 & \textbf{177} & 283 & 327 & \textbf{254} & 174 & 186 & \textbf{150} \\  \hline
        Mean   & 299 & 314 & \textbf{262} & 195 & 183 & \textbf{157} & 272 & 277 & \textbf{236} & 169 & 152 & \textbf{134} \\  
        Median & 297 & 311 & \textbf{259} & 193 & 180 & \textbf{155} & 271 & 273 & \textbf{234} & 169 & 149 & \textbf{133} \\ 
    \end{tabular}
    \caption{Comparison of the forecasting accuracies of our method (denoted as TNH ) with \cite{gao:shang:2018} (denoted as GSY) and with componentwise FTS forecasting (denoted CF) on the Japanese mortality curves for~$h=1,\ldots, 10$,  female and male curves. Boldface is used for the winning \emph{MAFE} and \emph{MSFE} values, which uniformly correspond to the TNH  forecast. Both metrics have been multiplied by $1000$ for readability.}
    \label{t:mortality_forecasting}
\end{table}

\section{Discussion}
\label{sec:discussion}

This paper proposes a new paradigm for the analysis of high-dimensional  time series with functional (and possibly also scalar) components. The approach is based on a new concept of  \emph{(high-dimensional) functional
factor model} which, in the particular case of purely scalar series reduces to the the well-established
concepts studied by  \cite{stock:watson:2002a, stock:watson:2002b} and \cite{bai:ng:2002}, albeit under weaker assumptions.
We extend to the functional context 
 the classical representation results of
\citet{chamberlain:1983,chamberlain:rothschild:1983} and propose consistent
estimation procedures for  the common
components, the factors, and the factor loadings as both the size $N$ of the
cross-section and the length~$T$ of the  observation period  tend to infinity, with no constraints
on their relative rates of divergence. We also propose a consistent identification method for the number of factors, extending to the functional context the methods of  \cite{bai:ng:2002} and \cite{Alessi2010}.

This is, however, only a first step in the development of a full-fledged toolkit for the analysis of high-dimensional time series, and a long list of important issues remains on the agenda of future research. This includes, but is not limited to, the following points.
\begin{enumerate} 
    \item[(i)] This paper only considers   consistency of the estimates. Further asymptotic analysis should be developed in the direction of consistency rates  and asymptotic distributions, on the model of \cite{bai2003inferential, FHLR:2004, lam:yao:2012} or \cite{FHLZ:2017}. 
    \item[(ii)] The factor model developed  here is an extension of the {\it static} scalar models by \cite{stock:watson:2002a, stock:watson:2002b} and \cite{bai:ng:2002}, where factor loadings are matrices. That factor model is only a particular case of the {\it general} (also called {\it generalized}) {\it dynamic factor model} introduced by \cite{FHLR:2000}, where factors are loaded in a fully dynamic way via filters (see \cite{hallin:lippi:2013} for the advantages that generalized approach). An extension of the general dynamic factor model to the functional setting  is, of course, highly desirable, but requires a more general representation result involving filters with operatorial coefficients and the concept of functional {\it dynamic} principal components \citep{panaretos:tavakoli:2013SPA, hormann:kidzinski:hallin:2015}. 
    \item[(iii)] The problem of identifying the number of factors is only briefly considered here, and requires further attention. In particular,  functional versions of the eigenvalue-based methods by \cite{onatski:2009:numberOfDynamicFactors, 
        Onatski2010} or the hypothesis testing-based ones by \citet{pan2008modelling, Kapetanios2010, lam:yao:2012, AHNH13} deserve being developed and compared to the method we are proposing here.
    \item[(iv)] Another crucial issue is the analysis of volatilities. In the scalar case, the challenge is in the estimation and forecasting of high-dimensional conditional covariance matrices and such quantities as values-at-risk or expected shortfalls: see \cite{fan:2008:highdimcov, fan2011, fan:2013:jrssb-discussion, barigozzi:hallin:2016generalized, barigozzi2017generalized, hallin2019forecasting}. In the functional context,  these matrices, moreover, are replaced with high-dimensional operators.
\end{enumerate}

\paragraph{Acknowledgements} We would like to thank Yoav Zemel, Gilles Blanchard and Yuan Liao for helpful technical discussions, as well as Yuan Gao and Hanlin Shang for details of the implementation of the forecasting methods described in \citet{gao:shang:2018}. Finally, we thank the referees and the associate editor for comments leading to an improved version of the paper.

\paragraph{Code}
    The R code  reproducing the numerical experiments of Section~\ref{s:simu},  as well as
        code used for Section~\ref{sec:forecasting_mortality_curves}, can be obtained by contacting
    the authors by email.

\bibliographystyle{agsm}
\citestyle{agsm}
\bibliography{journalsAbbr,factor-models}

\appendix

\section{Orthogonal Projections}%
\label{sec:orthogonal_projections}

Let $H$ be a separable Hilbert space. Denote by~$\LLPH{H}$  the space of~$H$-valued random elements $X: \Omega \rightarrow H$
with $\ee\! X \!= 0$ and~$\ee\! \hnorm{X}_H^2\! <~\! \infty$. For any finite-dimensional subspace~$\mathcal U \subset \LLP$, where $\LLP$ is the set of random variables with mean zero and finite variance, let 
%
\begin{equation}
    \label{eq:definition_spanH}
    \vspan_H(\mathcal U) \defeq \left\{ \sum_{j=1}^m u_j b_j : b_j \in H,\,  u_j \in \mathcal U,\, m
    =1,2, \ldots \right\}.
\end{equation}
Since $\mathcal U$ is finite-dimensional,   $\vspan_H(\mathcal U) \subseteq \LLPH{H}$
is a closed subspace. Indeed, denoting by~$u_1,\ldots, u_r \in~\!{\mathcal U}$  ($r <
\infty$) an orthonormal basis, 
\[
    \ee \hnorm{u_1 b_1 + \cdots + u_r b_r}_H^2 = \hnorm{b_1}_H^2 + \cdots + \hnorm{b_r}_H^2.
\]
By the orthogonal decomposition Theorem \cite[Theorem~2.5.2]{hsing:eubank:2015}, for any~$X \in \LLPH{H}$, there exists a unique $U[X] \in \vspan_H(\mathcal U)$ such that
\begin{equation} \label{eq:orthogonal_decomposition}
    X = U[X] + V[X],
\end{equation}
where $V[X] = X - U[X] \in \vspan_H(\mathcal U)^\perp$; hence $\eee{uV[X]} = 0$ for all $u
\in \mathcal U$ and 
\begin{equation}
    \label{eq:orthogonal_decomposition_pythagoras}
    \ee \hnorm{X}_H^2 = \ee \hnorm{U[X]}_H^2 + \ee \hnorm{V[X]}_H^2.
\end{equation}
Consider the following definition.
\begin{defn}
    \label{defn:orthogonal_projection}
The decomposition \eqref{eq:orthogonal_decomposition} of  $X \in \LLPH{H}$ into $U[X] $ and $ V[X]$ is called the \emph{orthogonal
decomposition} of $X$ onto $\vspan_H(\mathcal U)$ and its orthogonal
comple\-ment;~$U[X] = \vcentcolon \proj_H(X | \mathcal U)$ is the \emph{orthogonal projection} of $X$ onto $\vspan_H(\mathcal U)$.
\end{defn}

If $u_1,\ldots, u_r \in \LLP$ form a basis of $\mathcal U$, then $\proj_H(X |
\mathcal U) =\sum_{l=1}^r u_l b_l$ for some unique~$b_1, \ldots, b_r \in H$; if 
they form an  orthonormal basis, then~$b_l = \eee{u_l X }$,~$l=1,\ldots, r$. We shall also
use the notation
\[
    \proj_H(X | \bu) \defeq \proj_H(X | u_1,\ldots, u_r) \defeq \proj_H(X | \mathcal
    U),
\]
and 
\[
    \vspan_H(\bu) \defeq \vspan_H(\mathcal U),
\]
where $\mathcal U  = \vspan(u_1,\ldots, u_r)$ and $\bu = (u_1,\ldots, u_r)^\tp \in \bbR^r$.
Notice that if $\eee{ \bu \bu^\tp } = \b I_r $ (with~$\b I_r$ the $r \times r$ identity matrix), then
\begin{equation}
    \label{eq:tensorform_orthogonal_projection}
    \proj_H(X | \bu ) = \eee{X \tensor \bu} \bu.
\end{equation}

\section{Advantage of Proposed Approach}%
\label{sec:advantage_current_approach}

This section provides examples  where the approach proposed by  \citet{gao:shang:2018}, contrary to ours, runs into serious problems.  
 In \citet{gao:shang:2018}, each series $\{ x_{it} : t \in \bbZ \}$ is first projected onto the first $p_0$ eigenfunctions of the long-run covariance operator,
\[
    C^{(x_i)} \defeq \sum_{t \in \bbZ} \eee{x_{it} \tensor x_{i0}}.
\]
Consider the high-dimensional functional factor model with $r \in \bbN$ factors 
\[
    x_{it} = b_{i1} u_{1t} + \cdots + b_{ir} u_{rt} +  \xi_{it}, \quad i=1, 2, \ldots, \  t \in \bbZ,
\]
where   $b_{i1}, \ldots, b_{ir} \in H_i$ are deterministic, $i \geq 1$. If we assume that $\eee{ \bu_t \tensor \bxi_s } = {\bf 0}$ for all~$s,t \in \bbZ$,  the long-run covariance operator of $\{ x_{it} : t \in \bbZ \}$ is
\begin{align*}
    C^{(x_i)} =  \sum_{t \in \bbZ}  \eee{x_{it} \tensor x_{i0}}
                              & =  \sum_{t \in \bbZ}  \left( \eee{ (\b b_i \bu_t) \tensor (\b b_i \bu_0)} +  \eee{\xi_{it} \tensor \xi_{i0}} \right) 
                             \\ & = \b b_i C^{(u)} \b b_i^\adjoint +  C^{(\xi_{i})},
\end{align*}
where $\b b_i \in \bounded{\bbR^r, H_i}$ maps $(u_1,\ldots, u_r)^\tp \in \bbR^r$ to $u_1 b_{i1} + \cdots + u_r b_{ir} \in H_i$,  
\[
    {\bf C}^{(u)} \defeq  \sum_{t \in \bbZ}  \eee{\bu_t \tensor \bu_0},\quad\text{
and}\quad 
    C^{(\xi_i)} \defeq  \sum_{t \in \bbZ}  \eee{\xi_{it} \tensor \xi_{i0}}. 
\]
Let 
\[
    {\bf C}^{(u)} = \sum_{j=1}^r \lambda_j^{(u)} {\bf v}_j^{(u)} \tensor {\bf v}_j^{(u)} \quad\text{and}\quad
    C^{(\xi_i)} = \sum_{j=1}^\infty \lambda_j^{(\xi_i)} \vfi_j^{(\xi_i)} \tensor \vfi_j^{(\xi_i)} 
\]
denote the eigen-decompositions  of ${\bf C}^{(u)}$ and $C^{(\xi_i)}$, respectively. Then,
\begin{equation}
    \label{eq:specdensity_xi_decomposed}
    C^{(x_i)} = \sum_{j=1}^r \lambda_j^{(u)} {\bf v}_j^{(u)} \tensor {\bf v}_j^{(u)} +   \sum_{j=1}^\infty \lambda_j^{(\xi_i)} \vfi_j^{(\xi_i)} \tensor \vfi_j^{(\xi_i)} .
\end{equation}
If $\vfi_j^{(\xi_i)}$ is orthogonal to $b_{i1},\ldots, b_{ir} \in H_i$, $j \geq 1$, we can choose $\lambda_j^{(\xi_i)}, j = 1,\ldots, J$, large enough  that
\begin{enumerate}
    \item[(i)] the first $J$ eigenvectors of $C^{(x_i)}$ are $\vfi_j^{(\xi_i)}$, with associated eigen\-values~$\lambda_j^{(\xi_i)}$, $j=1,\ldots J$;
    \item[(ii)] the proportion  $\sum_{j=1}^J \lambda_j^{(\xi_i)} / \trace( C^{(x_i)} )$ of variance explained by the first $J$ principal components of $\{ x_{it} : t \in \bbZ \}$  is arbitrarily close to $1$.
\end{enumerate}

In such a case, the preliminary FPCA approximation of each series $\{ x_{it} : t \in \bbZ \}$ in \citet{gao:shang:2018} (at population level, based on the true long-run covariance of $C^{(x_i)}$) would completely discard the   common component. Since we do not perform such preliminary FPCA approximation,  the common component would not get lost with our method (see results for \texttt{DGP4}, Figure~\ref{fig:DGP4}, for the empirical performance of our method in a similar situation).
A similar point can be made if one performs, as a preliminary step, the usual FPCA projection of each series   based on the lag-$0$ covariance operators $\eee{x_{i0} \tensor x_{i0}}$ (as opposed to the long-run covariance), or even a more sophisticated dynamic FPCA projection \citep{panaretos:tavakoli:2013SPA,hormann:kidzinski:hallin:2015}.

\section{Uniform error bounds on loadings and the common component}
\label{sec:uniform_bounds_loadings_commonComponent}

In this Section, we provide uniform error bounds for the loadings and the common component. All the results in this Section are proved in Section~\ref{sec:proofs_uniform_bounds}.

We first turn our attention to the loadings.
Let $\bb_i \in \bounded{\bbR^r, H_i}$ be the linear operator mapping
$(\alpha_1,\ldots, \alpha_r)^\tp$ to $\sum_{l=1}^r \alpha_l b_{il}$. The estimator of $\bb_i$
is $\tildeb{r}_i\! = \singleProj_i \tildeB{r}\!$, where~$\singleProj_i:~\!\bH_N \to~\!H_i$ is the canonical projection, mapping $(v_1,\ldots, v_N)^\tp \in \bH_N$ to
$v_i \in H_i$, $i = 1,\ldots, N$. With this notation, notice that $\bb_i = \singleProj_i \bB_N$. To get uniform error bounds for the loadings, we need some technical assumptions, which rely on the concept of $\tau$-mixing:
for a sequence of random elements $\{Y_t\} \subset H_0$ such that  $\hnorm{Y_t}_{H_0} \leq c$ almost surely, where $H_0$ is a real separable Hilbert space with norm $\hnorm{\cdot}_{H_0}$, we define the $\tau$-mixing coefficients \citep[][Example~2.3]{blanchard2019concentration} 
\begin{align}
    \label{eq:tau_mixing}
    \MoveEqLeft
    \tau(h) \defeq \sup\Big\{ \eee{Z ( \vfi(Y_{t+h}) - \ee \vfi(Y_{t+h}) )} ~\Big|~ t \geq 1, Z \in L^1(\Omega, \mathcal{M}^Y_t, \pp), 
    \\ & \hspace{5cm} \ee \rpnorm{Z} \leq 1,  \vfi \in \LipschitzSpace(H_0, c) \Big\}, \nonumber
\end{align}
where $\mathcal{M}^Y_t$ is the $\sigma$-algebra generated by $Y_1,\ldots, Y_t$ and, denoting by~$B_c(H_0) \subset H_0$  the ball of radius $c > 0$ centered at zero, 
\[
    \LipschitzSpace(H_0, c) \defeq \Big\{ \vfi: B_c(H_0) \to \bbR \text{ such that } \rpnorm{\vfi(x) - \vfi(y)} \leq \hnorm{x - y}_{H_0}, \forall x,y \in H_0 \Big\}.
\]
 Notice that by the Kirszbraun Theorem \citep[][Theorem~1.31]{schwartz1965nonlinear}, the definition of~$\tau(h)$ is independent of the choice of upper bound $c > 0$. 
 
For obtaining uniform bounds on the loadings,  we make the following assumptions. 
\begin{customass}{H($\gamma$)} \label{assumption:uniform_bound_loadings} \mbox{}

    \begin{description}[labelindent=1cm]
        \item[(H1)] For all~$l=~\!1,\ldots, r$ and all $i,t \geq 1$,  $\eee{u_{lt} \xi_{it}} = 0,$
        $$ \abs{u_{lt}} < c_u < \infty \quad \text{and}\quad\hnorm{ \xi_{it}}_{H_i} \leq c_\xi < \infty\quad\text{almost surely. }$$
        \item[(H2)] There exists $\alpha, \theta, \gamma > 0$ such that, 
             for all $l=1,\ldots, r$ and all $i \geq~\!1$, the process 
             $$\{ (u_{lt}, \xi_{it}) : t = 1, 2, \ldots \} \subset \bbR \oplus H_i$$
              is $\tau$-mixing with coefficient  $\tau(h) \leq~\!\alpha \exp(- (\theta h)^{\gamma})$.
     \end{description}
\end{customass}
Assumption~\ref{assumption:uniform_bound_loadings} is needed in order to apply  Bernstein-type concentration inequalities for $T^{-1} \sum_{s=1}^T u_{ls} \xi_{is}$ \citep{blanchard2019concentration}, and could be replaced by any other technical condition that yields such concentration bounds. Notice that we do not assume joint $\tau$-mixing on the cross-section of idiosyncratics and factors. Assumption (H2) is implied by the stronger (joint $\tau$-mixing) assumption that 
\[
    \Big\{ (u_{1t}, \ldots, u_{rt}, \xi_{1t}, \xi_{2t}, \ldots) : t \geq 1 \Big\} \subset \bbR^r \oplus \bigoplus_{i \geq 1} H_i
\]
is exponentially $\tau$-mixing. Such a joint assumption on all factors and idiosyncratics is made in \citet{fan:2013:jrssb-discussion}, but with $\alpha$-mixing instead of $\tau$-mixing. 
Note that $\alpha$-mixing coefficients are not comparable to $\tau$-mixing coefficients. However, $\tau$-mixing  holds for some simple AR$(1)$ models that are not $\alpha$-mixing. Examples of $\tau$-mixing processes can be found in \citet{blanchard2019concentration}.
Other types of mixing coefficients equivalent to $\tau$-mixing when the random elements are almost surely bounded are the $\theta_\infty$-mixing coefficients \citep[][Definition~2.3]{dedecker2007weak} or the $\tau_\infty$-mixing coefficients \citep[][p.492]{wintenberger2010deviation}.

We can now state a uniform error bound result for the estimators of $\bb_i, i=1,\ldots, N$.
\begin{Theorem}
    \label{thm:unif_bound_loadings}
    Let Assumptions~\ref{assumption:a}, \ref{assumption:b}, \ref{assumption:c}, \ref{assumption:Bn_xi} and~\ref{assumption:uniform_bound_loadings} hold.
        Then,
    \[
        \max_{i=1,\ldots, N} \Fnorm{\tildeb{r}_i - \bb_i \tilde \bR^{-1}} = O_{\rm P}\left( \max\left\{ \frac{1}{\sqrt{N}}, \frac{\log(N) \log(T)^{1/{2\gamma}} }{\sqrt{T}} \right\} \right).
    \]
\end{Theorem}
We see that if $N,T \to \infty$ with $\log(N) = o(\sqrt{T}/\log(T)^{1/{2\gamma}})$, the estimator of the loadings is uniformly consistent. 
We notice that in the special case $\bH_N = \bbR^N$, \citet{fan:2013:jrssb-discussion} do not impose boundedness assumptions on the factors and idiosyncratics, but rather impose exponential tail bounds on them. They impose exponential $\alpha$-mixing conditions on the factors and idiosyncratics (jointly), whereas we impose exponential $\tau$-mixing on the product of factors and idiosyncratics.  Their rate of convergence is similar, except for the $\log(N) \log(T)^{1/{2\gamma}}$ factor, which appears as~$\sqrt{\log(N)}$ in their case. The reason for our different rates of convergence is that we rely on concentration inequalities for weakly dependent time series in Hilbert spaces \citep{blanchard2019concentration}, which are weaker and require stronger assumptions than their scalar counterparts \citep{merlevede2011bernstein}. Having said that, whether these rates are optimal or whether the assumptions here are the weakest possible ones remain open questions.

The next result gives uniform bounds for the estimation of the common component. Recall the definition of $\hat \chi_{it}^{(r)}$ in \eqref{eq:hat_chi_it}.
\begin{Theorem}
    \label{thm:uniform_bound_common_component}
    Under the assumptions of Theorems~\ref{thm:unif_bound_factors} and~\ref{thm:unif_bound_loadings},         \begin{align*}
            \MoveEqLeft
            \max_{t=1,\ldots T} \max_{i=1,\ldots, N} \hnorm{ \hat \chi^{(r)}_{it} - \chi_{it} }_{H_i} = O_{\rm P}\Big( \max\Big\{ \frac{T^{1/(2 \kappa)}}{\sqrt{N} }, \frac{\log(N)\log(T)^{1/{2\gamma}}}{ \sqrt{N}\, T^{( \kappa-1)/(2 \kappa)} } ,
            \\ & \hspace{7cm} \frac{\log(N)\log(T)^{1/{2\gamma}}}{\sqrt{T}} \Big\} \Big).
        \end{align*}
\end{Theorem}
If $N,T \to \infty$ with $T = o(N^\kappa)$ and $\log(N) = o(\sqrt{T}/\log(T)^{1/{2\gamma}})$, the estimator of the common component is uniformly consistent, and the rate simplifies to
\[
O_{\rm P}( \max\{ T^{1/(2 \kappa)} N^{-1/2} ,\  \log(N)\log(T)^{1/{2\gamma}} T^{-1/2} \} ),
\] which yields some similarities to the rate in \citet[][Corollary~1]{fan:2013:jrssb-discussion}, though slightly different from it due to  differences between our regularity assumptions.

\section{Proofs}

Let 
$v_i \in H_i$: we write, with a slight abuse of notation,~$v_i \in~\!\bounded{\bbR, H_i}$
for the mapping $v_i: \alpha \mapsto v_i \alpha \defeq \alpha v_i$ from $\bbR$ to $H_i$. 
Applying this notation for an ele- \linebreak ment~$ \b v = (v_1, \ldots, v_N)^\tp$ of~$\bH_N$, 
we have that $\b v \in \bounded{\bbR, \bH_N}$ is the mapping~$ a \mapsto a \b v$ from~$\bbR $ to~$\bH_N$;  in particular, $ \bv a = (a v_1,\ldots, a v_N)^\tp$.  
Finally, let $\rpnorm{\cdot}$ denote the Euclidean norm in~$\bbR^p, p \geq 1$.

\subsection{Proofs of   Theorems~\ref{thm:equiv_factormodel_eigenvalues} and~\ref{thm:factormodel_uniqueness}} \label{s:proof_rep}

All proofs in this section are for some arbitrary but fixed $t$, with 
 $N \rightarrow
\infty$. Whenever possible, we therefore omit the index~$t$, and write~$\b X_N$ for $\bX_{N,t} \defeq (X_{1t}, \ldots, X_{Nt})^\tp$, etc. 
Denote by $\b \Sigma_N = \eee{\bX_N \tensor \bX_N}$ the covariance operator of $\bX_N$ which, in view of stationarity, 
 does not depend on $t$. 
Denoting by $\b p_{N,i} \in~\!\bH_N$ the~$i$th eigenvector of $\b
\Sigma_N$ and by $\lambda_{N,i}$ the corresponding eigenvalue, we have the   eigendecomposition 
\[
  \b \Sigma_N= \sum_{i=1}^{\infty} \lambda_{N,i} \b p_{N,i} \tensor \b p_{N,i}.
\] 

The eigenvectors $\{\b p_{N,i} : i \geq 1\} \subset \bH_N$ form an orthonormal basis of  the\linebreak im\-age~$\image(\b \Sigma_N) \subseteq~\bH_N$ of $\b \Sigma_N$. We can extend this 
set of eigenvectors of $\b \Sigma_N$ to form an orthonormal basis of $\bH_N$. With a slight abuse of 
 notation, we shall denote this basis by~$\{ \b p_{N,i} : i \geq 1 \}$,  possibly
reordering the eigenvalues: we might no longer have non-increasing and positive eigenvalues, but still can 
enforce $\lambda_{N,1} \geq \cdots \geq \lambda_{N, r+1} \geq \lambda_{N, r+1+j}$  for all $j \geq 1$.
Define  
\[
    \b P_N = \sum_{k=1}^r \bp_{N,k} \tensor \bz_k \in \bounded{\bbR^r, \bH_N},
\]
where $\bz_k$ is the~$k$-th vector in the canonical basis  of $\bbR^r$. 
Let 
\[
    \ell_2\defeq \left\{\b \alpha= (\alpha_1, \alpha_2, \ldots) : \alpha_i \in \bbR, \sum_i \alpha_i^2 < +\infty \right\},
\]
and let $\b \alpha_k \in \ell_2$ be the $k$-th canonical basis vector, with $(\b \alpha_k)_k=1$  and $(\b \alpha_k)_i=0$ for~$i\neq k$. Define
\[
    \b Q_N \defeq \sum_{k=1}^\infty \bp_{N, r+k} \tensor \b \alpha_k \in \bounded{\ell_2, \bH_N}.
\]
Denote by $\b \Lambda_N \in \bounded{\bbR^r}$ 
the diagonal matrix with diagonal elements $\lambda_{N,1}, \ldots, \lambda_{N,r}$, that is, 
\[
    \b \Lambda_N \defeq \sum_{k=1}^r \lambda_{N,k} \bz_k \tensor \bz_k,
\]
and define $\b \Phi_N \in \bounded{\ell_2}$ as
\[
    \b \Phi_N \defeq \sum_{k = 1}^\infty \lambda_{N, r+k} \b \alpha_k \tensor \b \alpha_k.
\]
With this notation, we have
\begin{equation}
 \label{eq:eigenvector-decompositions}
 \b \Sigma_N = \b P_N \b \Lambda_N \b P_N^\adjoint + \b Q_N \b \Phi_N \b Q_N^\adjoint \in \bounded{\bH_N} \quad\text{and}\quad
 \b P_N \b P_N^\adjoint + \b Q_N \b Q_N^\adjoint =   \identity_N ,
\end{equation}
where $\identity_N$ is the identity operator on $\bH_N$. 
Let 
\begin{equation}
    \label{eq:defn-of-psi^N}
    \b \psi_N \defeq \b \Lambda_N^{-1/2} \b P_N^\adjoint \b X_N = (\lambda_{N,1}^{-1/2} \sc{\b p_{N,1}, \b X_N }, \ldots, \lambda_{N,r}^{-1/2} \sc{\b p_{N,r}, \b X_N})^\tp \in \bbR^r
\end{equation}
(note
that, by Lemma~\ref{lma:eigenvalue_inequalities}, $\b \psi_N$ is well defined for $N$ large
enough since~$\lambda_{N,r} \to \infty$ as ${N \to \infty}$). Using \eqref{eq:eigenvector-decompositions}, we get
\begin{equation}
  \label{eq:x_N-orthogonal-decomposition}
  \b X_N = \b P_N \b P_N^\adjoint \b X_N + \b Q_N \b Q_N^\adjoint \b X_N = \b P_N \b
  \Lambda_N^{1/2} \b \psi_N + \b Q_N \b Q_N^\adjoint \b X_N,
\end{equation}
where the two summands are uncorrelated since~$\b P_N^\adjoint \b \Sigma_N \b Q_N = 0 \in
\bounded{\ell_2, \bbR^r}$, the zero operator: \eqref{eq:x_N-orthogonal-decomposition}
 in fact is the orthogonal decomposition of $\b X_N$ into $\vspan_{\bH_N}(\b \psi_N)$ and its orthogonal complement, as defined in Appendix~\ref{sec:orthogonal_projections}. 
Let $\canonicalProj_{M,N}$ be the canonical projection of  $\bH_N$ onto $\bH_M$, $M < N$, namely, 
\begin{equation}
    \label{eq:canonical_projection}
    \canonicalProj_{M,N} \begin{pmatrix}
        v_1 \\ \vdots \\ v_M \\ v_{M+1} \\ \vdots \\ v_N
    \end{pmatrix}
    = \begin{pmatrix}
       v_1 \\ \vdots \\ v_M 
    \end{pmatrix}.
\end{equation}
Let $\mathcal O(r) \subset \bbR^{r \times r}$ denote the set of all $r \times r$ orthogonal matrices, i.e. the $r \times r$ matrices~$\b C$ such
that $\b C \b C^\adjoint = \b C^\adjoint \b C =\b I_r \in \bbR^{r \times r}$. 
 For $\b C \in
\mathcal O(r)$, left-multiplying both sides of \eqref{eq:x_N-orthogonal-decomposition}   by $\b C \b \Lambda_M^{-1/2} \b P_M^\adjoint \canonicalProj_{M,N}$, with $M < N$, yields
\begin{align}
  \b C \b \psi_M & = \b C \b \Lambda_M^{-1/2} \b P_M^\adjoint \canonicalProj_{M,N} \b P_N \b \Lambda_N^{1/2} \b
  \psi_N + \b C \b \Lambda_M^{-1/2} \b P_M^\adjoint \canonicalProj_{M,N} \b Q_N \b Q_N^\adjoint \b X_N  
  \nonumber
  \\ &= \b D \b \psi_N + \b R \b X_N \in \bbR^r,
  \label{eq:C-psi_m-orthogonal-decomposition}
\end{align}
which is the orthogonal decomposition of $\b C \b \psi_M$ onto the span of $\b
\psi_N$ and its orthogonal complement, since the two
summands in the right-hand side of  \eqref{eq:C-psi_m-orthogonal-decomposition} are uncorrelated. Notice that~$\b D  \in \bbR^{r \times r}$ and  $\b R
 \in \bounded{\bH_N, \bbR^r}$ both depend on $\b C, M$, and $N$, which we emphasize by writing~$\b D = \b D[\b C, M, N] $ and  $\b R
= \b R[\b C, M, N] $.

The following Lemma gives a bound on the residual term $\b R[ \b C, M, N] \b X_N$ in
\eqref{eq:C-psi_m-orthogonal-decomposition}. 
\begin{lma}
  \label{lma:largest-eigenvalue-of-residual}
  The largest eigenvalue of the covariance matrix of~$\b R[ \b C, M, N] \b X_N$  is bounded by $\lambda_{N, r+1}/\lambda_{M, r}$.
  \begin{proof}
    Write   $\b A \succeq \b B$
if $\b A - \b B$ is non-negative definite.
 We know that 
 $$\identity_N \succeq \b Q_N \b Q_N^\adjoint\quad\text{ and}\quad\lambda_{N, r + 1} \b Q_N \b
    Q_N^\adjoint \succeq \b Q_N \b \Phi_N \b Q_N^\adjoint.$$
     Hence, $\lambda_{N, r+1} \identity_N
    \succeq \b Q_N \b \Phi_N \b Q_N^\adjoint$. Multiplying to the left by $\b C \b
    \Lambda_M^{-1/2} \b P_M^\adjoint  \canonicalProj_{M,N}$ and to the right by its adjoint, and using the
    fact that $\b \Phi_N = \b Q_N^\adjoint \b \Sigma_N \b Q_N$, we get 
    \begin{align*}
        \lambda_{N, r+1} \b C \b \Lambda_M^{-1} \b C^\adjoint 
        &\succeq \b C \b \Lambda_M^{-1/2} \b P_M^\adjoint \canonicalProj_{M,N} \b Q_N \b Q_N^\adjoint \b \Sigma_N \b Q_N \b Q_N^\adjoint \b P_M \canonicalProj_{M,N}^\adjoint \b \Lambda_M^{-1/2} \b C^\adjoint 
     \\ &= \b R \b \Sigma_N \b R^\adjoint.
    \end{align*}
  Lemma~\ref{lma:eigenvalue_inequalities} completes the proof since the largest eigenvalue on the left-hand side is~$\lambda_{N,r+1}/ \lambda_{M, r}$.
  \end{proof}
\end{lma}

Let us now take  covariances on both sides of
\eqref{eq:C-psi_m-orthogonal-decomposition}. This yields 
\[
 \b I_r = \b D \b D^\adjoint + \b R \b \Sigma_N \b R^\adjoint.
\]
Denoting by $\delta_i$ the $i$th largest eigenvalue of $\b D \b D^\adjoint$,
we have, by Lemmas~\ref{lma:largest-eigenvalue-of-residual} and~\ref{lma:eigenvalue_inequalities},  
\begin{equation}
  \label{eq:bound-on-delta_i}
  1 - \frac{\lambda_{N, r+1}}{\lambda_{M, r}} \leq \delta_i \leq 1.
\end{equation}
Thus, for $N > M \geq M_0$, all eigenvalues $\delta_i$ 
are strictly positive and, since $\lambda_{r+1} < \infty$ and~$\lambda_r = \infty$,
$\delta_i$ can be made arbitrarily close to $1$ by choosing $M_0$ large enough.
Denoting by $\bW \b \Delta^{1/2} \b V^\adjoint$ the singular decomposition of $\b D$,
where $\b \Delta$ is the diagonal matrix of $\b D \b D^\adjoint$'s eigenvalues~$\delta_1 \geq \delta_2 \geq
\cdots \geq \delta_r$,  define 
\begin{equation}
  \label{eq:defn-of-F}
  \b F  \defeq \b F[\b D] = \b F[\b C, M, N] = \bW \b V^\adjoint, \qquad \b D = \bW \b
  \Delta^{1/2} \b V^\adjoint.
\end{equation}
Note that \eqref{eq:bound-on-delta_i} implies that $\b F$ is well defined for $M,N$ large enough, and that~$\b F \!\in~\!\!\mathcal O(r)$.
The following lemma shows that $\b C \b \psi_M$ is well approximated by $\b F \b
\psi_N$.
\begin{lma}
  \label{lma:largest-eigenvalue-of-residual-with-F}
  For every $\vep > 0$, there exists an $M_\vep$ such that, for all $N > M \geq
  M_\vep$, $\b F = \b F[\b C, M, N]$ is well defined and the largest eigenvalue of the covariance of 
  \[ 
  \b C \b \psi_M - \b F \b \psi_N = \b C \b \psi_M - \b F[\b C, M, N] \b \psi_N
  \]
   is smaller than $\vep$ for all $N > M \geq M_\vep$.
  \begin{proof}
    First notice that it suffices to take $M_\vep > M_0$ for $\b F$ to be
    well defined. We have 
    \[
      \b C \b \psi_M  - \b F \b \psi_M = \b R \b X_N + (\b D - \b F) \b \psi_N,
    \]
    and, since the two terms on the right-hand side are uncorrelated, the covariance
    of their sum is the sum of their covariances. Denoting by $\b S$ the covariance operator of
    the left-hand side and by $\opnorm{\b S}$  the
     operator norm of $\b S$, and
    noting that 
    $$\b D - \b F = \bW (\b \Delta^{1/2} - \b I_r) \b V^\adjoint,$$ we get
    \begin{align}
      \opnorm{\b S} &\leq \opnorm{\b R \b \Sigma_N \b R^\adjoint} + \opnorm{\bW (\b
      \Delta^{1/2} - \b I_r) \b V^\adjoint \b V (\b \Delta^{1/2} - \b I_r) \bW^\adjoint }
      \nonumber
      \\ &= \opnorm{\b R \b \Sigma_N \b R^\adjoint} + \opnorm{\b \Delta^{1/2} - \b I_r}^2,
      \label{eq:bound-on-opnorm-S}
    \end{align}
    since $\bW$ and $\b V$ are unitary matrices. The first summand of
    \eqref{eq:bound-on-opnorm-S} can be made smaller than~$\vep/2$ for $M$ large
    enough (by Lemma~\ref{lma:largest-eigenvalue-of-residual}), and the second summand
    can be made smaller than $\vep/2$ in view of \eqref{eq:bound-on-delta_i}.
  \end{proof}
\end{lma}
A careful inspection of the proofs of these results shows that they hold for all
values of~$t$, i.e., writing $\b \psi_{N,t} = \b \Lambda_N^{-1/2} \b P_N^\adjoint \b
X_{N,t}$, the result of Lemma~\ref{lma:largest-eigenvalue-of-residual-with-F} holds
for the diffe\-rence~$\b D \b \psi_{M,t} - \b F \b \psi_{N,t}$,  with a value of~$M_\vep$ that
does not depend on $t$, by stationarity.  Recall $\mathcal D_t$, defined in
Theorem~\ref{thm:factormodel_uniqueness}. The following results provides a
constructive proof of the existence of the process $\b u_t$.
\begin{prop}
\label{prop:construction-of-basis-of-D}
  There exists an $r$-dimensional second-order stationary process~$\b u_t =
  (u_{1t}, \ldots, u_{rt})^\tp \in \bbR^r$ such that
  \begin{enumerate}
      \item[(i)] $u_{lt} \in \mathcal D_t$ for $l = 1,\ldots, r$ and all $t \in \bbZ$, 
    \item[(ii)] $\eee{ \b u_t \b u_t^\tp} = \b I_r$, $\b u_t$ is second-order  stationary, and $\b u_t$ and $\bX_{N,t}$ are second-order  co-stationary.
  \end{enumerate}
  \begin{proof}
    Recall that $M_\vep$ is defined in
    Lemma~\ref{lma:largest-eigenvalue-of-residual-with-F}. The idea of the proof is
    that~$\b \psi_{M,t}$ is   converging after suitable rotation.
    \begin{description}
      \item[Step 1:] Let $s_1 \defeq M_{1/2^2}$,  let  $\b F_1 \defeq{ \b I_r}$, and let 
              $\b u^{1,t} \defeq \b F_1 \b \psi_{s_1,t}$.
      \item[Step 2:] Let $s_2 \defeq \max\left\{ s_1, M_{1/2^4} \right\}$, let $\b F_2
        \defeq \b F[\b F_1, s_1, s_2]$, and let $ \b u^{2,t} \defeq \b F_2 \b \psi_{s_2, t}$.

        $\quad \vdots$

      \item[Step $q+1$:] Let $s_{q+1} \defeq \max \left\{ s_q, M_{1/s^{2(q+1)} } \right\}$,
        let
         $\b F_{q+1} \defeq \b F[\b F_q, s_q, s_{q+1}]$, and\linebreak  let~$\b u^{q+1, t} \defeq \b
        F_{q+1} \b \psi_{s_{q+1},t}$.

        $\quad \vdots$

    \end{description}
    Denoting by $u_l^{q,t}$ the $l$th coordinate
    of $\b u^{q,t}$, we have
    \[
      \ee \rpnorm{u_l^{q, t} - u_l^{q+1,t} }^2  = \ee \rpnorm{ \b F_{q} \b \psi_{s_q, t} - \b F[ \b
      F_q, s_q, s_{q+1}] \b \psi_{s_{q+1}, t} }^2 \leq \frac{1}{2^{2q}},
    \]
    and thus 
    \[
        \sqrt{\ee \rpnorm{ u_l^{q, t} - u_l^{q+h, t} }^2 } \leq \sum_{j=1}^{h} \sqrt{\ee \rpnorm{ u_l^{q+j-1, t}
            - u_l^{q+j,t} }^2 } \leq \frac{1}{2^{q-1}}.
    \]
    Therefore $(u_l^{q,t})_{q \geq 1}$ is a Cauchy sequence and consequently converges in
    $\LLP$ to some limit~$u_{lt}$, $l=1,\ldots,r$. 
    Note that $\eee{ \b u^{q,t} (\b u^{q,t})^\tp } = \b I_r$ for each $q$ since~$\b
    F_q \in~\! \mathcal O(r)$,  so that
    \[
        \eee{ \b u_t \b u_t^\tp } = \lim_{q \rightarrow \infty} \eee{ \b u^{q,t} (\b
        u^{q,t})^\tp } =  \b I_r.
    \]
    Furthermore, $\eee{ \b u_t \b u_{t+h}^\tp }$ is well defined (and finite) for every
    $h \in \bbZ$, and 
    \[
      \ee \b u_t \b u_{t+h}^\tp = \lim_{q \rightarrow \infty} \ee \b u^{q,t} (\b
      u^{q,t+h})^\tp = \lim_{q \rightarrow \infty} \b F_q \eee{ \b  \psi_{s_q, t} (\b \psi_{s_q,
      t + h})^\tp } \b F_q^\tp. 
    \]
    The term inside the limit is independent of $t$ (since $\bX_{N,t}$ is
    second-order stationary), and hence $\ee \b u_t \b u_{t+h}^\tp$ does not depend on
    $t$, and $(\b u_t)_{t \in \bbZ}$ thus is second-order stationary. Furthermore,
    \[
        \eee{ x_{it} \tensor \b u_{t+s} } = \lim_{q \to \infty} \eee{ x_{it} \tensor \b u^{q, t+s} }
        = \lim_{q \rightarrow \infty }  \eee{ x_{it} \tensor \bX_{s_q, t+s}} \b P_{s_q}^\adjoint \b \Lambda_{s_q}^{-1} \b F_q^\adjoint,
    \]
    and since the term inside the limit does not depend on $t$, it follows that $\b
    u_t$ is co-stationary with $\bX_{N,t}$, for all $N$.
    
    Let us now show that
    $u_{lt} \in \mathcal D_t$. We have 
    \begin{align*}
        \b u^{q,t} = \b F_q \b \psi_{s_q, t} &= \begin{pmatrix}
                    \row_1(\b F_q) \b \psi_{s_q, t} 
                    \\ \vdots
                    \\ \row_r(\b F_q) \b \psi_{s_q, t} 
                \end{pmatrix} 
                = \begin{pmatrix}
                    \sc{\col_1(\b F_q^\tp), \b \psi_{s_q, t}}
                    \\ \vdots
                    \\ \sc{\col_r(\b F_q^\tp), \b \psi_{s_q, t}}
                \end{pmatrix}
                \\ &
                  = \begin{pmatrix}
                    \sc{\col_1(\b F_q^\tp), \b \Lambda^{-1/2}_{s_q} \b P_{s_q}^\adjoint \bX_{N,t}}
                    \\ \vdots
                    \\ \sc{\col_r(\b F_q^\tp), \b \Lambda^{-1/2}_{s_q} \b P_{s_q}^\adjoint \bX_{N,t}}
                \end{pmatrix}
                  = \begin{pmatrix}
                    \sc{\b P_{s_q} \b \Lambda_{s_q}^{-1/2}  \col_1(\b F_q^\tp),  \bX_{N,t}}
                    \\ \vdots
                    \\ \sc{\b P_{s_q} \b \Lambda_{s_q}^{-1/2}  \col_r(\b F_q^\tp),  \bX_{N,t}}
                \end{pmatrix},
%
    \end{align*}
    where $\row_l(A)$ and $\col_l(A)$ denote the $l$th row   and the $l$th column of $A$, respectively.
    We   need to show that the norm of $\b P_{s_q} \b \Lambda_{s_q}^{-1/2} \col_l(\b F_q)$, $l=1,\ldots, r$,  converges to zero. We have
    \begin{align*}
        \hnorm{\b P_{s_q} \b \Lambda_{s_q}^{-1/2}  \col_l(\b F_q^\tp)}^2
        &= \sc{\b P_{s_q} \b \Lambda_{s_q}^{-1/2}  \col_l(\b F_q^\tp), \b P_{s_q} \b \Lambda_{s_q}^{-1/2}  \col_l(\b F_q^\tp)}
     \\ &= \sc{\b \Lambda_{s_q}^{-1/2}  \col_l(\b F_q^\tp),\b P_{s_q}^\adjoint  \b P_{s_q} \b \Lambda_{s_q}^{-1/2}  \col_l(\b F_q^\tp)}
     \\ &= \sc{\b \Lambda_{s_q}^{-1/2}  \col_l(\b F_q^\tp), \b \Lambda_{s_q}^{-1/2}  \col_l(\b F_q^\tp)}
     \\ &= (\b F_q \b \Lambda_{s_q}^{-1} \b F_q^\tp )_{ll}
   \leq \opnorm{\b F_q \b \Lambda_{s_q}^{-1} \b F_q^\tp }
     = \opnorm{\b \Lambda_{s_q}^{-1}}
= \lambda_{s_q,r}^{-1}.
    \end{align*}
    Since $\lim_{q \rightarrow \infty} \lambda_{s_q, r} = \infty$, it follows that $u_{lt} \in \mathcal D_t$.
  \end{proof}
\end{prop}

We now know that each space $\mathcal D_t$ has dimension at least $r$. The following
result tells us that this dimension is exactly $r$.
\begin{lma}
  \label{lma:dim-Dt-exactly-r}
  The dimension of $\mathcal D_t$ is $r$, and $\left\{ u_{1t}, \ldots, u_{rt}
  \right\}$ is an orthonormal basis for~it. 
  \begin{proof}
    We 
     know that the dimension of $\mathcal D_t$ is at least $r$, and that~$u_{1t}, \ldots, u_{rt} \in \mathcal D_t$ are orthonormal. To complete the proof, we only need to show 
    that the dimension of~$\mathcal D_t$ is less than or equal to $r$. 
    Dropping the index $t$ to simplify notation, assume that~$\mathcal D$ has dimension larger than~$r$. 
   Then,  there exists $d_1, \ldots, d_{r+1} \in \mathcal D$ orthonormal, with~$d_j = \lim_{N \rightarrow \infty}
    d_{jN}$ in~$\LLP$, where $d_{jN} = \sc{\b v_{jN}, \b X_N }$ abd~$\b v_{jN} \in \bH_N$ such that~$\hnorm{ \b v_{jN}
    }^2 \rightarrow 0$ as $N \rightarrow \infty$. 
    
    Let~$\b A^{(N)}$ be the   $(r+1) \times (r+1)$ matrix with $(i,j)$th
    entry  $ A^{(N)}_{ij} = \eee{ d_{iN}d_{jN} }$. On the one
    hand,   $\b A^{(N)} \rightarrow \b I_{r+1}$. On the other hand,  
    \[
        A^{(N)}_{ij} = \sc{\b v_{iN}, \b \Sigma_N \b v_{jN}} = 
        \sc{ \b v_{iN}, \b P_N \b \Lambda_N \b P_N^\adjoint \b v_{jN}}  +
        \sc{ \b v_{iN}, \b Q_N \b \Phi_N \b Q_N^\adjoint \b v_{jN}}
    \]
    and, from the Cauchy--Schwarz inequality,  
    \begin{align*}
        | \sc{\b v_{iN}, \b Q_N \b \Phi_N \b Q_N^\adjoint \b v_{jN}} | & \leq
       \opnorm{\b Q_N}^2 \opnorm{\b \Phi_N} \hnorm{\b v_{jN}}\hnorm{\b v_{iN}} 
       \\ & \leq \lambda_{N, r+1} \hnorm{\b v_{jN}}\hnorm{\b v_{iN}} 
        \rightarrow 0,
    \end{align*}
    as $N \rightarrow \infty$. Therefore, the limit of $\b A^{(N)}$ is the same as
    the limit of $\b B^{(N)}$ with  $(i,j)$th entry~$B^{(N)}_{ij} = \sc{\b v_{iN}, \b P_N \b \Lambda_N \b P_N^\adjoint \b
    v_{jN}}$. But this is impossible since~$\b B^{(N)}$ is of rank at most $r$
    for all $N$. Therefore, the	dimension of $\mathcal D$ is at most $r$. 
  \end{proof}
\end{lma}

Consider the 
 orthogonal decomposition
$$
    X_{it} = \gamma_{it} + \delta_{it},\quad \text{with}\quad
     \gamma_{it} = \proj_{H_i}(X_{it} | \mathcal D_t)\quad \text{and}\quad 
 \delta_{it} = X_{it} - \gamma_{it},
$$
 of $X_{it}$ into its orthogonal  projection $\proj_{H_i}( \cdot | \mathcal{D}_t )$ onto $\vspan_{H_i}(\mathcal D_t)$ and its
orthogonal complement---see Appendix~\ref{sec:orthogonal_projections} for   definitions. 
Since $\gamma_{it} \in \vspan_{H_i}(\mathcal D_t)$, we can write it as a linear combination
 \[
    \gamma_{it} = b_{i1}u_{1t} + \cdots + b_{ir} u_{rt},
\]
(with coefficients in $H_i$) of $u_{1t}, \ldots, u_{rt} \in \bbR$,
where $b_{ij} = \ee \gamma_{it} u_{jt} = \ee X_{it} u_{jt} \in H_i$ does not depend on $t$ in view of the 
co-stationarity of $\b u_t$ and $\bX_{N,t}$.

The only technical result needed 
 to prove
Theorem~\ref{thm:equiv_factormodel_eigenvalues} is that $\b \xi$ is idiosyncratic,
that is,~$\lambda^\xi_1 < \infty$. The rest of this section is devoted to the derivation of this
result.


Although $\b \psi_N$ does not necessarily converge, we know intuitively that
the projection onto the entries of $\b \psi_N$ should somehow converge. The
following notion and result formalize this intuition.
\begin{defn}
    \label{defn:cauchy-sequence-subspaces}  
    Let $(\b v_N)_N \subset{\bbR^r}$ be an $r$-dimensional process with mean zero and~$\eee{ \b v_N \b v_N^\tp} =
    \b I_r$. Consider the orthogonal decomposition
    \[
        \b v_M = \b A_{MN}  \b v_N + \b \rho_{MN},
    \]
    and let $\cov(\b \rho_{MN})$ be the covariance matrix of $\b \rho_{MN}$. We say that $\{ \b v_N \}$
    generates a \emph{Cauchy sequence of subspaces} if, for all $\vep > 0$, there is an
    $M_{\vep} \geq 1$ such that  $\trace[\cov(\b \rho_{MN})] < \vep$    for all~$ N$ and $M > M_{\vep}$.
    .
\end{defn}


\begin{lma}
    \label{lma:projection-onto-cauchy-sequence-of-space-converge}
    Let $Y \in \LLPH{H}$. If $\{ \b v_N\} \subset L^2_{\bbR^r}(\Omega)$ generates a Cauchy sequence of subspaces and 
    $
    Y_N = \proj_H(Y | \b v_N),
    $ 
    then $\{ Y_N \}$ converges in $\LLPH{H}$.
    \begin{proof}
        Without any loss of generality, we can assume that $\ee \b v_N \b v_N^\tp = \b I_r$.
        Let 
        $$
        Y =    Y_N + r_N  = \b c_{N} \b v_N + r_N\quad\text{and}\quad 
        Y_M + r_M
        = \b c_{M } \b v_M + r_M
        $$          be orthogonal decompositions, with $\b c_{k} \defeq \sum_{l=1}^r b_{kl} \tensor \bz_l \in \bounded{\bbR^r, H}$, where $\bz_l$ is the $l$-th canonical basis vector of $\bbR^r$, $b_{kl} \in H$, $k=N, M$ 	and $l=1,\ldots, r$. We therefore get
        \[
            Y_N - Y_M = \b c_{N} \b v_N  - \b c_{M } \b v_M = r_M - r_N.
        \]
        The squared norm of the left-hand side can be written as the inner product between the
        middle and right expressions. Namely, 
        \begin{align*}
            \ee \hnorm{Y_N - Y_M}^2 &= \ee \sc{ \b c_{N} \b v_N - \b c_{M } \b v_M, r_M - r_N}
                                \\  &= \ee \sc{\b c_{N  }\b v_N, r_M} + \ee \sc{\b c_{M } \b v_M, r_N}
                                \\  &=: S_{1,NM} + S_{2,MN},\ \ \text{say,}
        \end{align*}
        where the cross-terms are zero by orthogonality.
        Since $\b v_N = \b A_{NM} \b v_M + \b \rho_{NM}$ and since~$\b v_M$ is
        uncorrelated with $r_M$,
        \[
            S_{1,NM} = \ee \sc{\b c_{N} \b \rho_{NM}, r_M};
        \]
        the Cauchy--Schwarz inequality, along with simple matrix algebra and matrix norm inequalities (see Section~\ref{sec:schatten_Norms}), then yields
        \begin{equation}
            \label{eq:bound-on-Smn}
            \left| S_{1,NM} \right|^2 \leq \trace\left[  \b c_{N} ^\adjoint \b c_{N}\right]
            {\ee \hnorm{r_M}^2}  \trace\left[\cov(\b \rho_{NM}) \right].
        \end{equation}
        Notice that 
        \[
            \ee \hnorm{Y}^2 = \trace[  \b c_{N} ^\adjoint \b c_{N}] + \ee \hnorm{r_N}^2 = \trace[  \b c_{M} ^\adjoint \b c_{M}] + \ee \hnorm{r_M}^2.
        \]
        Therefore, the first two terms of the right-hand side in
        \eqref{eq:bound-on-Smn} are bounded and, since~$\b v_N$ generates a Cauchy
        sequence of subspaces, $|S_{1,NM}|$ can be made arbitrarily small for large
        $N$ and $M$. A similar argument holds for $|S_{2,MN}|$ and, therefore, $\{ Y_N \}
        \subset \LLPH{H}$ is a Cauchy sequence, and thus converges, as was to be shown.
    \end{proof}
\end{lma}
We now show that the process $\{ \b \psi_N \}$, defined in \eqref{eq:defn-of-psi^N}, generates a Cauchy
sequence of subspaces. 
\begin{lma}
    \label{lma:psi-generates-cauchy-sequence-of-spaces} The process 
    $\{ \b \psi_N \}$ generates a Cauchy sequence of subspaces.
   \begin{proof}
       For $N > M$, we already have the orthogonal decomposition
       \begin{equation}
           \label{eq:decomposition-of-psi-into-psi-and-rho} 
           \b \psi_M = \b D \b \psi_N + \b \rho_{MN},
       \end{equation}
       with $\b D = \b \Lambda^{-1/2}_M \b P_M^\adjoint \canonicalProj_{M,N} \b P_N \b \Lambda_N^{1/2}$.
       Lemma~\ref{lma:largest-eigenvalue-of-residual} gives $\trace( \cov(\b
       \rho_{MN}) ) \leq r \lambda_{N, r+1}/ \lambda_{M,r}$, which can be made
       arbitrarily small for $N,M$ large enough. We now need to show that the
       residual of the projection of~$\b \psi_N$ onto $\b \psi_M$ is also
       small. By \eqref{eq:tensorform_orthogonal_projection}, the orthogonal projection of $\b \psi_N$ onto $\vspan_{\bbR^r}(\b \psi_M)$ is
       $\eee{ \b \psi_N \tensor \b \psi_M } \b \psi_M$. Expanding the expectation,  we get
       \begin{align*}
           \eee{ \b \psi_N \tensor \b \psi_M } 
            &= \eee{ ( \b \Lambda_N^{-1/2}\b P_N^\adjoint \b X_N ) \tensor (\b \Lambda_M^{-1/2} \b P_M^\adjoint  \b X_M  ) }
         \\ &= \eee{ \b \Lambda_N^{-1/2}\b P_N^\adjoint (\b X_N \tensor \b X_M) \b P_M \b \Lambda_M^{-1/2} ) }
       \\   &=   \b \Lambda_N^{-1/2}\b P_N^\adjoint \b \Sigma_N \canonicalProj_{M,N}^\adjoint \b P_M \b \Lambda_M^{-1/2} 
          \\   &=   \b \Lambda_N^{-1/2}\b P_N^\adjoint \left( 
            \b P_N \b \Lambda_N \b P_N^\adjoint + \b Q_N \b \Phi_N \b Q_N^\adjoint
        \right) \canonicalProj_{M,N}^\adjoint \b P_M \b \Lambda_M^{-1/2} 
         \\\   &=   \b  \Lambda_N^{-1/2}  \b \Lambda_N \b P_N^\adjoint \canonicalProj_{M,N}^\adjoint \b P_M \b
       \Lambda_M^{-1/2} 
            = \b D^\adjoint,
       \end{align*}
       where the third equality comes from the fact that $\bX_M = \canonicalProj_{M,N} \bX_N$,
       and the fourth one follows from
       \eqref{eq:eigenvector-decompositions}. We therefore have $\b \psi_N  = \b
       D^\adjoint \b \psi_M + \b \rho_{NM}$. Taking covariances, we get 
       \[
           \b I_r = \b D^\adjoint \b D + \cov(\b \rho_{NM})
           = \b D \b D^\adjoint + \cov( \b \rho_{MN} )
       \]
       where the second equality follows from
       \eqref{eq:decomposition-of-psi-into-psi-and-rho}. Taking traces yields
       $$\trace( \cov( \b \rho_{MN} ) ) = \trace( \cov( \b \rho_{NM} )),$$
        which
 completes the proof.
   \end{proof} 
\end{lma}
We now know that $\{ \b \psi_{ N,t } : N \in \bbN \} \subset \bbR^r$ generates a Cauchy sequence of subspaces. Let us show
that the projection onto $\vspan_H( \b \psi_{N,t})$ converges to the
projection onto~$\vspan_H(\mathcal D_t)$.
\begin{lma}
    \label{lma:projection-onto-psiN-asymptotically-equiv-to-proj-onto-D}
    For each $i \geq 1$ and $t \in \bbZ$,
    \[
        \lim_{N \rightarrow \infty} \proj_{H_i}( x_{it} | \b \psi_{N,t} ) = \proj_{H_i}( x_{it} | \mathcal D_t).
    \]
    \begin{proof}
        Let 
        \begin{equation}
            \label{eq:defn-gamma}
            \gamma_{N,it} \defeq \proj_{H_i}( x_{it} | \b \psi_{N,t} ).
        \end{equation}
        We know by Lemmas~\ref{lma:projection-onto-cauchy-sequence-of-space-converge} and~\ref{lma:psi-generates-cauchy-sequence-of-spaces} that
        \begin{equation}
            \label{eq:defn-of-gamma_infty}
            \gamma_{N,it} \rightarrow \gamma_{\infty,it} \in H_i.
            \qquad \text{as } N\rightarrow \infty,
        \end{equation}
        Let us show that $\gamma_{\infty,it} = \proj_{H_i}(x_{it} | \mathcal D_t)$. 
        Using \eqref{eq:tensorform_orthogonal_projection}, we have
        \begin{align*}
            \gamma_{\infty, it} &= \lim_{N \to \infty} \proj_{H_i}(x_{it} | \b \psi_{N,t})
                             = \lim_{N \to \infty} \eee{x_{it} \tensor \b \psi_{N,t}} \b \psi_{N,t}
                             \\ &= \lim_{N \to \infty} \eee{x_{it} \tensor \b \psi_{s_N,t}} \b \psi_{s_N,t}
                             = \lim_{N \to \infty} \eee{(x_{it} \tensor \b \psi_{s_N,t}) \b F_N^\adjoint} \b F_N \b \psi_{s_N,t}
                             \\ &= \lim_{N \to \infty} \eee{x_{it} \tensor (\b F_N \b \psi_{s_N,t}) } \b F_N \b \psi_{s_N,t}
                             = \eee{x_{it} \tensor \bu} \bu
                             = \proj_{H_i}(x_{it} | \mathcal D_t)
        \end{align*}
        where we have used the fact that $\b F_N \in \bbR^{r \times r}$, defined
        in the proof of Proposition~\ref{prop:construction-of-basis-of-D}, is
        a deterministic orthogonal matrix,  that $\lim_{N \to \infty} \b F_N \b \psi_{s_N,t} = \bu_t$ in $L^2_{\bbR^r}(\Omega)$, and Lemma~\ref{lma:dim-Dt-exactly-r}.
    \end{proof}
\end{lma}
Let $\delta_{N,i} \defeq x_{it} - \proj_{H_i}(x_{it} | \b \psi_N)$ and $\delta_{\infty,i} \defeq \lim_{N \to \infty} \delta_{N,i}$, which exists and is equal to~$x_{it} - \proj_{H_i}(x_{it} | \mathcal D_t)$ by Lemma~\ref{lma:projection-onto-psiN-asymptotically-equiv-to-proj-onto-D}.
We can now
show that the largest eigenvalue~$\lambda_{N,1}^{\delta_\infty}$ of the covariance operator of 
$ (\delta_{\infty,1t}, \ldots, \delta_{\infty,Nt})^\tp \in \bH_N$  is bounded.

\begin{lma}
    \label{lma:delta-is-idiosyncratic}
    For all $N \geq 1$, $\lambda_{N,1}^{\delta_\infty} \leq \lambda_{r+1}^x$. In particular, if $\lambda_{r+1}^x < \infty$,
    $\b \delta_\infty$ is idiosyncratic, i.e.\  $\sup_{N \geq 1} \lambda^{\delta_\infty}_{N,1} < \infty$.
    \begin{proof}
        We shall omit the subscript $t$ for simplicity. Let $\b \Sigma_M^{\delta_N}$ be the covariance operator of $(\delta_{N,1},
        \ldots,\delta_{N,M})^\tp$, for $N \geq M$ or $N = \infty$. Since $\delta_{N,i} $  converges to~$ \delta_{\infty,i}$ in~$L^2_{H_i}(\Omega)$, $\b \Sigma_M^{\delta_N}$ converges to $ \b \Sigma_M^{\delta_\infty}$ as $N \to \infty$.
        This implies that  the largest eigenvalue~$\lambda_{M,1}^{\delta_N}$  of the covariance operator of $(\delta_{N,1},\ldots, \delta_{N,M})^\tp$ converges to~$
        \lambda_{M,1}^{\delta_\infty}$ as~$N\rightarrow~\!\infty$, since
        $$ \left| \lambda_{M,1}^{\delta_N} - \lambda_{M,1}^{\delta_\infty} \right| \leq
        \opnorm{ \b \Sigma_M^{\delta_N} - \b \Sigma_M^{\delta_\infty} }$$
         \citep{hsing:eubank:2015}.
         Since~$\b \Sigma_M^{\delta_N}$ has the same eigenvalues as a {\it compression} of $\b \Sigma_N^{\delta_N}$ (see Lemma~\ref{lma:eigenvalue_inequalities} for the definition of  a {\it compression}), Lemma~\ref{lma:eigenvalue_inequalities} implies
        \[
            \lambda_{M,1}^{\delta_N} \leq \lambda_{N,1}^{\delta_N} = \lambda_{N,
            r+1}^{x},
        \]
        where we have used the fact that, by definition, $\lambda^{\delta_N}_{N,1} =
        \lambda^x_{N,r+1}$, see \eqref{eq:x_N-orthogonal-decomposition}. Taking the limit as $N \rightarrow \infty$, we get
        \[
            \lambda^{\delta_\infty}_{M,1} \leq \lambda^x_{r+1} < \infty,
        \]
        and, since this holds true for each $M$, it follows that $\lambda^{\delta_\infty}_1
        \leq \lambda^x_{r+1} < \infty$.
    \end{proof}
\end{lma}

\begin{proof}[Proof of Theorem~\ref{thm:equiv_factormodel_eigenvalues}]
    The ``only if'' part follows directly from Lemma~\ref{lma:eigenvalue_inequalities} since~$\b \chi_t$ and~$\bxi_t$ are uncorrelated. For the ``if'' part, let us assume   $\lambda^x_r =
    \infty$, $\lambda^x_{r+1} < \infty$. Then we know that $x_{it}$ has the representation
    \[
        x_{it} = \gamma_{it} + \delta_{it},
        \quad\text{with}\quad  \gamma_{it} = \proj_{H_i}(X_{it} | \mathcal D_t) 
        \quad	\text{and}\quad  \delta_{it} = X_{it} - \gamma_{it}.
    \]
    We know that 
    $$\gamma_{it} = b_{i1}u_{1t} + \cdots + b_{ir}u_{rt}, b_{il} \defeq \eee{u_{lt} x_{it}}$$
     (Lemma~\ref{lma:dim-Dt-exactly-r}) and is
    co-stationary with $\mathcal X$ (Proposition~\ref{prop:construction-of-basis-of-D}).
    From Lemma~\ref{lma:projection-onto-psiN-asymptotically-equiv-to-proj-onto-D},
    we also know that~$\gamma_{it} = \lim_{N \to \infty} \proj_{H_i}(x_{it} | \b \psi_N)$.
    Lemma~\ref{lma:delta-is-idiosyncratic} (with $\delta_{it} =
    \delta_{\infty, it}$) implies that~$\lambda_{N,1}^\delta \leq
    \lambda^x_{r+1}$; using Lemma~\ref{lma:eigenvalue_inequalities}, we get
    $\lambda^\chi_{N,r} \geq \lambda^x_{N,r} - \lambda^\delta_{N,1}$, and thus
    $\lambda^\chi_{N,r} \rightarrow \infty$ as $N \rightarrow \infty$.
\end{proof}

\begin{proof}[Proof of Theorem~\ref{thm:factormodel_uniqueness} ]
    Assume that $\chi_{it} = \tilde b_{i1} v_{1t} + \cdots + \tilde b_{ir} v_{rt}$, with $\tilde b_{il} \in H_i$ and~$ v_{lt} \in \LLP$ for all $i \geq 1$, $l=1,\ldots, r$, and $t \in \bbZ$.
    Also assume that $\mathfrak{p} \in \mathcal{D}_t$, so\linebreak that~$\mathfrak{p} = \lim_N \sc{\b a_N, \bX_{N}}$ for~$\b a_N \in \bH_N$ with 
    $\hnorm{\b a_N} \rightarrow 0$. Since $\lambda_1^\xi < \infty$, the non-correlation of $\chi$ and $\xi$ yields 
    $\mathfrak{p}  = \lim_N \sc{\b a_N, \b \chi_{N}}$, which implies that $\mathfrak{p} \in \vspan(\b v_t)$
    for~$\b v_t = (v_{1t}, \ldots, v_{rt})^\tp$, where $\b \chi_N \defeq
    (\chi_{1t},\ldots, \chi_{Nt})^\tp$. Therefore, 
    $$\vspan(\b v_t) \supset \vspan( \mathcal D_t ) 
    = \vspan( \b u_t ),$$
     where $\b u_t$ is constructed in Proposition~\ref{prop:construction-of-basis-of-D}.
    But $\vspan(\b v_t)$ and $\vspan(\b u_t)$ both have dimension $r$, so they are equal, and therefore
    \[
		\chi_{it} = \proj_{H_i}(X_{it} | \vspan(\b v_t) ) = \proj_{H_i}(X_{it} | \mathcal D_t ).
    \]
    If $X_{it} = \gamma_{it} + \delta_{it}$ is another functional factor
    representation with $r$ factors, 
    then we have 
    $$\gamma_{it} = \proj_{H_i}(X_{it} | \mathcal D_t) = \chi_{it}\quad\text{and}\quad \delta_{it} = X_{it} - \gamma_{it} 
    = X_{it} - \chi_{it} = \xi_{it},$$
     which shows the uniqueness of the decomposition.
\end{proof}

Before   turning to the next proofs, let us recall some operator theory.

\subsection{Classes of Operators}%
\label{sec:classes_of_operators}
\label{sec:schatten_Norms}

Let us first recall some basic definitions and properties of classes of operators on
separable Hilbert spaces \cite[Chapter~7]{Weid:80}.
Let $H_1$ and~$ H_2$ be separable (real) Hilbert spaces.
Denote by $\compact{H_1, H_2}$ the space of compact (linear) operators from
$H_1$ to~$H_2$. The space $\compact{H_1, H_2}$ is a subspace of~$\bounded{H_1, H_2}$
and consists 
 of all the operators $A \in \bounded{H_1, H_2}$ that admit a
singular value decomposition
\[
    A = \sum_{j \geq 1} s_j[A] u_j \tensor v_j,
\]
where $\{ s_j[A] \} \!\subset\! [0, \infty)$ are the singular values of $A$, ordered in decreasing order of magnitude and
satisfying $\lim_{j \rightarrow \infty} s_j[A] = 0$, $\{ u_j\} \!\subset\! H_1$ 
and $\{ v_j \} \!\subset \!H_2$ are orthonormal vectors. 
 
An operator $A\! \in\! \compact{H_1, H_2}$ satisfying $\tracenorm{A}\! \defeq\! \sum_j s_j[A]
\!<\! \infty$ is called a \emph{trace-class} operator;  the subspace of trace-class
operators is denoted by $\tc{H_1, H_2}$.~We have that~$\opnorm{A}\! \leq\! \tracenorm{A} \! =\!
\tracenorm{A^\adjoint}$ and,  if~$C\! \in\! \bounded{H_2, H}$, then $\tracenorm{CA}\! \leq\! \opnorm{C}
\tracenorm{A}$. 

An~operator $A \in \compact{H_1, H_2}$
satisfying $\Fnorm{A} \defeq \sqrt{\sum_{j} (s_j[A])^2} < \infty$ is called
\emph{Hilbert--Schmidt}; the subspace of Hilbert--Schmidt operators is denoted by
$\HS{H_1, H_2}$. We have that $\opnorm{A} \leq \Fnorm{A} = \Fnorm{A^\adjoint}$;  if $C \in \bounded{H_2,
H}$, then $\Fnorm{CA} \leq \opnorm{C}
\Fnorm{A}$. Furthermore, if $B \in \HS{H_2, H}$, then $\tracenorm{BA} \leq
\Fnorm{B} \Fnorm{A}$ and if $A \in \tc{H_1, H_2}$, \linebreak then~$\Fnorm{A} \leq
\tracenorm{A}$.
We shall use the shorthand notation $\tc{H}$ for $\tc{H, H}$, and similarly for
$\HS{H}$. If $A \in \tc{H}$,   we define its \emph{trace} by
\[
    \trace(A) \defeq \sum_{i \geq 1} \sc{A e_i, e_i},
\]
where $\{e_i\} \subset H$ is a complete orthonormal sequence (COS). The sum does not depend
on the choice of the COS, and $\abs{\trace(A)} \leq \tracenorm{A}$. If $A$ is symmetric
positive semi-definite (i.e.\ $A = A^\adjoint$ and $\sc{A u, u} \geq 0, \forall u \in H$), then
$\trace(A) = \tracenorm{A}$.  If $A \in \bounded{H_1, H_2}$ \linebreak and $B \in
\bounded{H_2, H_1}$ and either~$\tracenorm{A} < \infty$ or~$\Fnorm{A} +
\Fnorm{B} < \infty$, we\linebreak have $\trace(AB) = \trace(BA)$. For any $B
\in \HS{H}, \Fnorm{B} = \sqrt{ \trace(B B^\adjoint) }$. 

The spaces
$\compact{H}$, $\HS{H}$, and~$\tc{H}$ are also called \emph{Schatten spaces}.

\subsection{Matrix Representation of Operators}
\label{sec:matrices_and_operators}

We shall be dealing with $N \times L$ matrices whose $(i,l)$th entries, $ l=1,\ldots, L$ are
elements of the Hilbert space $H_i$, $i=1,\ldots, N$. If~$\bY$
is such a matrix,   letting  $\bz_l \in \bbR^L$ denote the $l$th canonical basis
vector of $\bbR^L$,   notice that $\bY$ induces a unique operator
$\operator{\bY} \in \bounded{\bbR^L, \bH_N}$ defined by $\operator{\bY}\bz_l =
\col_l(\bY)$, $l=1,\ldots, L$ where $\col_l(\bY)$ is the $l$th column of $\bY$, which is an
element of $\bH_N$. In other words, if 
\[
    \b \alpha = (\alpha_1,\ldots, \alpha_L)^\tp = \sum_{l=1}^L \alpha_l \bz_l \in \bbR^L, 
\]
then
\[
    \operator{\bY} \b \alpha = \sum_{l=1}^L \alpha_l \col_l(\bY) \in \bH_N.
\]
Thus, $\operator{\bY} \b \alpha$ is obtained by mimicking matrix multiplication. Notice that if $\bY_1$ is of the same nature as $\bY$, defining the sum $\bY + \bY_1$ through  entry-wise summation, we get 
\[
    \operator{\bY + \bY_1} = \operator{\bY} + \operator{\bY_1}.
\]
Conversely, an operator $C \in \bounded{\bbR^L, \bH_N}$ induces a unique $N \times L$ matrix $\matrixOf{C}$ with columns in $\bH_N$, where
\[
    \col_l( \matrixOf{C} ) \defeq C \bz_l \in \bH_N.
\]
If $C_1 \in \bounded{\bbR^L, \bH_N}$, 
$$\matrixOf{C + C_1} = \matrixOf{C} + \matrixOf{C_1}.$$
Since $\operator{ \matrixOf{C} } = C$ and~$\matrixOf{ \operator{ \bY } } = \bY$ for generic operators $C$ and matrices $\bY$, they yield a bijection between $N \times L$ matrices with columns in $\bH_N$ and operators in~$\bounded{\bbR^L, \bH_N}$:
\begin{equation}
    \label{eq:bijection_operators_and_matrices}
    \xymatrix{
        \{ N \times L \text{ matrices with columns in } \bH_N \} \ar@<2ex>[dd]^{\operator{\cdot}} 
        \\
        \\ \bounded{\bbR^L, \bH_N}  \ar@<2ex>[uu]^{\matrixOf{\cdot}} 
    }
\end{equation}
Notice that $\operator{\bY}^\adjoint \operator{\bY} \in \bounded{\bbR^L}$ and therefore it can be identified with a matrix in~$\bbR^{L \times L}$, and straightforward calculations yield that the $(l,k)$th entry of this matrix is~$\sc{\col_l(\bY), \col_k(\bY)}$.

\subsection{Proof of Proposition~\ref{prop:computing_loadings_and_factors}}
\label{sub:proof_computing_proposition}


Let $\bX_{NT}, \bchi_{NT},$ and $\bxi_{NT}$ denote the $N \times T$ matrices with $(i,t)$th entries  $x_{it}, \chi_{it},$ and  $\xi_{it}$,  respectively.
 The factor model then can be written under matrix form as
\[
    \bX_{NT} = \b \chi_{NT}+ \bxi_{NT}= \matrixOf{\bB_{N} \bU_T} + \bxi_{NT}, 
\]
following the notation of Section~\ref{sec:matrices_and_operators}.

In order to prove Proposition~\ref{prop:computing_loadings_and_factors}, notice that $\bX_{NT}$ induces (see Section~\ref{sec:matrices_and_operators}) a linear operator in~$\operator{\bX_{NT}} \in \bounded{\bbR^T, \bH_N}$. Recall that $\bB_N \in \bounded{\bbR^r, \bH_N}$ is the linear operator mapping the $l$-th canonical basis vector of $\bbR^r$ to $(b_{1l}, b_{2l}, \ldots, b_{Nl})^\tp \in \bH_N$. Thus, using the notation introduced in Section~\ref{sec:matrices_and_operators}, 
\begin{equation}
    \label{eq:bX_operator_notation}
    \operator{\bX_{NT}} =  \bB_N \bU_T + \operator{\bxi_{NT}}.
\end{equation}
\begin{proof}[Proof of Proposition~\ref{prop:computing_loadings_and_factors}]
    Notice that the objective function  \eqref{eq:least_squares_fit_criterion} can be re-written~as 
    \[
        P( \bB^{(k)}_N,  \bU^{(k)}_T) = \Fnorm{ \operator{\bX_{NT}} -  \bB_N^{(k)}  \bU_T^{(k)}}^2
    \]
    where $\Fnorm{\cdot}$ stands for the Hilbert--Schmidt norm (see
    Section~\ref{sec:schatten_Norms}). Since $\bB_N \bU_T$
    has rank at most $r$, the  problem of minimizing 
     \eqref{eq:least_squares_fit_criterion} has a clear solution: by the
    Eckart--Young--Mirsky Theorem
    \cite[Theorem~4.4.7]{hsing:eubank:2015},   the value of $\bB_N \bU_T$ minimizing  \eqref{eq:least_squares_fit_criterion}  is 
    the $r$-term truncation of the singular value decomposition of~$\operator{\bX_{NT}}$. 
    Let us write the singular value decomposition of~$\operator{ \bX_{NT} }$ under the form 
    \begin{equation} \label{eq:svdX}
        \operator{\bX_{NT}} = \sum_{l=1}^{N} \hat{\lambda}_{l}^{1/2} \hat{\b e}_{l} \tensor \hat{\b f}_{l}
    \end{equation}
    with $\hat \lambda_1 \geq \hat \lambda_2 \geq \cdots \geq 0$, where $\hat \be_l$
    (belonging to $\bH_N$) is rescaled to have norm~$\sqrt N$  and~$\hat {\b f}_l$s (belonging to $\bbR^T$) is rescaled to have norm~$\sqrt T$,  $l\geq 1$;  
     $\hat \lambda_l$, thus, is a rescaled singular values---we show in
    Lemma~\ref{lma:largest_sample_eigenvalue_bounded}
     that this rescaling is such that $\hat \lambda_1 =
    O_{\rm P}(1)$. To keep the notation simple, the sum
    is ranging over $l=1,\ldots,N$: if $N>T$, the last~$(N-T)$ values of~$\hat \lambda_{l}$ are set to zero. 

    To complete the proof, note that the steps listed in Proposition~\ref{prop:computing_loadings_and_factors} are just the steps involved in computing the singular value decomposition of $\operator{\bX_{NT}}$ from the spectral decomposition of $\operator{\bX_{NT}}^\adjoint \operator{\bX_{NT}} \in \bounded{\bbR^T}$; setting
    \begin{equation} \label{eq:tildebU_k} 
        \tildebU{k}_T  \defeq  \begin{pmatrix}
            \hat{\b f}_{1}^{\tp}
            \\ \vdots\vspace{1mm}
            \\ \hat{\b f}_{k}^{\tp}
        \end{pmatrix} \in    \mathbb{R}^{k \times T},
    \end{equation}
    and defining  $\tildeB{k}_N \in \bounded{\bbR^k, \bH_N}$ as  mapping $\hat{\b e}_l$, the $l$th canonical basis vector of $\bbR^k$,  to~$ \hat \lambda_l^{1/2} \hat{\b e}_l \in \bH_N$, 
    the composition $\tildeB{k}_N \tildebU{k}_T$ yields the $k$-term truncation of the singular value decomposition of $\operator{\bX_{NT}}$. 
\end{proof}
For the sake of completeness, let us provide the explicit form of the four steps in Proposition~\ref{prop:computing_loadings_and_factors} in the special case where~$H_i \equiv \LLR$ for all~$i \geq 1$, with the traditional inner-product $\sc{f,g}_{H_i} \defeq \int_0^1 f(\tau) g(\tau) d\tau$:
\begin{enumerate}
    \item compute $\bF \in \bbR^{T \times T}$, defined by $(\bF)_{st} \defeq \sum_{i=1}^N \int_0^1 x_{is}(\tau) x_{it}(\tau) d\tau /N$; 
    \item compute the leading $k$ eigenvalue/eigenvector pair $(\tilde \lambda_l, \tilde {\b f_l})$ of $\bF$, and set 
    \[\hat \lambda_l \defeq T^{-1/2} \tilde \lambda_l \in \bbR,  \qquad \hat {\b f}_l \defeq T^{1/2} \tilde {\b f}_l / \hnorm{\b f_l} \in \bbR^T;
    \]
    \item compute $\hat \be_l \defeq (\hat e_{1l}, \ldots, \hat e_{Nl})^\tp$, where $\hat e_{il} \in \LLR$ is defined by
        \[
            \hat e_{il}(\tau) \defeq \hat \lambda_l^{-1/2} T^{-1} \sum_{t=1}^T (\hat {\b f}_{l})_t x_{it}(\tau), \quad \tau \in [0,1],\ \ i=1,\ldots, N;
        \]
    \item set $\tildebU{k}_T \defeq (\hat {\b f}_1, \ldots, \hat {\b f}_k)^\tp \in \bbR^{k \times T}$ and define $\tildeB{k}_N$ as the operator in $\bounded{\bbR^k, \bH_N}$ mapping the $l$-th canonical basis vector of $\bbR^k$ to ${\hat \lambda_l}^{1/2} \hat \be_l$, $l = 1,\ldots, k$.
\end{enumerate}
Some remarks are in order. 
The reason for this choice is the following: $\tildebU{k}_T$ can be obtained by computing the
first $r$ eigenvectors of $\operator{\bX_{NT}}^\adjoint \operator{\bX_{NT}}$, and rescaling them by~$\sqrt T$.
Note that computing $\operator{\bX_{NT}}^\adjoint \operator{\bX_{NT}}$ requires computing $O(T^2)$ inner
products in $\bH_N$, and then computing the leading $k$ eigenvectors of a real $T \times T$ matrix. 
The loadings $\tildeB{k}_N$ can  be obtained by an eigendecomposition of $\operator{\bX_{NT}}\operator{\bX_{NT}}^\adjoint$.
However, this would require the eigendecomposition of an operator in $\bounded{\bH_N}$,
which could be computationally much more demanding than performing an eigendecomposition of
$\operator{\bX_{NT}}^\adjoint \operator{\bX_{NT}}$ to obtain~$\tildebU{k}_T$,   then multiplying it with $\bX_{NT}$
to obtain~$\tildeB{k}_N$.

Since our approach requires computing inner products~$\sc{\bx_s,
\bx_t} = \sum_{i = 1}^N \sc{x_{is}, x_{it}}_{H_i}$, any preferred method for
computing $\sc{x_{is},\  x_{it}}_{H_i}$ can be used. For instance, if $H_i\! = \LLR$
and~$x_{is},\  x_{it} \in \LLR$ are only observed on a grid with noise, i.e.\
\[
    y_{isk} = x_{is}(\tau_{ik}) + \vep_{isk}, \quad k=1,\ldots, K,
\]
then the raw observations $(\tau_{ik}, y_{isk})_k$ can be smoothed in order to
compute $$\sc{ x_{is}, x_{it}}_{H_i} = \int_{0}^1 x_{is}(\tau) x_{it}(\tau) d\tau.$$ 
If the curves are only sparsely observed, as in the setup of
\citet{yao:2005:sparse}, then the inner-products $\sc{ x_{is}, x_{it} }_{H_i}$ can
be computed by first obtaining FPC scores for the~$\{ x_{it}: t \in \bbZ\}$ processes, see
\citet{yao:2005:sparse} and \citet{rubin2020sparsely} for details.

\subsection{Proof of Theorem~\ref{thm:consistency_number_of_factors}}
\label{sub:proof_of_theorem_consistency_number_of_factors}

Before proceeding with the proof of  Theorem~\ref{thm:consistency_number_of_factors}, let us establish a couple of technical lemmas. 
Recall that  $\bxi_{NT}$ is the $N \times T$ matrix with  $(i,t)$th entry $\xi_{it}$ and that 
 $\operator{ \bxi_{NT} } \in \bounded{\bbR^T, \bH_N}$, see Section~\ref{sec:matrices_and_operators}.

\begin{Lemma}
    \label{lma:Frobenius_Norm_of_xi}
    Under Assumptions~\ref{assumption:c}, as $N$ and $T$ go to infinity, 
    \begin{equation}\label{eq:operator_norm_bxi_tp_bxi}
        \Fnorm{\operator{\bxi_{NT} }^\adjoint \operator{\bxi_{NT} } } = O_{\rm P}(N T /C_{N,T}).
    \end{equation}
    In particular, $\opnorm{ \operator{\bxi_{NT} } } = O_{\rm P}(\sqrt{N T/C_{N,T}})$.
    Additionally, if Assumption~\ref{assumption:a} holds, then Assumption~\ref{assumption:u_xi} holds with $\alpha=0$.
\end{Lemma}
\begin{proof}
    We have
    \[
        \Fnorm{ (NT)^{-1} \operator{\bxi_{NT} }^\adjoint \operator{\bxi_{NT} } }^2 = \sum_{t,s=1}^T (\sc{\bxi_t, \bxi_s}/N)^2/T^2 \leq 2 T^{-2}\sum_{t,s} ( \nu_N(t-s)^2 + \eta_{st}^2 ),
    \]
    where $\eta_{st}  \defeq   N^{-1} \sc{\bxi_{t}, \bxi_{s}} - \nu_{N}(t-s).$
    First, by Assumption~\ref{assumption:c},  
    \[
        T^{-2 }\sum_{t,s} \nu_N(t-s)^2 = O(T^{-1}), \quad N,T \to \infty.
    \]
    Second, by Assumption~\ref{assumption:c} and the Markov inequality,  for any $c > 0$, 
    \[
        \pp\left( T^{-2} \sum_{t,s} \eta_{st}^2 > c \right) \leq c^{-1} T^{-2} \sum_{s,t} \ee \eta_{st}^2 \leq c^{-1} M_1/N.
    \]
 Therefore, 
    \[
        T^{-2} \sum_{t,s} \eta_{st}^2 = O_{\rm P}(N^{-1}), \quad N,T \to \infty,
    \]
    which entails \eqref{eq:operator_norm_bxi_tp_bxi}. 
    The second statement of the lemma then 
    follows from the fact that~$\opnorm{\operator{\bxi_{NT}} }^2 = \opnorm{ \operator{\bxi_{NT} }^\adjoint \operator{\bxi_{NT} } }$.
    The last statement holds by noticing that
    $$\sum_{t=1}^T \bu_t \tensor \bxi_t = \bU_T \operator{\bxi_{NT}}^\adjoint,$$
     and applying the second statement of the lemma. 
\end{proof}

\begin{Lemma}
    \label{lma:frobenius_Norm_of_Bn_xi}
    Under Assumption~\ref{assumption:Bn_xi}, there exists $M_1 < \infty$
such that
    \[
        (NT)^{-1} \ee{ \Fnorm{\b B_N^\adjoint \operator{ \bxi_{NT} }}^2 } \leq M_1 \quad\text{ for all $N, T \geq 1$.}
    \]
    In particular,
    \(
    \Fnorm{\b B_N^\adjoint \operator{ \bxi_{NT} }} = O_{\rm P}(\sqrt{NT}),
    \)
    as $N,T \to \infty$.
\end{Lemma}
\begin{proof} Recall that $\rpnorm{\cdot}$ stands for the Euclidean norm and that $\bB_N^\adjoint \bxi_s \in \bbR^r$. 
    We have 
    $$\Fnorm{ \b B_N^\adjoint \operator{ \bxi_{NT} }}^2  = \sum_{s=1}^T \rpnorm{\b B_N^\adjoint \b \xi_s}^2.$$  
     Let $\singleProj_i: \bH_N \to H_i$ be the canonical projection mapping $(v_1,\ldots, v_N)^\tp \in \bH_N$\linebreak to $v_i \in H_i$, $i = 1,\ldots, N$ and define $\b b_i \defeq \singleProj_i \bB_N \in \bounded{\bbR^r, H_i}$. Since 
     $$ \bB_N\! =\! \sum_{i=1}^N \singleProj_i^\adjoint \singleProj_i \bB_N\! =\! \sum_{i=1}^N \singleProj_i^\adjoint \b b_i,$$
      we get
    $\b B_N^\adjoint \b \xi_s\! =\! \sum_{i=1}^N \b b_i^\adjoint \xi_{is}$. Since $\b b_i^\adjoint \xi_{is}\! \in\!\bbR^r$,
        \begin{align*}
        N^{-1} \ee \rpnorm{\b B_N^\adjoint \b \xi_s}^2 &= N^{-1}  \eee{ \left(\sum_{i=1}^N \b b_{i}^\adjoint \xi_{is} \right)^\tp \left( \sum_{j=1}^N \b b_{j}^\adjoint \xi_{js} \right) } 
                                                   \\ &
                                                   = N^{-1}  \sum_{i,j=1}^N \eee{ \left( \b b_{i}^\adjoint \xi_{is} \right)^\tp \left(  \b b_{j}^\adjoint \xi_{js} \right) } 
                                                \quad\   = N^{-1}  \sum_{i,j=1}^N \eee{ \trace\left( \left( \b b_{i}^\adjoint \xi_{is} \right) \tensor \left(  \b b_{j}^\adjoint \xi_{js} \right) \right) } 
                                                   \\ &= N^{-1}  \sum_{i,j=1}^N \eee{ \trace\left( \b b_{i}^\adjoint \left( \xi_{is} \tensor  \xi_{js} \right)\b b_{j} \right) } 
                                                   = N^{-1}  \sum_{i,j=1}^N \trace\left( \b b_{j} \b b_{i}^\adjoint \eee{ \xi_{is} \tensor  \xi_{js}} \right)  
                                                   \\ &\leq  N^{-1}  \sum_{i,j=1}^N \tracenorm{ \b b_{j} \b b_{i}^\adjoint  } \opnorm{ \eee{ \xi_{is} \tensor  \xi_{js}} }
                                                   \\ &
                                                   \leq  N^{-1}  \max_{i,j=1,\ldots, N} \tracenorm{ \b b_{j} \b b_{i}^\adjoint  } \sum_{i,j=1}^N  \opnorm{ \eee{ \xi_{is} \tensor  \xi_{js}} }
                                                   \\ &\leq  \left( \max_{i,j=1,\ldots, N} \Fnorm{ \b b_{j} } \Fnorm{ \b b_{i} } \right) N^{-1} \sum_{i,j=1}^N  \opnorm{ \eee{ \xi_{is} \tensor  \xi_{js}} }
                                                   \leq r M^3
        \end{align*}
        where we have used H\"older's inequality for operators.
        The claim follows directly since the bound $r M^3$ is independent of $s$.
\end{proof}

Let 
\begin{equation} \label{eq:hatU_k} 
    \hatbU{k}_T \defeq 
    \begin{pmatrix}
        \hat \lambda_1 \hat{\b f_1 }^\tp
        \\ \vdots
        \\ \hat \lambda_k \hat{\b f_k}^\tp
    \end{pmatrix} \in    \mathbb{R}^{k\times T};
\end{equation}
recall $\tildebU{k}$ from \eqref{eq:tildebU_k} 
and that $\bU_T \in \bbR^{r \times T}$ is the matrix containing the actual $r$ factors. 
Notice that the following equation links $\hatbU{k}_T$ and $\tildebU{k}_T$:
\begin{equation}
    \label{eq:relation-u_hat-u_tilde}
    \hatbU{k}_T = \tildebU{k}_T \operator{\bX_{NT}}^{\adjoint} \operator{\bX_{NT}}/(NT).
\end{equation}
Let us also define the   $k \times r$  matrix
\begin{equation}
    \label{eq:Qk}
    \bQ_k \defeq \tildebU{k}_T \bU_T^\adjoint \bB_N^\adjoint \bB_N/(NT) \in \bbR^{k \times r}.
\end{equation}
The next result is key in subsequent proofs.  Recall   that~$C_{N,T} \defeq \min\{ \sqrt{N}, \sqrt{T} \}$ and let~$\hatbu{k}_t \in \bbR^k$ be the $t$-th column of $\hatbU{k}_T$.
\begin{Theorem}  \label{thm:consistency_hatbu}
    Under Assumptions~\ref{assumption:a}, \ref{assumption:c}, \ref{assumption:Bn_xi}, for any fixed $k
    \in \{1, 2, \ldots \}$, 
    \[
        T^{-1} \Fnorm{\hatbU{k}_T - \bQ_k \bU_T}^2 = 
        \frac{1}{T}
        \sum_{t=1}^{T} \hnorm{ \hatbu{k}_{t} -\b Q_{k} \bu_{t} }^{2} =
        O_{\rm P}(C_{N,T}^{-2}),
    \]
    as $N,T \to \infty$,
    where $C_{N,T} \defeq \min \{\sqrt{N},\sqrt{T} \}$. 
\end{Theorem}
\begin{proof}
    Using \eqref{eq:relation-u_hat-u_tilde}, we get
    \begin{align*}
        \hatbU{k}_T -\b Q_{k} \bU_T  & = \frac{1}{NT}\tildebU{k}_T \operator{ \bxi_{NT} }^{\adjoint}
        \operator{ \bxi_{NT} } + \frac{1}{N T} \tildebU{k}_T \bU_T^\adjoint \bB_N^{\adjoint}
        \operator{ \bxi_{NT} } 
                                  \\ & \qquad\qquad\qquad\qquad\qquad\qquad\qquad\! +\frac{1}{NT}\tildebU{k}_T \operator{ \bxi_{NT} }^{\adjoint} \bB_N \bU_T
    \end{align*}
    and, therefore, 
    \begin{align*}
        \Fnorm{\hatbU{k}_T -\b Q_{k} \bU_T}  
            & \leq  \frac{1}{NT}\Fnorm{\tildebU{k}_T \operator{ \bxi_{NT} }^{\adjoint} \operator{ \bxi_{NT} }} + \frac{1}{N T} \Fnorm{\tildebU{k}_T \bU_T^\adjoint \bB_N^{\adjoint} \operator{ \bxi_{NT} }} 
         \\ &  \qquad \qquad\qquad\qquad\qquad\qquad\qquad + \frac{1}{NT} \Fnorm{\tildebU{k}_T \operator{ \bxi_{NT} }^{\adjoint} \bB_N \bU_T}.
    \end{align*}
    Let us consider each terms separately. For the first term,
    \begin{align*}
        \frac{1}{NT}\Fnorm{\tildebU{k}_T \operator{ \bxi_{NT} }^{\adjoint} \operator{ \bxi_{NT} }} 
         & \leq \frac{1}{NT}\opnorm{\tildebU{k}_T} \Fnorm{ \operator{ \bxi_{NT} }^{\adjoint} \operator{ \bxi_{NT} }} 
         \\ & = O_{\rm P}(\sqrt{T}C_{N,T}^{-1}),
    \end{align*}
    by Lemma~\ref{lma:Frobenius_Norm_of_xi}.
    For the second term, it follows from  Lemma~\ref{lma:frobenius_Norm_of_Bn_xi} that
    \begin{align*}
        \frac{1}{N T} \Fnorm{\tildebU{k}_T \bU_T^\adjoint \bB_N^{\adjoint} \operator{ \bxi_{NT} }}
         & \leq \frac{1}{N T} \opnorm{\tildebU{k}_T} \opnorm{ \bU_T^\adjoint } \Fnorm{\b B_{N}^{\adjoint} \operator{ \bxi_{NT} }} 
        \\ & = O_{\rm P}(\sqrt{T/N}),
    \end{align*}
    For the third term,
    \begin{align*}
        \frac{1}{NT} \Fnorm{\tildebU{k}_T \operator{ \bxi_{NT} }^{\adjoint} \bB_N \bU_T}
         & \leq \frac{1}{NT} \opnorm{\tildebU{k}_T} \opnorm{ \operator{ \bxi_{NT} }^{\adjoint} \b B_{N}} \Fnorm{ \bU_T} 
         \\& = O_{\rm P}(\sqrt{T/N}),
    \end{align*}
    again by Lemma~\ref{lma:frobenius_Norm_of_Bn_xi}.
    Therefore, 
    \begin{align}
    \frac{1}{T} \sum_{t=1}^T \hnorm{ \hatbu{k}_t - \bQ_k \bu_t }^2 = T^{-1} \Fnorm{\hatbU{k}_T - \bQ_k \bU_T}^2 & = O_{\rm P}(C_{N,T}^{-2}) + O_{\rm P}(N^{-1}) \nonumber
        \\ &
         = O_{\rm P}(C_{N,T}^{-2}). \label{eq:theorem_consistency_proof_bound}\vspace{-4mm}
    \end{align}\end{proof}

The next result tell us that the (rescaled) sample eigenvalue $\hat \lambda_1$ is bounded.
\begin{Lemma}
    \label{lma:largest_sample_eigenvalue_bounded}
    Under Assumptions~\ref{assumption:a}, \ref{assumption:b}, and \ref{assumption:c}, as $N,T \to \infty$,
    \[
        \hat \lambda_1 = O_{\rm P}(1) \qquad \text{and} \qquad \opnorm{ \operator{\bX_{NT}} } =
        O_{\rm P}(\sqrt{NT}).
    \]
    In particular, $\Fnorm{\hatbU{k}_T} \!\!= O_{\rm P}(\sqrt{T})$, where $\hatbU{k}_T\!\!$ is
    defined in the proof of Theorem~\ref{theorem:factors_average_bound}.
\end{Lemma}
\begin{proof}
    We have, by definition of $\hat \lambda_1$, and using Assumptions~\ref{assumption:a}, \ref{assumption:b},
    \ref{assumption:c}, \eqref{eq:bX_operator_notation}, and Lemma~\ref{lma:Frobenius_Norm_of_xi},
    \begin{align*}
        \hat \lambda_1 & = \opnorm{ \operator{ \bX_{NT}}^\adjoint \operator{\bX_{NT}}/(NT)} 
                    \\ & = (NT)^{-1} \opnorm{\operator{\bX_{NT}} }^2
                    \\ & \leq 2 (NT)^{-1} ( \opnorm{\b B_N \bU_T}^2 +
                    \opnorm{ \operator{ \bxi_{NT} } }^2)
                    \\ & \leq 2 (NT)^{-1} ( \opnorm{\b B_N}^2 \opnorm{\bU_T}^2 +
                    \opnorm{ \operator{ \bxi_{NT} } }^2)
                    \\ & \leq 2 (NT)^{-1} ( O(N) O_{\rm P}(T) + O_{\rm P}(N T C_{N,T}^{-1} ))
                     = O_{\rm P}(1),
    \end{align*}
    as $N,T \to \infty$.
    The last statement of the Lemma follows from the fact that, as $N,T \to \infty$,
    \[
        T^{-1} \Fnorm{\hatbU{k}_T}^2 = \hat \lambda_1^2 + \cdots + \hat
        \lambda_k^2 \leq k \hat \lambda_1^2 =  O_{\rm P}(1).\vspace{-9mm} 
    \]
\end{proof}


For a sequence of random variables $Y_N > 0$ and a sequence of constants $a_N > 0$,
we write~$Y_N = \Omega_{\rm P}(a_N)$ for 
$Y_N^{-1} = O_{\rm P}(a_N^{-1})$. 
\begin{Lemma}
    \label{lma:r_th_sample_eigenvalue_away_from_zero}
    Under Assumptions~\ref{assumption:a}, \ref{assumption:b}, \ref{assumption:c}, and
    \ref{assumption:Bn_xi},
 $
 \hat \lambda_r = \Omega_{\rm P}(1)$.
\end{Lemma}
\begin{proof}
    Write $\lambda_k[ A ]$ for the $k$-th largest eigenvalue of a self-adjoint
    operator $A$. By definition,
    \begin{align*}
        \hat \lambda_r & \defeq \lambda_r\left[ \operator{\bX_{NT}}^\adjoint \operator{\bX_{NT}}/(NT) \right]
                    \\ &= \lambda_r\big[ \bU_T^\adjoint \bB_N^\adjoint \bB_N \bU_T/(NT)  
                    \\ & \qquad + (NT)^{-1}( \bU_T^\adjoint \bB_N^\adjoint \operator{ \bxi_{NT} } + \operator{ \bxi_{NT} }^\adjoint \bB_N \bU_T + \operator{ \bxi_{NT} }^\adjoint \operator{ \bxi_{NT} }) \big].
    \end{align*}
Since the operator norm of  $ (NT)^{-1} \bU_T^\adjoint \bB_N^\adjoint \operator{ \bxi_{NT} } $ 
 is $o_{\rm P}(1)$ under
Assumptions~\ref{assumption:a},~\ref{assumption:c}, and~\ref{assumption:Bn_xi} (see
Lemma~\ref{lma:Frobenius_Norm_of_xi}), we have, by Lemma~\ref{lma:perturbation_of_singular_values},
\[
    \abs{ \hat \lambda_r - \lambda_r[\bU_T^\adjoint \bB_N^\adjoint \bB_N \bU_T/(NT)]} = o_{\rm P}(1)
\] 
as $N,T \to \infty$.
We
therefore just need to show that $\lambda_r[\bU_T^\adjoint \bB_N^\adjoint \bB_N \bU_T/(NT)] =
\Omega_{\rm P}(1)$ as $N,T \to \infty$.
Using the Courant--Fischer--Weyl minimax characterization of eigenvalues \citep{hsing:eubank:2015}, we obtain 
\[
    \lambda_r[\bU_T^\adjoint \bB_N^\adjoint \bB_N \bU_T/(NT)] \geq \lambda_r[\bB_N^\adjoint \bB_N/N]
    \cdot
    \lambda_r[\bU_T^\adjoint \bU_T/T]. 
\]
Now, by Assumption~\ref{assumption:b}, $\lambda_r[\bB_N^\adjoint \bB_N/N] = \Omega_{\rm P}(1)$ as $N \to \infty$, and, by Assumption~\ref{assumption:a},

\[
    \lambda_r[\bU_T^\adjoint \bU_T/T] = \lambda_r[\bU_T \bU_T^\adjoint/T] = \Omega_{\rm P}(1),
\]
as $T \to \infty$.  The result follows.
\end{proof}

The next lemma is dealing with the singular values of the matrix $\bQ_k$ defined in \eqref{eq:Qk}.  Denote by $s_j[A]$ the $j$th largest singular value of a matrix $A$. 
\begin{Lemma}
    \label{lma:minimal_singular_value_Qk}
    Under Assumptions~\ref{assumption:a}, \ref{assumption:b}, \ref{assumption:c}, and \ref{assumption:Bn_xi},   letting $r_* \defeq \min(k,r)$, we have
    \[ 
        s_{r_*}[\bQ_k] = \Omega_{\rm P}(1) \quad \text{and} \quad s_{r_*}\left[\bQ_k
        \bU_T/\sqrt{T}\right] =
        \Omega_{\rm P}(1)
    \]
    as $N,T \to \infty$.
    In other words, for $k \leq r$,  the rows of $\bQ_k$ are, for $N,T$ large enough,  
    linearly independent   and, for  $k \geq r$,  $\bQ_k$ has a left generalized inverse
    $\bQ_k^-$ with $\opnorm{\bQ_k^-} = O_{\rm P}(1)$.
\end{Lemma}
\begin{proof}
The Courant--Fischer--Weyl minimax characterization of singular values
    \citep{hsing:eubank:2015} yields 
    \[ 
        s_{r_*}\left[\bQ_k \bU_T/\sqrt{T}\right] \leq s_1\left[\bU_T/\sqrt{T}\right] s_{r_*}[\bQ_k] = \left(s_1\left[\bU_T
        \bU_T^\adjoint/\sqrt{T}\right]\right)^{1/2} s_{r_*}[\bQ_k],
    \]
    hence, given that $s_1[\bU_T \bU_T^\adjoint/T] = O_{\rm P}(1)$ as $T \to \infty$, by Assumption~\ref{assumption:a},
    \[ 
        s_{r_*}[\bQ_k] \geq (s_1\left[\bU_T \bU_T^\adjoint/T\right])^{-1/2} s_{r_*}\left[\bQ_k \bU_T/\sqrt{T}\right] =
        \Omega_{\rm P}(1) s_{r_*}\left[\bQ_k \bU_T/\sqrt{T}\right],
    \]
    and, by \eqref{eq:theorem_consistency_proof_bound} and Lemma~\ref{lma:perturbation_of_singular_values},
    \[
        s_{r_*}\left[\bQ_k \bU_T/\sqrt{T}\right] = s_{r_*}\left[\hatbU{k}_T/\sqrt T
        \right] + o_{\rm P}(1) = \hat
        \lambda_{r_*} + o_{\rm P}(1). 
    \]
 Lemma~\ref{lma:r_th_sample_eigenvalue_away_from_zero} then implies that 
    $s_{r_*}\left[\bQ_k \bU_T/\sqrt{T}\right] = \Omega_{\rm P}(1)$ as $N,T \to \infty$, which completes the proof.
\end{proof}

We will use the following result (which follows directly from ordinary least
square regression theory): if $\bW$ is $k \times T$, then
\begin{equation}
    \label{eq:OLS_and_orthogonal_projection}
    V(k, \bW) = N^{-1} T^{-1} \trace( \operator{\bX_{NT}} \b M_{\bW} \operator{\bX_{NT}}^\adjoint ),
\end{equation}
where $\b M_{\bW} \defeq \b I_T -  \b P_{\bW}$  is the $T
\times T$ orthogonal projection matrix with span the null space of $\bW$ or,  in
other words, $\b P_{\bW}$ is the orthogonal projection onto the column space of $\bW^\tp$.
If $\bW$ is of full row rank, then
\begin{equation}
    \label{eq:orthogonal_projection_matrix}
    \b P_{\bW} \defeq \bW^\adjoint (\bW \bW^\adjoint)^{-1} \bW.
\end{equation}
Recall the definition of $V(k, \tildebU{k}_T)$ given in \eqref{eq:defn_V}. Assuming that $\hat \lambda_k > 0$ (otherwise it is clear that $r < k$), notice that $V(k, \tildebU{k}_T) = V(k, \hatbU{k}_T)$,   hence 
 $$\IC(k) = V(k, \hatbU{k}_T) + k\, g(N,T).$$
\begin{Lemma} \label{lemma:2}
    Under Assumptions~\ref{assumption:a}, \ref{assumption:b}, and \ref{assumption:c},
    for $1\leq k \leq r$, with $\bQ_{k}$ defined in \eqref{eq:Qk}, we have
    \[
        \abs{ V(k,\hatbU{k}_T) - V(k,\b  Q_{k}\bU_T)} = O_{\rm P}(C_{N,T}^{-1})
    \]
    as $N,T \to \infty$.
\end{Lemma}
\begin{proof}
    Using \eqref{eq:OLS_and_orthogonal_projection}, we get
    \begin{align*}
        \abs{ V(k,\hatbU{k}_T) - V(k,\b Q_{k} \bU_T) } & = N^{-1} T^{-1} \trace\left[ \operator{\bX_{NT}} ( \b P_{\b Q_k \bU_T} - \b P_{\hatbU{k}_T} ) \operator{\bX_{NT}}^\adjoint \right]
        \\ & \leq \tracenorm{  \b P_{\hatbU{k}_T} - \b P_{\b Q_k \bU_T} } \opnorm{
        \operator{\bX_{NT}}^\adjoint \operator{\bX_{NT}}/(NT) }
    \\ & = \tracenorm{  \b P_{\hatbU{k}_T} - \b P_{\b Q_k \bU_T} } O_{\rm P}(1)
    \end{align*}
    as $N,T \to \infty$, by Lemma~\ref{lma:largest_sample_eigenvalue_bounded}.
    We now deal with $\tracenorm{  \b P_{\hatbU{k}_T} - \b P_{\b Q_k \bU_T} }$. 
    Letting 
    $$\b D_{k} \defeq  \hatbU{k}_T (\hatbU{k}_T)^\adjoint/T\quad\text{and}\quad\b D_{0}
    \defeq \bQ_{k} \bU_T  \bU_T^\adjoint \bQ_{k}^{\adjoint} /T,$$
     we know by
    Lemmas~\ref{lma:r_th_sample_eigenvalue_away_from_zero} and
    \ref{lma:minimal_singular_value_Qk} that these two matrices are
    invertible  for $N$ and $T$ large enough, and 
    \begin{align*}
        \b P_{\hatbU{k}_T} - \b P_{\b Q_k \bU_T} 
        &= T^{-1}\left((\hatbU{k}_T)^\adjoint \bD_k^{-1} \hatbU{k}_T - (\bQ_k \bU_T)^\adjoint\bD_0^{-1} (\bQ_k \bU_T) \right)
     \\ &= T^{-1}  (\hatbU{k}_T - \bQ_k \bU_T)^\adjoint \bD_k^{-1} \hatbU{k}_T 
     \\ & \quad +  T^{-1}  (\bQ_k \bU_T)^\adjoint (\bD_k^{-1} - \bD_0^{-1}) (\hatbU{k}_T)  
     \\ &  \quad + T^{-1} (\bQ_k \bU_T)^\adjoint\bD_0^{-1} (\hatbU{k}_T - \bQ_k \bU_T).  
    \end{align*}
    Therefore,
\begin{align*}
    \tracenorm{\b P_{\hatbU{k}_T} - \b P_{\b Q_k \bU_T} } & \leq  T^{-1}
    \Fnorm{\hatbU{k}_T - \bQ_k \bU_T} \Fnorm{ \bD_k^{-1} } \Fnorm{\hatbU{k}_T }
     \\ & \quad +  T^{-1}  \Fnorm{\bQ_k \bU_T} \Fnorm{\bD_k^{-1} - \bD_0^{-1}}
     \Fnorm{\hatbU{k}_T }  
     \\ &  \quad + T^{-1} \Fnorm{\bQ_k \bU_T} \Fnorm{\bD_0^{-1}} \Fnorm{\hatbU{k}_T - \bQ_k \bU_T}  .
\end{align*}
We know from 
Theorem~\ref{thm:consistency_hatbu} that $\Fnorm{\hatbU{k}_T - \bQ_k \bU_T} = O_{\rm P}(\sqrt{T}/C_{N,T})$ 
 and, by
Lemmas~\ref{lma:r_th_sample_eigenvalue_away_from_zero} and~\ref{lma:largest_sample_eigenvalue_bounded}, that 
 $\Fnorm{ \bD_k^{-1} } = O_{\rm P}(1)$ and $\Fnorm{\hatbU{k}_T} =
O_{\rm P}(\sqrt{T})$, all as $N,T \to \infty$. Furthermore,
\begin{align*}
    \Fnorm{\bQ_k \bU_T}  & \leq \opnorm{\bQ_k} \Fnorm{\bU_T}  
    \\ &\leq \opnorm{\tildebU{k}_T} \opnorm{\bU_T} \opnorm{\b B_N^\adjoint \b B_N} \Fnorm{\bU_T}/(NT)
    = O_{\rm P}(\sqrt{T})
\end{align*}
as $N,T \to \infty$.
We are only left to bound $\Fnorm{\bD_k^{-1} - \bD_0^{-1}}$ and $\Fnorm{\bD_0^{-1}}$.
First, note that, by Lemma~\ref{lma:minimal_singular_value_Qk},  
\[ 
    \Fnorm{\bD_0^{-1}} \leq \sqrt{k} (\lambda_k[\bD_0])^{-1} = \sqrt{k}
    (\Omega_{\rm P}(1))^{-1} = O_{\rm P}(1),
\]
as $N,T \to \infty$.
Using this, we get 
\begin{align*}
    \Fnorm{\bD_k^{-1} - \bD_0^{-1}} &=  \Fnorm{\bD_k^{-1}(\bD_0 - \bD_k)\bD_0^{-1}}
                                 \\ &\leq  \Fnorm{\bD_k^{-1}} \Fnorm{\bD_0 - \bD_k} \Fnorm{\bD_0^{-1}}
                                = O_{\rm P}(C_{N,T}^{-1})
\end{align*}
as $N,T \to \infty$.
Indeed, from Theorem~\ref{thm:consistency_hatbu} and
Lemma~\ref{lma:largest_sample_eigenvalue_bounded}, as $N,T \to \infty$, we have
\begin{align*}
\MoveEqLeft {\Fnorm{\bD_0 - \bD_k}}\\ 
  &=  \Fnorm{ (\hatbU{k}_T - \bQ_k \bU_T) (\hatbU{k}_T)^\adjoint/T + \bQ_k
    \bU_T (\hatbU{k}_T - \bQ_k \bU_T )^\adjoint/T }
                       \\ & \leq \Fnorm{\hatbU{k}_T - \bQ_k \bU_T} \Fnorm{\hatbU{k}_T}/T + \opnorm{\bQ_k}
                       \Fnorm{\bU_T} \Fnorm{\hatbU{k}_T - \bQ_k \bU_T}/T
    \\ &= O_{\rm P}(\sqrt{T}/C_{N,T}) O_{\rm P}(\sqrt{T})/T + O_{\rm P}(1) O_{\rm P}(\sqrt{T})
    O_{\rm P}(\sqrt{T}/C_{N,T})/T
    = O_{\rm P}\left(C_{N,T}^{-1}\right),
\end{align*}
hence
 $ \tracenorm{\b P_{\hatbU{k}_T} - \b P_{\b Q_k \bU_T} } =
O_{\rm P}(C_{N,T}^{-1})$ and 
$$\abs{ V(k,\hatbU{k}_T) - V(k,\b Q_{k} \bU_T) } =
O_{\rm P}(C_{N,T}^{-1}).\hfill$$
$\,$\vspace{-12mm}

\end{proof}

\begin{Lemma} \label{lma:VQu_minus_Vu_bounded_below} 
    Under Assumptions~\ref{assumption:a}, \ref{assumption:b}, and  \ref{assumption:c},
    for $k<r$, 
\[
    V(k,\b Q_{k} \bU_T)-V(r,\bU_T) = \Omega_{\rm P}(1)
\]
as $N,T \to \infty$.
\end{Lemma}
\begin{proof}
    Using \eqref{eq:OLS_and_orthogonal_projection} and the notation of
    \eqref{eq:orthogonal_projection_matrix}, we have
\begin{align}
    V(k,\b Q_{k}\bU_T)-V(r,\bU_T) &= N^{-1} T^{-1} \trace\left[ \operator{\bX_{NT}} (\b P_{\bU_T} - \b
    P_{\bQ_k \bU_T} ) \operator{\bX_{NT}}^\adjoint \right]  \nonumber
    \\ &= (NT)^{-1}\trace[ \bB_N \bU_T (\b P_{\bU_T} - \b P_{\bQ_k \bU_T} ) \bU_T^\adjoint
    \bB_N^\adjoint]\nonumber
    \\ & \qquad + (NT)^{-1}\trace\left[ \bB_N \bU_T( \b P_{\bU_T} - \b P_{\bQ_k \bU_T} )
    \operator{ \bxi_{NT} }^\adjoint \right]  \label{eq:difference_VQU_Vu}
    \\ & \qquad + (NT)^{-1}\trace\left[
    \operator{ \bxi_{NT} } ( \b P_{\bU_T} - \b P_{\bQ_k \bU_T} ) \operator{\bX_{NT}}^\adjoint \right]. \nonumber
\end{align}
Let us show that the first term in \eqref{eq:difference_VQU_Vu}  is $\Omega_{\rm P}(1)$, while the other two   are
$o_{\rm P}(1)$, both as~$N,T \to \infty$. 

For the first term in \eqref{eq:difference_VQU_Vu}, we have
\begin{align*}
    \MoveEqLeft (NT)^{-1}\trace[ \bB_N \bU_T ( \b P_{\bU_T} - \b P_{\bQ_k \bU_T} ) \bU_T^\adjoint \bB_N^\adjoint ]
    \\ &= (NT)^{-1}\trace[ \{\bU_T
    ( \b P_{\bU_T} - \b P_{\bQ_k \bU_T} ) \bU_T^\adjoint\} \{ \bB_N^\adjoint \bB_N \} ]
    \\ & \geq \lambda_1[ \bU_T ( \b P_{\bU_T} - \b P_{\bQ_k \bU_T} ) \bU_T^\adjoint/T ] \cdot \lambda_r[
    \bB_N^\adjoint \bB_N/N ],
\end{align*}
by Lemma~\ref{lma:trace_of_product_inequality}, which is applicable because the
matrices in brackets are symmetric positive semi-definite, since $\b P_{\bU_T} - \b
P_{\bQ_k \bU_T}$ is an orthogonal projection matrix (the row space of $\bQ_k
\bU_T$ is a subspace of the row space of $\bU_T$).
We also have that 
\begin{align*}
     \lambda_1[ \bU_T ( \b P_{\bU_T}\! &- \b P_{\bQ_k \bU_T} ) \bU_T^\adjoint/T ] 
    \\ 
    & = \lambda_1[ \bU_T \bU_T^\adjoint/T \!- \{ \bU_T \bU_T^\adjoint \bQ_k^\adjoint (\bQ_k \bU_T \bU_T^\adjoint \bQ_k^\adjoint)^{-1} \bQ_k \bU_T
    \bU_T^\adjoint/T \} ]
    \\ & \geq \lambda_r[\bU_T \bU_T^\adjoint/T] 
     = \Omega_{\rm P}(1)
\end{align*}
as $N,T \to \infty$,
where the inequality follows from
Lemma~\ref{lma:largest_eigenvalue_of_difference_of_two_spd_matrices_with_unequal_rank}, since the matrix in
brackets has rank~$k < r$;  the $\Omega_{\rm P}(1)$ is a consequence of Assumption~\ref{assumption:a}. 
We also know by Assumption~\ref{assumption:b} that  $\lambda_r[ \bB_N^\adjoint \bB_N/N ] =
\Omega_{\rm P}(1)$ and, therefore, 
\[
    (NT)^{-1}\trace[ \bB_N \bU_T ( \b P_{\bU_T} - \b P_{\bQ_k \bU_T} ) \bU_T^\adjoint \bB_N^\adjoint ]
    = \Omega_{\rm P}(1)
\]
as $N,T \to \infty$.

Turning to the second term in \eqref{eq:difference_VQU_Vu}, we have 
\begin{align*}
    \MoveEqLeft
    \abs{ (NT)^{-1}\trace\left[ \bB_N \bU_T( \b P_{\bU_T} - \b P_{\bQ_k \bU_T} ) \operator{ \bxi_{NT} }^\adjoint \right]   } 
        \\ & \leq \tracenorm{ \bB_N \bU_T( \b P_{\bU_T} - \b P_{\bQ_k \bU_T} ) \operator{ \bxi_{NT} }^\adjoint }/(NT)
        \\ & \leq \opnorm{ \bB_N }\opnorm{\bU_T} \tracenorm{ \b P_{\bU_T} - \b P_{\bQ_k \bU_T}
        } \opnorm{\operator{ \bxi_{NT} }^\adjoint}/(NT)
        \\ & = O_{\rm P}\left(C_{N,T}^{-1/2}\right) \leq o_{\rm P}(1)
    \end{align*}
    as $N,T \to \infty$,
where we have used inequalities for the trace norm of operator products, the fact
that $\b P_{\bU_T} - \b P_{\bQ_k \bU_T}$ has rank less than $r$,
Assumptions~\ref{assumption:a} and \ref{assumption:b}, and
Lemma~\ref{lma:Frobenius_Norm_of_xi}.

To conclude, let us   consider the third term in \eqref{eq:difference_VQU_Vu}:
\begin{align*}
    \MoveEqLeft
    \abs{ (NT)^{-1}\trace\left[ \operator{ \bxi_{NT} } ( \b P_{\bU_T} - \b P_{\bQ_k \bU_T} ) \operator{\bX_{NT}}^\adjoint \right]} 
         \\ & \leq  (NT)^{-1}\tracenorm{\operator{ \bxi_{NT} } ( \b P_{\bU_T} - \b P_{\bQ_k \bU_T} ) \operator{\bX_{NT}}^\adjoint}
         \\ & \leq (NT)^{-1}\opnorm{\operator{ \bxi_{NT} }} \tracenorm{ \b P_{\bU_T} - \b P_{\bQ_k \bU_T} } \opnorm{\operator{\bX_{NT}}}
          \\ & = O_{\rm P}\left(C_{N,T}^{-1/2}\right)
           \leq o_{\rm P}(1)
\end{align*}
as $N,T \to \infty$,
where we  used H\"older's inequalities  again for the trace norm of operator products, the fact
that $\b P_{\bU_T} - \b P_{\bQ_k \bU_T}$ has rank less than $r$,
and Lemmas~\ref{lma:Frobenius_Norm_of_xi} and
\ref{lma:largest_sample_eigenvalue_bounded}.
\end{proof}

\begin{Lemma} \label{lma:bound_Vuk_Vur} 
    Assume $r < k \leq k_\text{max}$.
Under Assumptions~\ref{assumption:a}, \ref{assumption:b}, \ref{assumption:c}, and \ref{assumption:Bn_xi},
we~have, as~$N,T \to \infty$, 
\begin{equation}
    \label{eq:bound_Vuk_Vur_general}
    V(k,\hatbU{k}_T)-V(r, \hatbU{r}_T)= O_{\rm P}\left(C^{-1}_{NT}\right).
\end{equation}
\end{Lemma}
\begin{proof}
    Since
    \begin{align*}
        \abs{ V(k,\hatbU{k}_T)-V(r,\hatbU{r}_T) } 
        & \leq \abs{ V(k,\hatbU{k}_T)-V(r,\bU_T) } + \abs{ V(r,\bU_T)-V(k,\hatbU{r}_T)} \\
        & \leq 2 \max_{r\leq k \leq k_{\text{max}} } \abs{ V(k,\hatbU{k}_T)-V(r,\bU_T) },
    \end{align*}
    it is sufficient to prove that for each $k \geq r$, under Assumptions~\ref{assumption:a}, \ref{assumption:b}, \ref{assumption:c}, \ref{assumption:Bn_xi},
    \[
        \abs{ V(k,\hatbU{k}_T)-V(r,\bU_T) } = O_{\rm P}(C_{N,T}^{-1})
    \]
    as $N,T \to \infty$ (here and below, all rates are meant as $N,T \to \infty$).

    Let $\b Q_{k}^{-}$ denote the generalized inverse of $\b Q_{k}$, satisfying 
    $\b Q_{k}^{-}\b Q_{k}=I_{r}$. This generalized inverse
     is well
    defined (for
    $N,T$ large enough) in view of Lemma~\ref{lma:minimal_singular_value_Qk}, and such that~$\opnorm{\bQ_k^-} = O_{\rm P}(1)$.
    Since 
    \[
        \operator{\bX_{NT}} = \bB_N \bU_T + \operator{\bxi_{NT}} = \bB_N \bQ_k^- \bQ_k \bU_T + \operator{\bxi_{NT}}  = \bB_N \bQ_k^-
        \hatbU{k}_T + \be
    \]
    where $\be \defeq \operator{\bxi_{NT}} - \bB_N \bQ_k^-(\hatbU{k}_T - \bQ_k \bU_T) \in \bounded{\bbR^T, \bH_N}$, we have that
    \[
        V(k, \hatbU{k}_T) = (NT)^{-1} \trace[ \operator{\bX_{NT}} \bM_{\hatbU{k}_T} \operator{\bX_{NT}}^\adjoint ] =
        (NT)^{-1} \trace[ \be \bM_{\hatbU{k}_T} \be^\adjoint ].
    \]
    We also have that
    \[
        V(r, \bU_T) = \trace[\operator{\bX_{NT}} \bM_{\bU_T} \operator{\bX_{NT}}]/(NT) = \trace[ \operator{ \bxi_{NT} } \bM_{\bU_T} \operator{ \bxi_{NT} }^\adjoint ]/(NT).
    \]
    Therefore,
    \begin{align}
        \MoveEqLeft \abs{V(r, \bU_T) - V(k, \hatbU{k}_T) } \nonumber 
        \\ & \leq 2 \abs{ (NT)^{-1} \trace[ \operator{ \bxi_{NT} } \bM_{\hatbU{k}_T} \{ \bQ_k^- (\hatbU{k}_T - \bQ_k \bU_T) \}^\adjoint \bB_N^\adjoint ]}  \nonumber
        \\ & \quad + \abs{ (NT)^{-1} \trace[ \bB_N \bQ_k^- (\hatbU{k}_T - \bQ_k \bU_T) \bM_{\hatbU{k}_T} \{ \bQ_k^- (\hatbU{k}_T - \bQ_k \bU_T) \}^\adjoint \bB_N^\adjoint ]} \label{eq:bound_Vuk_Vur_proof_steps}
        \\ & \quad + \abs{ (NT)^{-1} \trace[ \operator{ \bxi_{NT} } (\bM_{\bU_T} - \bM_{\hatbU{k}_T}) \operator{ \bxi_{NT} }^\adjoint ] }\eqdef S_1+S_2+S_3, \text{ say.}
        \nonumber
    \end{align}

    In view of Theorem~\ref{thm:consistency_hatbu} and 
    Assumption~\ref{assumption:Bn_xi}, 
    \begin{align*}
      S_1  & \leq  2 (NT)^{-1} \opnorm{\bM_{\hatbU{k}_T}} \tracenorm{ \{ \bQ_k^-
        (\hatbU{k}_T - \bQ_k \bU_T ) \}^\adjoint \bB_N^\adjoint \operator{ \bxi_{NT} } }
        \\ & \leq 2 (NT)^{-1} \opnorm{\bQ_k^-} \Fnorm{\hatbU{k}_T - \bQ_k \bU_T}
        \Fnorm{\bB_N^\adjoint \operator{ \bxi_{NT} }}
        \\ & = (NT)^{-1} O_{\rm P}(\sqrt{T}/C_{N,T}) O_{\rm P}(\sqrt{NT})
         = O_{\rm P}\left(C_{N,T}^{-2}\right).
    \end{align*}
  It follows from Theorem~\ref{thm:consistency_hatbu} and Assumption~\ref{assumption:b} that 
    \begin{align*}
  S_2      & \leq  (NT)^{-1} \opnorm{\bB_N}^2 \opnorm{\bQ_k^-}^2 \opnorm{\bM_{\hatbU{k}_T}} \Fnorm{\hatbU{k}_T - \bQ_k \bU_T }^2 
        \\ & = (NT)^{-1} O_{\rm P}(n) O_{\rm P}(1) O_{\rm P}(1) O_{\rm P}(T/C_{N,T}^2) 
        \\ & = O_{\rm P}\left(C_{N,T}^{-2}\right).
    \end{align*}
    Turning to the last summand of  \eqref{eq:bound_Vuk_Vur_proof_steps}, we have 
    \begin{align*}
    S_3    & = \abs{ (NT)^{-1} \trace[ \operator{ \bxi_{NT} } (\b P_{\bU_T} - \b P_{\hatbU{k}_T})\operator{ \bxi_{NT} }^\adjoint ] } 
         \\ & \leq (NT)^{-1} \opnorm{\bxi_{NT}}^2 \tracenorm{ \b P_{\bU_T} - \b P_{\hatbU{k}_T} }
         \\ & \leq (NT)^{-1} O_{\rm P}(NT C_{N,T}^{-1}) (r + k)
     = O_{\rm P}\left(C_{N,T}^{-1}\right).
\end{align*}
$\,$\vspace{-12.5mm}

\end{proof}

\begin{proof}[{Proof of Theorem~\ref{thm:consistency_number_of_factors}}]
    Recall that $\IC(k) = V(k,\hatbU{k}_T)+  k \, g(N,T)$.
    We will prove that 
    \[
        \lim_{N,T \to \infty} \pp(\IC(k)< \IC(r))=0, \quad k\neq r, \; k\leq k_{max}.
    \]
    We have
    \[
    \IC(k)- \IC(r)= V(k,\hatbU{k}_T)-V(r,\hatbU{r}_T)-(r-k)g(N,T). \]
Thus, it is sufficient to prove that
    \[
        \lim_{N,T \to \infty} \pp\left( V(k,\hatbU{k}_T)-V(r,\hatbU{r}_T)<(r-k)g(N,T)\right)=0.
    \]
Let us split the proof into the two cases $k<r\leq k_{max}$ and $r<k\leq k_{max}$. For the
    first case,
    \begin{align*}
        V(k,\hatbU{k}_T)-V(r,\hatbU{r}_T) & =[V(k,\hatbU{k}_T)-V(k,\b Q_{k}\bU_T)]
                                       \\ & \quad +[V(k,\b Q_{k}\bU_T)-V(r,\b Q_{r}\bU_T)]
                                       \\ &
                                        \quad
                                         + [V(r,\b Q_{r}\bU_T)-V(r,\hatbU{r}_T)].
    \end{align*}
The first and third terms are $O_{\rm P}(C_{N,T}^{-1})$ by Lemma~\ref{lemma:2}. For
    the second term, since~$\b Q_{r}\bU_T$ and $\bU_T$ are spanning the same row space
    (Lemma~\ref{lma:minimal_singular_value_Qk}),  
    $V(r,\b
    Q_{r}\bU_T)=V(r,\bU_T)$. Now,~$V(k,\b Q_{k}\bU_T)-V(r,\bU_T) = \Omega_{\rm P}(1)$ in view of  
    Lemma~\ref{lma:VQu_minus_Vu_bounded_below}. 
    We conclude that 
    \[
        \lim_{N,T \to \infty} \pp\left(V(k,\hatbU{k}_T)-V(r,\hatbU{r}_T)<(r-k)g(N,T)\right)=0,
    \]
    because $g(N,T) \to 0$. For $k>r$, 
    \begin{align*}
        \MoveEqLeft \lim_{N,T \to \infty} \pp\left(\IC(k)- \IC(r) < 0 \right) 
        \\ &= \lim_{N,T \to \infty} \pp\left(V(r,\hatbU{r}_T)-V(k,\hatbU{k}_T)>(k-r)g(N,T)\right) 
        \\ & \leq  \lim_{N,T \to \infty} \pp\left(V(r,\hatbU{r}_T)-V(k,\hatbU{k}_T)>g(N,T)\right). 
    \end{align*}
    By Lemma~\ref{lma:bound_Vuk_Vur}, $V(r,\hatbU{r}_T) - V(k,\hatbU{k}_T) =
    O_{\rm P}(C^{-1}_{NT})$  and thus, because $C_{N,T}\,~g(N,T)$ diverges, 
    \[
        \lim_{N,T \to \infty} \pp\left(\IC(k)- \IC(r)<0\right) =0.
    \]
    This concludes the proof. 
\end{proof}

\subsection{Proofs of Section~\ref{sub:average_error_bounds}}
\label{sec:proofs_average_bounds}

Recall that  
\begin{equation} \label{eq:defn_tildeR}
    \tildeR = \hat \bLambda^{-1} \tildebU{r}_T \bU_T^\adjoint \bB_N^\adjoint \bB_N/(NT) = \hat \bLambda^{-1} \bQ_r
\end{equation}
where $\hat \bLambda$ is the $r \times r$ diagonal matrix with diagonal entries $\hat \lambda_i$ defined in \eqref{eq:svdX}\linebreak   and~$\bQ_r \in \bbR^{r \times r}$ is defined in \eqref{eq:Qk}. 
\begin{proof}[Proof of Theorem~\ref{theorem:factors_average_bound}]
    Recall that $\hatbU{r}_T = \hat \bLambda \tildebU{r}_T$ and $\bQ_r = \hat \bLambda \tildeR$, where $\bQ_r$ is defined in \eqref{eq:Qk}. We have  
    \[
        \Fnorm{\tildebU{r}_T - \tildeR \bU_T} \leq \opnorm{ \hat \bLambda^{-1}}
        \Fnorm{\hatbU{k}_T - \bQ_r \bU_T}.
    \]
    The claim follows from applying Lemma~\ref{lma:r_th_sample_eigenvalue_away_from_zero} and Theorem~\ref{thm:consistency_hatbu}.
\end{proof}

\begin{proof}[Proof of Theorem~\ref{theorem:factors_average_bound_up_to_sign}]
    Recall that $\bchi_{NT}$ is the $N \times T$ matrices with $(i,t)$ entry~$\chi_{it}$. By   assumptions, for $N,T$ large enough, 
    \begin{equation}
        \label{eq:chi_transpose_chi}
        \operator{ \bchi_{NT} }^\adjoint \operator{\bchi_{NT}}/(NT) = \sum_{k=1}^r \lambda_k \bu_{(k)} \tensor \bu_{(k)}
    \end{equation}
    where the $\lambda_k$s are distinct and decreasing, and $\bu_{(k)} \in \bbR^T$ is the
    $k$th column of $\bU_T^\tp$. Note that~$\lambda_k$ and  $\bu_{(k)}$ depend on
    $N,T$, but we suppress this dependency in the notation. Notice in 
    particular that, given our identification
   constraints,~\eqref{eq:chi_transpose_chi} is in fact a spectral
    decomposition. We now recall the spectral decomposition
    $$\operator{\bX_{NT}}^\adjoint
    \operator{\bX_{NT}}/(NT) = \sum_{k \geq 1} \hat \lambda_k \hat {\b f_k} \tensor \hat {\b
    f_k}.$$
      Lemma~4.3 of \citet{bosq:2000} then yields
    \begin{align*}
        \MoveEqLeft
        \hnorm{ \hat{\b f}_k - \text{sign}\left(\sc{\hat{\b f}_{k}, {\bu_{(k)}}}\right) \bu_{(k)} }/\sqrt{T}  
        \\ &\qquad = O_{\rm P}\left( \opnorm{ \operator{\bX_{NT}}^\adjoint \operator{\bX_{NT}} -  \operator{\bchi_{NT}}^\adjoint \operator{\bchi_{NT}} }/(NT) \right)
    \end{align*}
    for $k=1,\ldots,r,$ since the gaps between  $\lambda_1, \ldots, \lambda_r$ 
    remain bounded from below by Assumption~\ref{assumption:identification_factors}.
    Now, 
    \begin{align*}
        \MoveEqLeft
        \opnorm{\operator{\bX_{NT}}^{\adjoint} \operator{\bX_{NT}}- \operator{\bchi_{NT}}^{\adjoint} \operator{\bchi_{NT}} }   
         \\ &\qquad\qquad \leq \opnorm{ \operator{ \bxi_{NT} }^{\adjoint}\operator{ \bxi_{NT} }} + 2\opnorm{ \bU_T }
        \opnorm{ \b B_{N}^{\adjoint} \operator{ \bxi_{NT} }};
    \end{align*}
  applying Lemmas~\ref{lma:Frobenius_Norm_of_xi},
    \ref{lma:largest_sample_eigenvalue_bounded}, and \ref{lma:frobenius_Norm_of_Bn_xi},
    we get
    \[ 
        N^{-1}T^{-1} \opnorm{\operator{\bX_{NT}}^{\adjoint} \operator{\bX_{NT}}- \operator{\bchi_{NT}}^{\adjoint} \operator{\bchi_{NT}} }
        = O_{p}(C_{N,T}^{-1})
    \] 
    as $N,T \to \infty$,
which completes the proof since the $k$th row of $\tildebU{k}_T$ is   $\hat {\b f_k}^\tp$ for $k=1,\ldots, r$. 
\end{proof}

Denote by $s_j[A]$ the $j$th largest singular value of a matrix $A$.  We
have the following result. 
\begin{Lemma}
    \label{lma:bound_singular_values_tildeR}
    Under Assumptions~\ref{assumption:a}, \ref{assumption:b}, \ref{assumption:c}, and~\ref{assumption:Bn_xi}, as $N,T \to \infty$,
    \[ 
        s_{1}[\tildeR] = O_{\rm P}(1) \quad         \text{and} \quad  s_{r}[\tildeR]
        = \Omega_{\rm P}(1) .
    \]
    In other words, $\tildeR$ has  bounded operator norm, is invertible, and its inverse has
    a bounded operator norm.
\end{Lemma}
\begin{proof}
    The first statement follows directly from 
    Lemma~\ref{lma:r_th_sample_eigenvalue_away_from_zero}. For the second statement,
    using the Courant--Fischer--Weyl minimax characterization of singular values
    \citep{hsing:eubank:2015}, we obtain 
    \[
        s_r[\tildeR ] \geq s_r[ \hat \bLambda^{-1} ] s_r[ \bQ_r ] = \Omega_{\rm P}(1),
    \]
    by Lemmas~\ref{lma:r_th_sample_eigenvalue_away_from_zero} and~\ref{lma:minimal_singular_value_Qk}.
\end{proof}

\begin{proof}[Proof of Theorem~\ref{theorem:loadings_average_bound}]
    We shall show that 
    \[
        \Fnorm{ \tildeB{r}_N - \bB_N \tildeR^{-1} } = O_{\rm P}\left(\sqrt{N/C_{N,T}^{(1+ \alpha)}}\right)
    \]
    (here and below, all rates are meant as $N,T \to \infty$)
    where $\tildeR$ is defined in \eqref{eq:defn_tildeR}, and is invertible by
    Lemma~\ref{lma:bound_singular_values_tildeR}; the desired result then
    follows, since
    \[
        \barB{r}_N - \bB_N \tildeR^{-1} \hat \bLambda^{-1/2} = (\tildeB{r}_N - \bB_N
        \tildeR^{-1})
        \hat \bLambda^{-1/2}
    \]
    and $\opnorm{ \hat \bLambda^{-1/2} } = O_{\rm P}(1)$ by
    Lemma~\ref{lma:r_th_sample_eigenvalue_away_from_zero}.

    First, notice that $\tildeB{r}_N = \operator{\bX_{NT}} (\tildebU{r}_T)^\adjoint /T $, so that  
    \begin{align*}
        \tildeB{r}_N & = \bB_N \bU_T (\tildebU{r}_T)^\adjoint/T + \operator{ \bxi_{NT} } (\tildebU{r}_T)^\adjoint/T
        \nonumber
                  \\ & = \bB_N \left(\tildeR^{-1} {\tildebU{r}_T} + \bU_T - \tildeR^{-1}
                  {\tildebU{r}_T}\right) (\tildebU{r}_T)^\adjoint/T + \operator{ \bxi_{NT} } (\tildebU{r}_T)^\adjoint/T \nonumber
                  \\ & = \bB_N \tildeR^{-1} + \bB_N\left( \bU_T - \tildeR^{-1} {\tildebU{r}_T}\right)
                  (\tildebU{r}_T)^\adjoint/T 
                  \\ & \quad + \operator{ \bxi_{NT} } (\tildebU{r}_T - \tildeR \bU_T)^\adjoint/T + \operator{ \bxi_{NT} } \bU_T^\adjoint \tildeR^\adjoint/T,
        \label{eq:theorem:consistency_loadings_eq1}
    \end{align*}
    where we have used the fact that ${\tildebU{r}_T} (\tildebU{r}_T)^\adjoint/T = \bI_r$.
    Hence,
    \begin{align*}
        \Fnorm{\tildeB{r}_N - \bB_N \tildeR^{-1}}  & \leq \frac{1}{T}\Big\{ \opnorm{\bB_N} \Fnorm{\bU_T - \tildeR^{-1} {\tildebU{r}_T}} \Fnorm{ {\tildebU{r}_T}}
                                             \\ & \quad + \opnorm{\operator{ \bxi_{NT} }} \Fnorm{ \tildebU{r}_T - \tildeR \bU_T} 
                                             \\ & \quad + \Fnorm{\operator{ \bxi_{NT} } \bU_T^\adjoint} \opnorm{\tildeR}\Big\}\eqdef S_1+S_2+S_3, \text{ say.}
    \end{align*}
    By Lemma~\ref{lma:bound_singular_values_tildeR} and Theorem~\ref{theorem:factors_average_bound},  we have
    \[
        \Fnorm{\bU_T - \tildeR^{-1} {\tildebU{r}_T}} \leq \opnorm{\tildeR^{-1}}
        \Fnorm{{\tildebU{r}_T} - \tildeR \bU_T} = O_{\rm P}(\sqrt{T}/C_{N,T});
    \]
   thus, $S_1$ is $O_{\rm P}\left(\sqrt{N/C_{N,T}^2}\, \right)$. By 
      Lemma~\ref{lma:Frobenius_Norm_of_xi} and
      Theorem~\ref{theorem:factors_average_bound},~$S_2$ is~$O_{\rm P}\left(\sqrt{N/C_{N,T}^3}\, \right)$. As for~$S_3$, it is~$O_{\rm P}\left(\sqrt{N/C_{N,T}^{1+\alpha}}\,\right)$ by Assumption~\ref{assumption:u_xi}. This completes the proof.
\end{proof}

\begin{proof}[Proof of Theorem~\ref{theorem:common_component_average_bound}]
    Recalling the definition \eqref{eq:defn_tildeR}, we have
    \[
        \Fnorm{ \hat \bchi_{NT} - \bchi_{NT} } \leq \Fnorm{ \tildeB{r}_N - \bB_N
        \tildeR^{-1} } \opnorm{{\tildebU{r}_T}} + \opnorm{\bB_N} \opnorm{\tildeR^{-1}}
        \opnorm{{\tildebU{r}_T} - \tildeR \bU_T}.
    \]
    The desired result follows from applying the results from the proofs of Theorems~\ref{theorem:factors_average_bound} and~\ref{theorem:loadings_average_bound}, and Lemma~\ref{lma:bound_singular_values_tildeR}.
\end{proof}

\subsection{Proofs of Section~\ref{sub:uniform_error_bounds}}
\label{sec:proofs_uniform_bounds}

\begin{Lemma}
    \label{lma:max_bxiNT_bxi_t}
   Let  Assumption~\ref{assumption:c} hold  and assume that, for some positive integer~$ \kappa$ and some $M_2 < \infty$,
    \[
        \ee \abs{ \sqrt{N}\left(  N^{-1}  \sc{\bxi_{t}, \bxi_{s}}   - \nu_N(t-s) \right) }^{2 \kappa} < M_2
    \]
    for all $N$, $t$, and $s \geq 1$.
    Then,
    \[
        \max_{t=1,\ldots, T} \rpnorm{ \operator{\bxi_{NT}}^\adjoint \bxi_t }^2 = O_{\rm P}(N \max(N, T^{1 + 1/ \kappa}))
    \]
    as $N,T \to \infty$.
\end{Lemma}
\begin{proof}
    We know, by the proof of Lemma~\ref{lma:Frobenius_Norm_of_xi}, that
    \[
        \rpnorm{ (NT)^{-1} \operator{\bxi_{NT}}^\adjoint \bxi_t}^2 \leq 2 T^{-2} \sum_{s=1}^T ( \nu_N(t-s) )^2 + 2 T^{-2} \sum_{s=1}^T \eta_{st}^2,
    \]
    where we recall that $\eta_{st}  \defeq   N^{-1} \sc{\bxi_{t}, \bxi_{s}} - \nu_{N}(t-s).$
    The first term is $O(T^{-2})$. For the second term, Jensen's inequality yields
    \[
        \eee{ ( N T^{-1} \sum_{s=1}^T \eta_{st}^2)^ \kappa } \leq T^{-1} \sum_{s=1}^T \ee (\sqrt{N}\eta_{st})^{2 \kappa} \leq M_2
    \]
    and Lemma~\ref{lma:sup_of_time_series} implies that
    \[
        \max_{t=1,\ldots, T} T^{-2 }\sum_{s=1}^T \eta_{st}^2 = O_{\rm P}( N^{-1} T^{-1+1/ \kappa})
    \]
    as $N,T \to \infty$.
    The proof follows by piecing these results together. 
\end{proof}

\begin{proof}[Proof of Theorem~\ref{thm:unif_bound_factors}]
    Let $\tildebu{r}_t, \hatbu{r}_t \in \bbR^r$ denote the $t$-th columns of $\tildebU{r}_T$ and~$\hatbU{r}_T$, respectively. Here and below, all rates are meant as $N,T \to \infty$. Following the initial steps and notation of the proof of Theorem~\ref{theorem:factors_average_bound}, we have, from Lemma~\ref{lma:r_th_sample_eigenvalue_away_from_zero}, 
    \begin{align*}
        \max_{t=1,\ldots, T} \rpnorm{ \tildebu{r}_t - \tilde \bR \bu_t }
         & \leq \opnorm{\hat \bLambda^{-1}} \,  \max_{t=1,\ldots, T} \rpnorm{ \hatbu{r}_t - \tilde \bQ \bu_t } 
      \\ & \leq O_{\rm P}(1) \max_{t=1,\ldots, T} (NT)^{-1} \Big\{ \rpnorm{ \tildebU{r}_T \operator{\bxi_{NT}}^\adjoint \bxi_t }
      \\ & \qquad + \rpnorm{ \tildebU{r}_T \bU_T^\adjoint \bB_N^\adjoint \bxi_t}  + \rpnorm{ \tildebU{r}_T \operator{\bxi_{NT}}^\adjoint \bB_N \bu_t } \Big\}.
    \end{align*}
    Using Lemma~\ref{lma:max_bxiNT_bxi_t}, we obtain 
    \begin{align*}
        \max_{t=1,\ldots, T} (NT)^{-1} \rpnorm{ \tildebU{r}_T \operator{\bxi_{NT}}^\adjoint \bxi_t} 
        & \leq (NT)^{-1} \opnorm{ \tildebU{r}_T } \max_{t=1,\ldots, T} \rpnorm{ \operator{\bxi_{NT}}^\adjoint \bxi_t} 
     \\ & = O_{\rm P}\left( \max\left\{ \frac{1}{\sqrt{T}}, \frac{T^{1/(2 \kappa)}}{\sqrt{N}}\right\} \right).
    \end{align*}
    For the second summand, 
    \begin{align*}
        \max_{t=1,\ldots, T} (NT)^{-1} \rpnorm{ \tildebU{r}_T \bU_T^\adjoint \bB_N^\adjoint \bxi_t } 
            & \leq T^{-1} \opnorm{ \tildebU{r}_T } \opnorm{ \bU_T^\adjoint }  \max_{t=1,\ldots, T} N^{-1/2} \rpnorm{ N^{-1/2} \bB_N^\adjoint \bxi_t } 
         \\ & = O_p(T^{1/(2 \kappa)}/\sqrt{N})
    \end{align*}
    by Lemma~\ref{lma:sup_of_time_series}.
    As for the third summand, 
    \begin{align*}
        \max_{t=1,\ldots, T} (NT)^{-1} \rpnorm{ \tildebU{r}_T \operator{\bxi_{NT}}^\adjoint \bB_N \bu_t }
        & \leq (NT)^{-1} \opnorm{ \tildebU{r}_T } \opnorm{ \operator{\bxi_{NT}}^\adjoint \bB_N } \max_{t=1,\ldots, T} \rpnorm{ \bu_t }
     \\ & = O_p(T^{1/(2 \kappa)}/\sqrt{N})
    \end{align*}
    by Lemmas~\ref{lma:frobenius_Norm_of_Bn_xi} and~\ref{lma:sup_of_time_series}.
    Therefore,
    \[
        \max_{t=1,\ldots, T} \rpnorm{ \tildebu{r}_t - \tilde \bR \bu_t } = O_p\left( \max\left\{ \frac{1}{\sqrt{T}}, \frac{T^{1/(2 \kappa)}}{\sqrt{N}} \right\}\right). 
    \]
    $\,$\vspace{-13mm}
    
\end{proof}

\begin{Lemma}
    \label{lemma:product_is_tau_mixing}
    Let Assumption~\ref{assumption:uniform_bound_loadings} hold.
    Then, for each $l=1,\ldots,r$ and $i \geq 1$, the process $\{ u_{lt} \xi_{it} : t \geq 1 \}$ is~$\tau$-mixing with coefficients $\tau(h) \leq \sqrt{2} \sqrt{c_u^2 + c_\xi^2} \alpha \exp(-(\theta h)^\gamma)$.
\end{Lemma}
\begin{proof}
    Fix $l$ and $i$ and, for simplicity, write $u_t$, $\xi_t$, $H$, and $\hnorm{\cdot}$ for   $u_{lt}$,  $\xi_{it}$, $H_i$, and~$\hnorm{\cdot}_{H_i}$, respectively;   let $H_0 \defeq \bbR \oplus H$. 
    First, notice that 
    $$\LipschitzSpace(H_0, c) = \{ \vfi: B_c(H_0) \to \bbR :
        \Lipnorm{\vfi} \leq 1 \},$$ 
        where $\Lipnorm{\vfi} \defeq \sup\{
        \rpnorm{\vfi(x) - \vfi(y)}/\hnorm{x-y}_{H_0} : x \neq y \}$,  
     is the Lipschitz norm.
    We extend the definition of the Lipschitz norm to mappings $\psi: B_c(H_0) \to
    H$ by replacing the norm $\rpnorm{\cdot}$ by $\hnorm{\cdot}$ in the definition: $\Lipnorm{\psi} \defeq \sup\{
    \hnorm{\psi(x) - \psi(y)}/\hnorm{x-y}_{H_0} : x \neq y \}$. 
    Let~$c \defeq \sqrt{c_u^2 + c_\xi^2}$, and define the product mapping $\psi:
    B_{c}(\bbR \oplus H) \to H$ that maps~$(u, \xi)$ to~$u\, \xi \in H$. Notice that
    $\psi(u_t, \xi_t) = u_t \xi_t$ is almost surely well defined since the norm
    of~$(u_t, \xi_t)$ is almost surely less than $c$. Furthermore,
    $\hnorm{u_t \xi_t} = \rpnorm{u_t}\hnorm{\xi_t} \leq c^2$ almost surely.
    Straightforward calculations show that $\Lipnorm{\psi} \leq \sqrt{2}c $.
    Let us now compute the $\tau$-mixing coefficients for the process $\{ u_t \xi_t \} \subset
    H$. We have
    \begin{align*}
        \tau(h) & = \sup_{} \Big\{ \eee{Z \{ \vfi(u_t \xi_t) - \ee \vfi(u_t \xi_t)\}} ~|~ Z \in L^1(\Omega, \mathcal{M}^{u \xi}_t, \pp),
             \\ & \hspace{7cm}  \ee \rpnorm{Z} \leq 1, \vfi \in \LipschitzSpace(H,c^2) \Big\}
            \\  & = \sup_{} \Big\{ \eee{Z \{ \vfi(\psi(u_t, \xi_t)) - \ee \vfi(\psi(u_t,  \xi_t))\}} ~|~ Z \in L^1(\Omega, \mathcal{M}^{\psi(u, \xi)}_t, \pp), 
             \\ & \hspace{7cm} \ee \rpnorm{Z} \leq 1, \vfi \in \LipschitzSpace(H,c^2) \Big\},
    \end{align*}
    by definition of $\psi$. Notice that $\Lipnorm{\vfi \circ \psi} \leq
    \Lipnorm{\vfi}\Lipnorm{\psi} \leq \sqrt{2}c$, where ``$\circ$'' denotes function composition. Hence, $2^{-1/2}c^{-1} (\vfi \circ \psi) \in \LipschitzSpace(H_0, c)$. Furthermore,  since
    $\psi$ is conti\-nuous,~$\mathcal{M}^{\psi(u, \xi)}_t \subset \mathcal{M}^{(u, \xi)}_t$. Therefore,
    \begin{align*}
        \tau(h) & \leq \sqrt{2}\, c\,  \sup\Big\{ \eee{Z \{ \tilde \vfi(u_t, \xi_t) - \ee \tilde \vfi(u_t, \xi_t)\}} ~|~ Z \in L^1(\Omega, \mathcal{M}^{(u, \xi)}_t, \pp),
                 \\ & \hspace{7cm}  \ee \rpnorm{Z} \leq 1, \tilde \vfi \in \LipschitzSpace(H_0, c) \Big\}
                 \\ & \leq \sqrt{2}\, c\, \alpha\, \exp(-(\theta h)^{\gamma}),
    \end{align*}
    as was to be shown.
\end{proof}

\begin{Lemma}
    \label{lma:uniform_bound_bxi_ut}
    Let  Assumption~\ref{assumption:uniform_bound_loadings} hold. 
    Then,
    \[
        \max_{i=1,\ldots, N} \Fnorm{ \singleProj_i \operator{ \bxi_{NT} } \bU_T^\adjoint / T } = O_{\rm P}( \log(N) \log(T)^{1/{2\gamma}} /\sqrt{T} )
    \]
    as $N,T \to \infty$.
\end{Lemma}
\begin{proof}
    Using the union upper bound, we get, for  $\nu > 2$,
    \begin{align}
        \MoveEqLeft \pp\left( \max_{i=1,\ldots, N} \Fnorm{ \singleProj_i \operator{ \bxi_{NT}} \bU_T^\adjoint / T}^2 \geq r \nu^2 \log(N)^2 \log(T)^{1/{\gamma}}/{T} \right) \nonumber
        \\ & \leq r N \max_{i=1,\ldots, N} \max_{l=1,\ldots, r} \pp\left(  \hnorm{ \frac{1}{T} \sum_{s=1}^T u_{ls} \xi_{is} }_{H_i} \geq \nu \log(N) \log(T)^{1/{2\gamma}}/\sqrt{T} \right) \label{eq:uniform_bound_bxi_ut}.
    \end{align}
    Letting $c \defeq \sqrt{c_u^2 + c_\xi^2}$, we almost surely have $\hnorm{ u_{lt} \xi_{it} }_{H_i} \leq c^2$ for all $t \geq 1$  and
    each process $\{ u_{lt} \xi_{it} : t \geq 1 \} \subset H_i$ is $\tau$-mixing with coefficients $\tau(h) = \sqrt{2}\, c\, \alpha \exp(-(\theta h)^{\gamma})$ by Lemma~\ref{lemma:product_is_tau_mixing}. 
    Using the Bernstein-type inequality for $\tau$-mixing time series in Hilbert spaces \citep[][Proposition~3.6 and Corollary~3.7]{blanchard2019concentration}, the right-hand side in  \eqref{eq:uniform_bound_bxi_ut} is bounded by
    \begin{align*}
        \MoveEqLeft 2 r N \exp\left\{ -\nu \log(N) \log(T)^{1/{2\gamma}} T^{-1/2} \sqrt{ \left\lfloor \frac{T \theta}{2} \frac{1}{\log(T \sqrt{2} \alpha \theta / c)^{1/\gamma}} \right\rfloor}/(34c^2)  \right\}
        \\ & \leq 2 r N \exp\left\{ -\nu \log(N) \log(T)^{1/{2\gamma}} T^{-1/2} \sqrt{  \frac{ T \theta}{\log(T \sqrt{2} \alpha \theta / c)^{1/\gamma}} }/(68 c^2)  \right\}
        \\ & = 2 r N \exp\left\{ -\nu \log(N) \sqrt{\theta} \left(  \frac{\log(T) }{\log(T \sqrt{2} \alpha \theta / c)} \right)^{1/(2 \gamma)} /(68 c^2)  \right\}
        \\ & \leq 2 r N \exp\left\{ -\nu \log(N) \sqrt{\theta}/(136 c^2)  \right\}
        \\ & = 2 r \exp\left\{\log(N)\left( 1 -  \nu  \sqrt{\theta}/(136 c^2) \right)  \right\}
        \\ & \leq 2 r \exp\left\{ 1 -  \nu  \sqrt{\theta}/(136 c^2)   \right\}
    \end{align*}
    where the first and second inequalities hold if $T$ is large enough,   the third one if $\nu$ is large enough and $N \geq 3$. The last expression can be made arbitrarily small for $\nu$ large enough, independently of $N,T$ large enough. This concludes the proof.
\end{proof}

\begin{proof}[Proof of Theorem~\ref{thm:unif_bound_loadings}]
    Mimicking the steps in the proof of Theorem~\ref{theorem:loadings_average_bound}, we have
    \begin{align*}
        \max_{i=1,\ldots, N} \Fnorm{\tildeb{r}_i - \bb_i \tilde \bR^{-1} } 
                & \leq T^{-1} \max_{i=1,\ldots, N} \Fnorm{\bb_i} \Fnorm{\bU_T - \tilde \bR^{-1} \tildebU{r}_T} \Fnorm{\tildebU{r}_T}
             \\ & \qquad + T^{-1} \max_{i=1,\ldots, N} \Fnorm{\singleProj_i \operator{\bxi_{NT}} } \Fnorm{\tildebU{r}_T - \tilde \bR \bU_T}
             \\ & \qquad + T^{-1} \max_{i=1,\ldots, N} \Fnorm{\singleProj_i \operator{\bxi_{NT}} \bU_T^\adjoint } \Fnorm{\tilde \bR}.
    \end{align*}
    From the proof of Theorem~\ref{theorem:factors_average_bound} and Assumption~\ref{assumption:Bn_xi}, the first summand is $O_{\rm P}(C_{N,T}^{-1})$ 
    (here and below, all rates are meant as $N,T \to \infty$).
    The second summand is also $O_{\rm P}(C_{N,T}^{-1})$ since
    \[
        \Fnorm{\singleProj_i \operator{ \bxi_{NT} } }^2 = \sum_{t=1}^T \hnorm{\xi_{it}}^2 \leq c^2_\xi T.
    \]
    The third summand is $O_{\rm P}(\log(N) \log(T)^{1/{2\gamma}}/\sqrt{T})$ by Lemmas~\ref{lma:uniform_bound_bxi_ut} and~\ref{lma:bound_singular_values_tildeR}. Piecing the results together, we get
    \[
        \max_{i=1,\ldots, N} \Fnorm{\tildeb{r}_i - \bb_i \tilde \bR^{-1}} = O_{\rm P}\left( \max\left\{ \frac{1}{\sqrt{N}}, \frac{\log(N)\log(T)^{1/{2\gamma}}}{\sqrt{T}} \right\} \right).
    \]
    $\,$\vspace{-13mm}
    
\end{proof}

\begin{proof}[Proof of Theorem~\ref{thm:uniform_bound_common_component}]
    We have
    \begin{align*}
         \hat \chi_{it} - \chi_{it}
                & =  \tildeb{r}_i \tilde \bu_t - \bb_i \bu_t
            \\  & = (\tildeb{r}_i  - \bb_i \tilde \bR^{-1} )  (\tilde \bu_t - \tilde \bR \bu_t) 
             \\ & \qquad + (\tildeb{r}_i  - \bb_i \tilde \bR^{-1} ) \tildeR \bu_t  +   \bb_i \tilde \bR^{-1} (  \tilde \bu_t - \tilde \bR \bu_t ).
    \end{align*}
    Therefore, using Lemmas~\ref{lma:largest_sample_eigenvalue_bounded}, \ref{lma:r_th_sample_eigenvalue_away_from_zero} and~\ref{lma:bound_singular_values_tildeR} and Theorems~\ref{thm:unif_bound_factors} and~\ref{thm:unif_bound_loadings}, we get
    \begin{align*}
        \MoveEqLeft \max_{t=1,\ldots T} \max_{i=1,\ldots, N} \hnorm{ \hat \chi_{it} - \chi_{it} }_{H_i}
                \\ & \leq  \max_{i=1,\ldots, N} \opnorm{\tildeb{r}_i  - \bb_i \tilde \bR^{-1} }  \max_{t=1,\ldots T} \rpnorm{ \tilde \bu_t - \tilde \bR \bu_t}
                \\ & \qquad + \max_{i=1,\ldots, N} \opnorm{\tildeb{r}_i  - \bb_i \tilde \bR^{-1} }  \opnorm{\tildeR} \max_{t=1,\ldots T} \rpnorm{ \bu_t  }
                \\ & \qquad + \max_{i=1,\ldots, N} \opnorm{\bb_i} \opnorm{\tilde \bR^{-1} } \max_{t=1,\ldots T} \rpnorm{  \tilde \bu_t - \tilde \bR \bu_t }
                \\ & = O_{\rm P}\left( \max\left\{ \frac{1}{\sqrt{N}}, \frac{\log(N)\log(T)^{1/{2\gamma}}}{\sqrt{T}} \right\} \right)  O_{\rm P}\left( \max\left\{ T^{-1/2}, T^{1/(2 \kappa)} N^{-1/2} \right\} \right)
             \\ & \qquad +  O_{\rm P}\left( \max\left\{ \frac{1}{\sqrt{N}}, \frac{\log(N)\log(T)^{1/{2\gamma}}}{\sqrt{T}} \right\} \right) O_{\rm P}(1) c_u
             \\ & \qquad + O(1) O_{\rm P}(1) O_{\rm P}\left( \max\left\{ T^{-1/2}, T^{1/(2 \kappa)} N^{-1/2} \right\} \right)
             \\ & = O_{\rm P}\left( \max\left\{ \frac{T^{1/(2 \kappa)}}{\sqrt{N} }, \frac{\log(N)\log(T)^{1/{2\gamma}}}{ \sqrt{N}\, T^{( \kappa-1)/(2 \kappa)}} , \frac{\log(N)\log(T)^{1/{2\gamma}}}{\sqrt{T}} \right\} \right)
                 \end{align*}
                 as $N,T \to \infty$
\end{proof}


\subsection{Background results and technical lemmas}

We are grouping here the technical lemmas that have been used in the previous sections. 

\begin{Lemma}
    \label{lma:bai_Ng_rth_largest_eigenvalue_of_covariance_diverges}
    Assume that $\eee{\bu_t \tensor \bu_t}$ is positive definite, and let Assumption~\ref{assumption:b}
    hold. If $\ee \bu_t =\b 0, \ee \bxi_t = \b 0$, $\ee \hnorm{\bxi_t}^2 < \infty$ and $\eee{\bu_t \tensor \bxi_t} = \b 0$,
    \[
        \lambda_r\left[ \ee{\bx_t \tensor \bx_t} \right] = \Omega(N) \quad \text{as } N \to \infty.
    \]
    If, in addition, $\sum_{j = 1}^\infty \opnorm{ \ee{\xi_{it} \tensor \xi_{jt}} } < M < \infty$ 
    for all $i \geq 1$, then
    \[
        \lambda_{r+1}\left[ \ee{\bx_t \tensor \bx_t} \right] = O(1) \quad \text{as } N \to \infty.
    \]
\end{Lemma}
\noindent In other words, the $r$-th largest eigenvalue of the covariance of $\bx_t$
    diverges and the~$(r+1)$-th  largest one remains bounded as $N \rightarrow
    \infty$, which implies that the covariance of $\bx_t$ satisfies the conditions of
    Theorem~\ref{thm:equiv_factormodel_eigenvalues}.

\begin{proof}
    Here and below, all rates are meant as $N \to \infty$.
    Since $\eee{\bu_t \tensor \bxi_t} = \b 0$, we get  
    \[
        \Sigma \defeq \eee{\bx_t \tensor \bx_t} =  \bB_N \eee{\bu_t \tensor
    \bu_t} \bB_N^\adjoint + \eee{\bxi_t \tensor \bxi_t},
\]
and $\Sigma$ is compact and self-adjoint, since $\ee \hnorm{\bxi_t}^2 < \infty$ implies that $\eee{\bxi_t \tensor \bxi_t}$ is trace-class.
    Lemma~\ref{lma:eigenvalue_inequalities} yields 
    \[
        \lambda_r[\Sigma] \geq \lambda_r\left[ \bB_N \ee( \bu_t \tensor \bu_t ) \bB_N^\adjoint \right]
        \geq \lambda_r\left[\ee( \bu_t \tensor \bu_t) \right] \lambda_r[\bB_N \bB_N^\adjoint],
    \]
    where the second inequality follows from the Weyl--Fischer characterization of
    eigenvalues \citep{hsing:eubank:2015}. By assumption, the first term is
    bounded away from zero. For the second term, we have,     by Assumption~\ref{assumption:b}, 
    \[
        \lambda_r[ \bB_N \bB_N^\adjoint ] = \lambda_r[ \bB_N^\adjoint \bB_N ] = \Omega(N).
    \]
    For the second statement, using Lemma~\ref{lma:eigenvalue_inequalities}, we
    get
    \[
        \lambda_{r+1}[ \ee{\bx_t \tensor \bx_t} ] \leq \lambda_{r+1}[ \bB_N \ee( \bu_t \tensor
        \bu_t) \bB_N^\adjoint ] + \lambda_1[\ee \bxi_t \tensor \bxi_t ] \leq \lambda_1[\ee \bxi_t \tensor \bxi_t ]
    \]
    since $\bB_N \ee( \bu_t \tensor \bu_t) \bB_N^\adjoint$ has rank at most $r$. We want to show
    that $\opnorm{ \ee( \bxi_t \tensor \bxi_t) }$ is~$O(1)$. We will show that, for any
    norm $\hnorm{\cdot}_*$ on the operators $\bounded{\bH_N}$ that is a matrix norm, that
    is, satisfies,    for all $A,B \in \bounded{\bH_N}$ and  $\gamma \in \bbR$, 
    \begin{enumerate}
        \item[(i)] $\hnorm{A}_* \geq 0$ with equality if and only if $A = \b 0$, 
        \item[(ii)] $\hnorm{\gamma A}_* = \abs{\gamma} \hnorm{A}_*$,
        \item[(iii)] $\hnorm{A + B}_* \leq \hnorm{A}_* + \hnorm{B}_*$,
        \item[(iv)] $\hnorm{AB}_* \leq \hnorm{A}_* \hnorm{B}_*$,
    \end{enumerate}
  we have $|\lambda_1[A] | \leq
    \hnorm{A}_*$ if $A$ is compact and self-adjoint. 
    Indeed, since $A$ is compact and self-adjoint, $A x = \lambda_1[A] x$ for some non-zero $x \in \bH_N$. Then $A x \tensor x = \lambda_1[A] x \tensor x$, and thus 
    \begin{align*}
        \abs{\lambda_1[A]} \hnorm{x \tensor x}_* = \hnorm{(\lambda_1[A] x) \tensor x}_* = \hnorm{(A x) \tensor x}_*  &= \hnorm{A (x \tensor x)}_* 
        \\ & \leq \hnorm{A}_* \hnorm{x \tensor x}_*, 
    \end{align*}
hence  $\abs{\lambda_1[A] } \leq \hnorm{A}_*$. 
    To complete the proof, we still need to show  that 
    \[
        \hnorm{A}_* \defeq \max_{i=1,\ldots, N} \sum_{j=1}^N \opnorm{ \singleProj_i A \singleProj_j^\adjoint },
    \]
 where $\singleProj_i\!:\! \bH_N \to H_i$, $i = 1,\ldots, N$,   is the canonical
    projection mapping $(v_1,\ldots, v_N)^\tp\! \in~\!\bH_N$ to~$v_i \in~\!H_i$,   is a matrix norm. 
    Once we notice that $\sum_{i=1}^N
    \singleProj_i^\adjoint \singleProj_i = \identity_N$,      the identity operator in~$\bounded{\bH_N}$,
 (i), (ii), and (iii) are straightforward. Let us show that (iv) also holds.  For any~$A,B \in
    \bounded{\bH_N}$, we have
    \begin{align*}
        \hnorm{AB}_* & = \max_{i=1,\ldots, N} \sum_{j=1}^N \opnorm{\singleProj_i AB \singleProj_j^\adjoint}
                  \\ & = \max_{i=1,\ldots, N} \sum_{j=1}^N \opnorm{\singleProj_i A \left(\sum_{l=1}^N \singleProj_l^\adjoint \singleProj_l\right) B \singleProj_j^\adjoint}
                  \\ & \leq \max_{i=1,\ldots, N} \sum_{j=1}^N \sum_{l=1}^N \opnorm{\singleProj_i A  \singleProj_l^\adjoint \singleProj_l B \singleProj_j^\adjoint}
                  \\ & \leq \max_{i=1,\ldots, N} \sum_{l=1}^N \opnorm{\singleProj_i A  \singleProj_l^\adjoint} \sum_{j=1}^N \opnorm{\singleProj_l B \singleProj_j^\adjoint}
                  \\ & \leq \left(\max_{i=1,\ldots, N} \sum_{l=1}^N \opnorm{\singleProj_i A  \singleProj_l^\adjoint}\right) \left( \max_{l=1,\ldots, N} \sum_{j=1}^N \opnorm{\singleProj_l B \singleProj_j^\adjoint}\right)
                   = \hnorm{A}_* \hnorm{B}_*.
    \end{align*}
    This completes the proof.
\end{proof}

\begin{Lemma} \label{lma:norm_bu_bxi}
    Let Assumption~\ref{assumption:a} hold. Assume that 
    $$\eee{\sc{\bxi_t, \bxi_s} u_{lt}u_{ls}} = \eee{\sc{\bxi_t,
    \bxi_s}} \eee{ u_{lt}u_{ls}}$$
     for  all~$l = 1,\ldots, r$ and 
    $s,t \in \bbZ$ and  that $\sum_{t \in \bbZ} \abs{\nu_N(t)} < M < \infty$ for all $N \geq 1$. 
    Then, Assumption~\ref{assumption:u_xi} holds with $\alpha=1$.
\end{Lemma}
\begin{proof}
    We have 
    \[ \Fnorm{\bU_T \operator{ \bxi_{NT} }^\adjoint}^2  = \trace\left[ \bU_T \operator{ \bxi_{NT} }^\adjoint \operator{ \bxi_{NT} }
        \bU_T^\adjoint\right] = \sum_{l=1}^r \sum_{s,t=1}^T u_{lt}u_{ls} \sc{\bxi_t, \bxi_s},
    \]
    and thus 
    \begin{align*}
        \ee \Fnorm{\bU_T \operator{ \bxi_{NT} }^\adjoint}^2  & = \sum_{l=1}^r \sum_{s,t=1}^T \eee{u_{lt}u_{ls}} \eee{ \sc{\bxi_t, \bxi_s}},
                                      \\ & \leq r N \left( \max_l \ee u_{lt}^2 \right) \sum_{s,t=1}^T \abs{\nu_N(t-s)}
                                       = O(NT)
    \end{align*}
    as $N,T \to \infty$
   since, because    $(\bu_t)_t$ is
    stationary, $\ee u_{lt}^2$ does not depend on $t$.  Hence, 
    \[
        \Fnorm{\bU_T \operator{ \bxi_{NT} }^\adjoint}^2 = O_{\rm P}(NT) \leq O_{\rm
        P}(NT^2 C_{N,T}^{-2})
    \]
as $N,T \to \infty$
\end{proof}

The next Lemma tells us that the singular values of compact operators are stable under compact perturbations.
\begin{lma}
    \citep[Chapter~7]{Weid:80}
    \label{lma:perturbation_of_singular_values}
    Let $A,B:H_1 \rightarrow H_2$ be compact operators between two separable Hilbert
    spaces $H_1$ and $ H_2$, with the singular value decompositions
    \[
        A = \sum_{j \geq 1} s_j[A] u_j \tensor v_j \quad\text{and}\quad 
        B = \sum_{j \geq 1} s_j[B] w_j \tensor z_j,
    \]
    where $(s_j[A])_j$ and~$(s_j[B])_j$  are the singular values of $A$ and $B$, respectively, arranged in decreasing order. 
    Then, 
    \[
        \abs{s_j[A] - s_j[B]} \leq \opnorm{A-B}, \quad \text{for all } j \geq 1.
    \]
 \end{lma}

\begin{lma}
    \label{lma:trace_of_product_inequality}
    Let $A,B \in \bbR^{r \times r}$ be symmetric positive semi-definite matrices, and let~$\lambda_j[A]$, $\lambda_j[B]$ denote their respective $j$-th largest eigenvalues. Then, 
    \[
        \trace(AB) \geq \lambda_1[A] \lambda_r[B].
    \]
\end{lma}
\begin{proof}
    Since $\trace(AB) = \trace(A^{1/2}B A^{1/2})$ and $ A^{1/2}B A^{1/2}$ is
    symmetric positive semi-definite, 
    \begin{align*}
        \trace(BA) \geq \sup_{v: v^\tp v = 1} \sc{A^{1/2} B A^{1/2}v, v} &= \sup_v \sc{B
        A^{1/2}v, A^{1/2}v} 
        \\ &\geq \sup_v \lambda_r[B] \sc{A^{1/2}v, A^{1/2}v} 
        = \lambda_r[B] \lambda_1[A].
    \end{align*}
$\,$\vspace{-13.5mm}

\end{proof}

\begin{lma}
    \label{lma:largest_eigenvalue_of_difference_of_two_spd_matrices_with_unequal_rank}
    Let $A,B \in \bbR^{r \times r}$ be  symmetric positive semi-definite matrices, and let~$\lambda_j[A]$ denote the $j$-th largest eigenvalue of $A$. If $\mathrm{rank}(B) <
    \mathrm{rank}(A)$, then
    \[
        \lambda_1[A-B] \geq \lambda_r[A].
    \]
\end{lma}
\begin{proof}
   Since the rank of $B$ is less than the rank of $A$, there exists a vector $v_0$ of
   norm $1$ such that $Bv_0 = 0$. Therefore,
   \begin{align*}
       \lambda_1[A-B] & = \sup_{v : v^\adjoint v = 1} \sc{(A-B)v, v} 
        \geq \sc{(A-B)v_0,
       v_0} 
       \\ &= \sc{Av_0, v_0} 
       \geq \lambda_r[A] \sc{v_0, v_0} 
       = \lambda_r[A],
   \end{align*}
   as was to be shown.
\end{proof}

\begin{lma}
  \label{lma:eigenvalue_inequalities}
  Let $D, E \in \compact{H}$ be symmetric positive semi-definite operators on a
  separable Hilbert space $H$, and let
  $\lambda_s[C]$ denote the $s$-th largest eigenvalue of an operator~$C \in \compact{H}$.
  \begin{enumerate}
    \item[(i)] Letting $F = D + E$, we have,  for all $i \geq 1$, 
      \end{enumerate}
      \begin{equation*}
          \lambda_{i}[F] \leq \min(\lambda_1[D] + \lambda_i[E], \lambda_i[D] +
          \lambda_1[E]) \;\text{and}\; 
          \max(\lambda_i[D], \lambda_i[E]) \leq \lambda_i[F].
      \end{equation*}
   \begin{enumerate} 
    \item[(ii)] Let $G$ be a compression of $D$, meaning that $G = PDP$ for some orthogonal
      projection operator $P \in \bounded{H}$ ($P^2 = P = P^\adjoint$). Then, 
      \[
          \lambda_i[G] \leq \lambda_i[D] \quad \text{for all $i \geq 1$}.
      \]
  \end{enumerate}
  \begin{proof}
    This is a straightforward consequence of the Courant--Fischer--Weyl minimax characterization of eigenvalues of compact
    operators, see, e.g.\ \citet{hsing:eubank:2015}.
  \end{proof}
\end{lma}

\begin{Lemma}
    \label{lma:sup_of_time_series}
    Let $(Y_t)_{t =1,2, \ldots,}
    $ be a sequence of non-negative random variables satisfying $\sup_{t \geq 1} \ee {Y_t}^\alpha < \infty$ for some $\alpha \geq 1$. Then, 
    \[
        \max_{t=1,\ldots, T} {Y_t} = O_{\rm P}(T^{1/\alpha}) \quad \text{as } T \to \infty.
    \] 
    If $\sup_{t} \ee \exp( {Y_t} ) < \infty$, then 
    \[
        \max_{t=1,\ldots, T} {Y_t} = O_{\rm P}(\log(T)) \quad \text{as } T \to \infty.
    \] 
\end{Lemma}
\begin{proof}
    Assume $\sup_t \ee Y_t^\alpha < \infty$.
    By the union bound, for any $M > 0$,
    \begin{align*}
        \pp( \max_{t=1,\ldots, T} Y_t \geq M T^{1/\alpha} ) 
                    & \leq \sum_{t=1}^T \pp( Y_t \geq M T^{1/\alpha})
                 \\ & \leq  T \sup_{t \geq 1} \pp( Y_t^\alpha \geq M^\alpha T )
                 \\ & \leq  T \sup_{t \geq 1} \ee Y_t^\alpha /(M^\alpha T) 
                  =  \sup_{t \geq 1} \ee Y_t^\alpha /M^\alpha, 
    \end{align*}
    where the third inequality follows from the Markov inequality.  The
    first statement of the Lemma follows, since the 
    last term does not depend on $T$ and can be made arbitrarily small for $M$ large enough. The second statement of the Lemma follows from similar arguments: if $M > 1$ and $T \geq 3$,
    \begin{align*}
        \pp( \max_{t=1,\ldots, T} Y_t \geq M \log(T) )
                    & \leq \sum_{t=1}^T \pp( Y_t \geq M \log(T) )
                 \\ & \leq  T \sup_{t \geq 1} \ee \exp(Y_t) \, \exp(- M \log(T) )
                 \\ & \leq  \sup_{t \geq 1} \ee \exp(Y_t) \, \exp( (1 - M) \log(T) )
                 \\ & 
                 \leq  \sup_{t \geq 1} \ee \exp(Y_t) \, \exp( 1 - M ).
    \end{align*}
    The last term can be made arbitrarily small for $M$ large enough.
    This concludes the proof.
\end{proof}

\section{Results without the mean zero assumption}
\label{sec:results_no_mean_zero}

\subsection{Notation}

Here we assume that $x_{it} - \mu_i$ follows a  functional factor model with $r$ factors and  mean zero
\[
    x_{it} = \mu_i + b_{i1}u_{1t} + \cdots + b_{ir} u_{rt} + \xi_{it},
\]
where
 $\mu_i = \ee x_{it} \in H_i$    by stationarity does not depend on $t$. 
 
 Put~$\hat \mu_{i} \coloneqq \sum_{t=1}^T x_{it}/T$, 
 $\bmu_N = (\mu_1,\ldots, \mu_N)^\tp \in \bH_N$, and~$\hat \bmu_N = (\hat \mu_1,\ldots, \hat \mu_N)^\tp \in \bH_N$. 
Let $\bmu_{NT}$ and 
  $\hat \bmu_{NT}$
be the $N \times T$ matrix with all columns equal to $\bmu_N \in \bH_N$ and  the~$N \times T$ matrix with all columns equal to $\hat \bmu_N \in \bH_N$, respectively. 
 Notice that 
\begin{equation}
    \hat \bmu_N - \bmu_N = \bB_N \left( \sum_{t=1}^T \bu_t/T \right) + \sum_{t=1}^T \bxi_t/T 
\end{equation}
and
\begin{equation}
    \label{eq:operator_hatBmu-Bmu}
    \operator{\hat \bmu_{NT} - \bmu_{NT}} = \bB_N (\bU_T \ones_T/T) \ones_T^\tp  + \operator{\bxi_{NT}} \ones_T \ones_T^\tp /T
\end{equation}
where $\ones_T\coloneqq (1, 1, \ldots, 1)^\tp \in \bbR^T$. 
 Using this notation, we can rewrite our $r$-factor model (without the mean zero assumption) as
\begin{equation}
    \label{eq:model_not_mean_zero}
    \bX_{NT} = \bmu_{NT} + \bchi_{NT} + \bxi_{NT} = \bmu_{NT} + \matrixOf{\bB_N \bU_T} + \bxi_{NT} 
\end{equation}

\subsection{Representation results}

Theorems~\ref{thm:equiv_factormodel_eigenvalues} and~\ref{thm:factormodel_uniqueness} trivially hold   with  $x_{it} - \mu_i$ instead of $x_{it}$ and   $\bx_t - \bmu_N$ instead of $\bx_t$.

\subsection{Estimators}%

The estimators for the factors, loadings, and common component without the mean zero assumption are defined as in Section~\ref{sub:estimation_of_factors_loadings_and_common_component}, with the empirically centered $\bx_t - \hat \bmu_N$ replacing  $\bx_t$; in order to simplify notation, however, we keep the same notation for these  estimators (of factors, loadings, etc.) based on  the centered observations~$\bx_t - \hat \bmu_N$.

\subsection{A fundamental technical result}

A central technical result, on which most of the proofs depend of for the case $\bmu_N = 0$, is Theorem~\ref{thm:consistency_hatbu}. We will now prove it for the case~$\bmu_N \neq \bf 0$.
First, let us introduce a new assumption. 
\begin{customass}{A$_0$} \label{assumption:u_weak_dependence} \mbox{}
    $\rpnorm{T^{-1}(\bu_1 + \cdots + \bu_T)} = O_{\rm P}(T^{-1/2})$ as $T \to \infty$.
\end{customass}
This is an extremely mild assumption: it implicitly assumes some form of weak dependence on the series $\{ \bu_t \}$. A sufficient condition for Assumption~\ref{assumption:u_weak_dependence} to hold is that the traces of the autocovariances of $\{\bu_t \}$ are absolutely summable, that is,
\[
    \sum_{h \in \bbZ} \rpnorm{ \trace( \eee{\bu_0 \tensor \bu_t} )} < \infty.
\]
Notice that Assumption~\ref{assumption:u_weak_dependence} implies that $\Fnorm{\bU_T \ones_T/T} = O_{\rm P}(T^{-1/2})$.

The following equation links $\hatbU{k}_T$ and $\tildebU{k}_T$:
\begin{equation}
    \label{eq:relation-u_hat-u_tilde_non_mean_zero}
    \hatbU{k}_T = \tildebU{k}_T \operator{\bX_{NT} - \hat \bmu_{NT}}^{\adjoint} \operator{\bX_{NT} - \hat \bmu_{NT}}/(NT).
\end{equation}
Let us also define the   $k \times r$  matrix
\begin{equation}
    \label{eq:Qk_non_mean_zero}
    \bQ_k \defeq \tildebU{k}_T \bU_T^\adjoint \bB_N^\adjoint \bB_N/(NT) \in \bbR^{k \times r}.
\end{equation}
Recall that~$C_{N,T} \defeq \min\{ \sqrt{N}, \sqrt{T} \}$ and let $\hatbu{k}_t \in \bbR^k$ be the $t$-th column of $\hatbU{k}_T$ 

\begin{Theorem}
    \label{thm:consistency_hatbu_non_mean_zero}

    Under Assumptions~\ref{assumption:u_weak_dependence}, \ref{assumption:a}, \ref{assumption:c}, \ref{assumption:Bn_xi}, for any fixed $k
    \in \{1, 2, \ldots \}$, 
    \[
        T^{-1} \Fnorm{\hatbU{k}_T - \bQ_k \bU_T}^2 = 
        \frac{1}{T}
        \sum_{t=1}^{T} \hnorm{ \hatbu{k}_{t} -\b Q_{k} \bu_{t} }^{2} =
        O_{\rm P}(C_{N,T}^{-2})
    \]
    as $N,T \to \infty$,
    where $C_{N,T} \defeq \min \{\sqrt{N},\sqrt{T} \}$. 
\end{Theorem}
\begin{proof}
    By \eqref{eq:relation-u_hat-u_tilde_non_mean_zero}, we have
    \begin{align*}
        \hatbU{k}_T &= \tildebU{k}_T \operator{\bX_{NT} - \hat \bmu_{NT} }^\adjoint \operator{\bX_{NT} - \hat \bmu_{NT} }/(NT)
                 \\ &= \tildebU{k}_T \operator{(\bX_{NT} - \bmu_{NT}) - (\hat \bmu_{NT} - \bmu_{NT}) }^\adjoint \operator{(\bX_{NT} - \bmu_{NT}) - (\hat \bmu_{NT} - \bmu_{NT}) }/(NT)
                 \\ &= \frac{1}{NT} \Big\{ \tildebU{k}_T \operator{\bX_{NT} - \bmu_{NT} }^\adjoint \operator{\bX_{NT} - \bmu_{NT} }
                 \\ &\phantom{=} - \tildebU{k}_T \operator{\bX_{NT} - \bmu_{NT}}^\adjoint \operator{\hat \bmu_{NT} - \bmu_{NT}}
                 \\ &
                 \phantom{=} - \tildebU{k}_T \operator{\hat \bmu_{NT} - \bmu_{NT}}^\adjoint \operator{\bX_{NT} - \bmu_{NT}} 
                 \\ &\phantom{=} + \tildebU{k}_T \operator{\hat \bmu_{NT} - \bmu_{NT}}^\adjoint \operator{\hat \bmu_{NT} - \bmu_{NT}} \Big\} 
                 \\ &\eqdef S_1 + S_2 + S_3 + S_4.
    \end{align*}
    By the proof of Theorem~\ref{thm:consistency_hatbu}, we know that
    \begin{align*}
        \hatbU{k}_T - \bQ_k \tildebU{k}_T = \tilde S_1 + S_2 + S_3 + S_4
    \end{align*}
    where 
    \begin{align*}
        \tilde S_1 =& = \frac{1}{NT}\tildebU{k}_T \operator{ \bxi_{NT} }^{\adjoint}\operator{ \bxi_{NT} } + \frac{1}{N T} \tildebU{k}_T \bU_T^\adjoint \bB_N^{\adjoint} \operator{ \bxi_{NT} } 
                 \\ & \qquad\qquad\qquad\qquad\qquad\qquad\qquad\! +\frac{1}{NT}\tildebU{k}_T \operator{ \bxi_{NT} }^{\adjoint} \bB_N \bU_T
    \end{align*}
    and 
    \[
        \Fnorm{\tilde S_1} = O_{\rm P}(\sqrt{T} C_{N,T}^{-1})\quad\text{ as $N,T \to \infty$.}
    \]

    Let us now bound $S_2, S_3, S_4$. Using \eqref{eq:operator_hatBmu-Bmu}, we get
    \begin{align*}
        - S_2 &= \tildebU{k}_T \operator{\bX_{NT} - \bmu_{NT}}^\adjoint \operator{\hat \bmu_{NT} - \bmu_{NT}}/(NT)
           \\ &= \tildebU{k}_T \left( \bB_N \bU_T + \operator{\bxi_{NT}} \right)^\adjoint \left(\bB_N (\bU_T \ones_T/T) \ones_T^\tp + \operator{\bxi_{NT}}\ones_T \ones_T^\tp/T\right)/(NT)
           \\ &= \tildebU{k}_T \bU_T^\adjoint \bB_N^\adjoint \bB_N (\bU_T \ones_T/T) \ones_T^\tp /(NT)
           \\ &\phantom{= } + \tildebU{k}_T \bU_T^\adjoint \bB_N^\adjoint \operator{\bxi_{NT}}\ones_T \ones_T^\tp T^{-1}/(NT)
           \\ &\phantom{= } + \tildebU{k}_T \operator{\bxi_{NT}}^\adjoint \bB_N (\bU_T \ones_T/T) \ones_T^\tp /(NT)
           \\ &\phantom{= } + \tildebU{k}_T \operator{\bxi_{NT}}^\adjoint \operator{\bxi_{NT}}\ones_T \ones_T^\tp T^{-1}/(NT)
           \\ & \eqdef a_1 + a_2 + a_3 + a_4.
    \end{align*}
    We have $\Fnorm{a_1} = O_{\rm P}(T) O_{\rm P}(N) O_{\rm P}(T^{-1/2}) \sqrt{T} / (NT) = O_{\rm P}(1)$, where we have used, in particular, Assumption~\ref{assumption:u_weak_dependence}. Using Lemma~\ref{lma:frobenius_Norm_of_Bn_xi}, we get~$\Fnorm{a_2} = O_{\rm P}(\sqrt{T/N})$ and~$\Fnorm{a_3} = O_{\rm P}(\sqrt{T/N})$. Finally, using Lemma~\ref{lma:Frobenius_Norm_of_xi}, we get $\Fnorm{a_4} = O_{\rm P}(\sqrt{T}/C_{N,T})$. Piecing these together, we get $\Fnorm{S_2} = O_{\rm P}(\sqrt{T}/C_{N,T})$. The bound $\Fnorm{S_3} = O_{\rm P}(\sqrt{T}/C_{N,T})$ is obtained similarly.

    Let us now turn to $S_4$. From \eqref{eq:operator_hatBmu-Bmu}, we obtain
    \begin{align*}
        S_4 &= \tildebU{k}_T \operator{\hat \bmu_{NT} - \bmu_{NT}}^\adjoint \operator{\hat \bmu_{NT} - \bmu_{NT}}/(NT)
         \\ &= \tildebU{k}_T \left( \bB_N (\bU_T \ones_T/T) \ones_T^\tp  + \operator{\bxi_{NT}} \ones_T \ones_T^\tp /T \right)^\adjoint \left( \bB_N (\bU_T \ones_T/T) \ones_T^\tp  + \operator{\bxi_{NT}} \ones_T \ones_T^\tp /T \right)/(NT)
         \\ &= \tildebU{k}_T \ones_T (\bU_T \ones_T/T)^\adjoint \bB_N^\adjoint \bB_N (\bU_T \ones_T/T) \ones_T^\adjoint/(NT)
         \\ &\phantom{= } + \tildebU{k}_T \ones_T (\bU_T \ones_T/T)^\adjoint \bB_N^\adjoint  \operator{\bxi_{NT}} \ones_T \ones_T^\tp /(NT^2)
         \\ &\phantom{= } + \tildebU{k}_T \ones_T \ones_T^\adjoint \operator{\bxi_{NT}}^\adjoint  \bB_N (\bU_T \ones_T/T) \ones_T^\tp /(NT^2)
         \\ &\phantom{= } + \tildebU{k}_T \ones_T \ones_T^\adjoint \operator{\bxi_{NT}}^\adjoint \operator{\bxi_{NT}} \ones_T \ones_T^\tp /(NT^3)
          \eqdef b_1 + b_2 + b_3 + b_4.
    \end{align*}
    Assumption~\ref{assumption:u_weak_dependence}, Lemmas~\ref{lma:Frobenius_Norm_of_xi} and \ref{lma:frobenius_Norm_of_Bn_xi} yield 
    $$\Fnorm{b_1} = O_{\rm P}(1/\sqrt{T}),\quad\Fnorm{b_2} = O_{\rm P}(1/\sqrt{N}),\quad\Fnorm{b_3} = O_{\rm P}(1/\sqrt{N}),$$ and $$\Fnorm{b_4} = O_{\rm P}(\sqrt{T}/ C_{N,T}).$$ 
    Hence,
    $\Fnorm{S_4} = O_{\rm P}(\sqrt{T}/C_{N,T})$. 
    Piecing all the results together completes the proof.
\end{proof}
Notice that the proof implies in particular that
\begin{equation}
    \Fnorm{\operator{\hat \bmu_{NT} - \bmu_{NT}}^\adjoint \operator{\hat \bmu_{NT} - \bmu_{NT}} } = O_{\rm P}(NT/C_{N,T})
\end{equation}
and
\begin{equation}
    \opnorm{\operator{\hat \bmu_{NT} - \bmu_{NT}}} = O_{\rm P}(\sqrt{NT/C_{N,T}}).
\end{equation}

\subsection{Consistent estimation of the number of factors}

All the results hold with the same rates once Assumption~\ref{assumption:u_weak_dependence} is assumed in addition to their current assumptions. The proofs follow by mimicking the steps of the existing proofs, and bounding quantities corresponding to terms involving the bias term $\hat \bmu_{NT} - \bmu_{NT}$ as in the proof of Theorem~\ref{thm:consistency_hatbu_non_mean_zero}. Details are left to the reader.

\subsection{Average error bounds}

All the results concerning the average error bounds hold with the same rates once Assumption~\ref{assumption:u_weak_dependence} is assumed in addition to the current assumptions.
The proof for the average error bound on the loadings needs to be modified by using the identity 
\begin{equation}
    \label{eq:tildeB-eigenvector-of-X}
    \tildeB{k}_N = \operator{\bX_{NT} - \hat \bmu_{NT} } \operator{\bX_{NT} - \hat \bmu_{NT} }^\adjoint \barB{k}_N /(NT)
\end{equation}
and fleshing out the equation as in the proof of Theorem~\ref{thm:consistency_hatbu_non_mean_zero}.
Details are left to the reader.

\subsection{Uniform error bounds}

All the results concerning the uniform error bounds hold with the same rates once Assumption~\ref{assumption:u_weak_dependence} is assumed in addition to the current assumptions.
The result of Theorem~\ref{thm:unif_bound_factors} is slightly different, though, with $\tilde \bR$ replaced by 
\[
    \check{\bR} = \hat \bLambda^{-1} \tildebU{r}_T \left[ \bU_T^\adjoint - \ones_T (\bU_T \ones_T/T)^\adjoint \right] \bB_N^\adjoint \bB_N/(NT).
\]
The proof follows along the same lines as the proofs of Theorems~\ref{thm:unif_bound_loadings} and~\ref{thm:consistency_hatbu_non_mean_zero}, noting in particular that many terms do not depend on $t$. 

The proof for the uniform error bound on the loadings needs to be modified by using the identity \eqref{eq:tildeB-eigenvector-of-X},
and a technical result, similar to Lemma~\ref{lma:uniform_bound_bxi_ut}, is needed to bound 
\[
    \max_{i=1,\ldots, N} \Fnorm{\singleProj_i \operator{ \bxi_{NT} } \ones_T/T }.
\]
The proof for the uniform error bound on the common component follows once it is noticed that $\opnorm{ \tilde \bR^{-1} - \check \bR^{-1}} = O_{\rm P}(T^{-1/2})$ as $N,T \to \infty$.
Details are left to the reader.

\section{Additional Simulations}

\subsection{Estimation with Misspecified Number of Factors}
\label{s:simulation_r_misspecified}

Comparing the empirical performance of the estimators of factors and loadings when the number $r$ of factors is misspecified is moot, since  these estimators, for $k$ and $k^\prime$ factors, are   nested. More precisely,  for~$k < k^\prime $,   the matrix $\tildebU{k}_T \in \bbR^{k \times T}$ based on the assumption of $k$ factors is a submatrix of its $k^\prime$-factor counterpart~$\tildebU{k^\prime }_T \in \bbR^{k^\prime  \times T}$. Therefore, the error rates obtained in Sections~\ref{sub:average_error_bounds} and~\ref{sub:uniform_error_bounds} remain valid for $k \leq r$:
\begin{align*}
    \min_{\bR \in \bbR^{k \times r}} \Fnorm{ \tildebU{k}_T - \bR \bU_T}/{\sqrt{T}}  & \leq \min_{\bR \in \bbR^{r \times r}} \Fnorm{ \tildebU{r}_T - \bR \bU_T}/{\sqrt{T}} 
                                                                                 \\ & = O_{\rm P}(C_{N,T}^{-1})  \quad \text{as } N,T \to \infty,
\end{align*}
while, for $k > r$, they are bounded from  below:
\[
    \min_{\bR \in \bbR^{k \times r}} \Fnorm{ \tildebU{k}_T - \bR \bU_T}/{\sqrt{T}} \geq 1, \quad k > r
\]
since the rows of $\tildebU{k}$ are orthogonal with norm $\sqrt{T}$. 

A similar fact holds for the loadings.

\begin{table}[t]
    \centering
    \small
    \begin{tabular}{|c|ccccc|ccccc|}
    \hline
    & \multicolumn{5}{c|}{\texttt{DGP1}} & \multicolumn{5}{c|}{\texttt{DGP2}} \\
    \hline
    & $k=1$& $k=2$& $k=r=3$& $k=4$& $k=5$& $k=1$& $k=2$& $k=r=3$& $k=4$& $k=5$ \\
    \hline
        $T=25$   & 0.217 & 0.150 & 0.154 & 0.218 & 0.277 & 0.210 & 0.134 & 0.129 & 0.187 & 0.241 \\ 
        $T=50$   & 0.206 & 0.118 & 0.086 & 0.127 & 0.166 & 0.199 & 0.104 & 0.064 & 0.099 & 0.132 \\ 
        $T=100$  & 0.200 & 0.103 & 0.052 & 0.081 & 0.107 & 0.193 & 0.089 & 0.032 & 0.054 & 0.075 \\ 
        $T=200$  & 0.198 & 0.097 & 0.036 & 0.056 & 0.075 & 0.192 & 0.083 & 0.016 & 0.031 & 0.045 \\ 
  \hline
      & \multicolumn{5}{c|}{\texttt{DGP3}} & \multicolumn{5}{c|}{\texttt{DGP4}} \\
      \hline
  $T=25$  & 0.532 & 0.911 & 1.327 & 1.729 & 2.113 & 0.470 & 0.779 & 1.148 & 1.511 & 1.865 \\ 
  $T=50$  & 0.381 & 0.567 & 0.819 & 1.076 & 1.325 & 0.328 & 0.444 & 0.649 & 0.866 & 1.080 \\ 
  $T=100$  & 0.307 & 0.373 & 0.508 & 0.686 & 0.860 & 0.258 & 0.261 & 0.351 & 0.490 & 0.628 \\ 
  $T=200$  & 0.271 & 0.271 & 0.321 & 0.457 & 0.588 & 0.224 & 0.169 & 0.173 & 0.271 & 0.367 \\ 
  \hline
    \end{tabular}
    \caption{ Mean Squared Errors $\phi_{N,T}(k)$ for  the estimation of   the common component $ \chi_{it}$  based on various misspecified values $k$ of the actual number  $r=3$ of factors, under different data-generating processes, averaged over 500 replications of the simulations, for $N=100$.}
    \label{table:rmisspecified}
\end{table}

The case  is different for the estimation of the common component. Considering the same data-generating processes as in Section~\ref{s:simu}, we computed the estimation errors 
\begin{equation}
    \label{eq:phiNTr}
    \phi_{N,T}(k) \defeq (NT)^{-1} \sum_{i=1}^N \sum_{t=1}^T \hnorm{ \chi_{it} - \hat{ \chi}^{(k)}_{it}}_{H_i}^2 
\end{equation}
resulting from assuming $k$ factors, $k=1,\ldots,5$. %
Notice that $\chi_{it}$ here  is independent of $k$ and only depends on the actual number of factors, namely $r=3$. The values of $\phi_{N,T}(k)$, averaged over 500 replications of the simulations, are given in Table~\ref{table:rmisspecified} for $N=100$.
For \texttt{DGP1} and \texttt{DGP2}, the minimal value of $\phi_{N,T}(k)$ is obtained for the correct value $k = 3$ of the estimate, except for \texttt{DGP1} and $T = 25$. The situation is different for \texttt{DGP3} and \texttt{DGP4}, where the amount of idiosyncratic variance is larger than in \texttt{DGP1} and~\texttt{DGP2}. For \texttt{DGP3}, underestimating  the number of factors ($k < 3$) yields smaller error values~($ \phi_{N,T}(k)<\phi_{N,T}(3)$,  even for $T=200$. This is due to the fact that it is harder to estimate the factors in that setting. For~$N=200$ and $T=500$  (simulations not shown here),  $k = 3$, however, minimizes~$\phi_{N,T}(k)$ again, as expected. For \texttt{DGP4}, the minimal error is almost attained at $k = 3$ for $T=200$; looking at the rates of decrease of the errors (at $T$ increases), it is clear that the minimum will be achieved for $T > 200$. Overall, for $N,T$ large enough, it seems that overestimating $r$ is less problematic than underestimating it. For small sample size (e.g.\ smaller values of $N$, results not reported here) or higher idiosyncratic variance, small values of $k$ yield better results. The interpretation is clear: we do not have enough data to estimate well a large number of parameters.

\section{Additional Results for the Forecasting of Mortality Curves}
\label{sec:additional_application}

In this section, we present the results obtained by implementing the methodology described in Section~\ref{sec:forecasting_mortality_curves} on the level curves (as opposed to the log curves). Similarly to \cite{gao:shang:2018}, the predicted  curves are the log curves and a prediction of the original level curves is obtained as the exponential of the   log prediction. The performabe assessment metrics presented in Table~\ref{t:mortality_forecasting2} are   based on the level curves. As explained in Section~\ref{sec:forecasting_mortality_curves}, we do not believe that getting back to the level curves (that is, inverting the log transformation before evaluating predictive performance) is a good idea,  but we nonetheless present these results  for the sake of comparison. Note that, even with this disputable metric, our method (TNH) still outperforms the other two (GSY, the \cite{gao:shang:2018} one, and IF, the componentwise forecasting one).

\begin{table}[ht]
    \centering
    \small
    \begin{tabular}{l||rrr|rrr||rrr|rrr|}
            & \multicolumn{6}{c||}{Female} & \multicolumn{6}{c|}{Male} \\
            & \multicolumn{3}{c|}{MAFE} & \multicolumn{3}{c||}{MSFE} & \multicolumn{3}{c|}{MAFE} & \multicolumn{3}{c|}{MSFE}  \\
            & GSY & IF &TNH & GSY &IF & TNH & GSY &IF & TNH & GSY &IF & TNH  \\
            \hline
        $h=1$  & 3.47 & 2.96 & \textbf{1.65} & 2.26 & 0.75 & \textbf{0.30} & 4.54 & 3.02 & \textbf{2.55} & 2.26 & 0.66 & \textbf{0.65} \\ 
        $h=2$  & 3.35 & 3.30 & \textbf{1.69} & 2.04 & 0.88 & \textbf{0.31} & 4.49 & 3.29 & \textbf{2.58} & 2.16 & 0.73 & \textbf{0.66} \\ 
        $h=3$  &3.36 & 3.51 & \textbf{1.70} & 2.28 & 0.95 & \textbf{0.32} & 4.43 & 3.54 & \textbf{2.62} & 2.20 & 0.79 & \textbf{0.68} \\ 
        $h=4$  & 3.57 & 3.65 & \textbf{1.73} & 2.39 & 0.99 & \textbf{0.33} & 4.43 & 3.75 & \textbf{2.69} & 2.34 & 0.83 & \textbf{0.72} \\  
        $h=5$  & 3.46 & 3.76 & \textbf{1.67} & 2.35 & 1.00 & \textbf{0.32} & 4.20 & 3.95 & \textbf{2.66} & 2.04 & 0.86 & \textbf{0.71} \\ 
        $h=6$  &3.54 & 3.94 & \textbf{1.61} & 2.57 & 1.05 & \textbf{0.33} & 4.15 & 4.27 & \textbf{2.73} & 1.95 & 0.96 & \textbf{0.76} \\  
        $h=7$  & 3.67 & 4.16 & \textbf{1.64} & 2.50 & 1.12 & \textbf{0.36} & 3.75 & 4.55 & \textbf{2.77} & 1.60 & 1.03 & \textbf{0.79} \\ 
        $h=8$  & 3.48 & 4.39 & \textbf{1.72} & 2.08 & 1.19 & \textbf{0.39} & 3.61 & 4.87 & \textbf{2.84} & 1.40 & 1.12 & \textbf{0.83} \\ 
        $h=9$  & 3.63 & 4.70 & \textbf{1.81} & 2.28 & 1.32 & \textbf{0.43} & 3.78 & 5.25 & \textbf{2.97} & 1.46 & 1.25 & \textbf{0.89} \\  
        $h=10$ & 4.08 & 4.91 & \textbf{1.87} & 3.20 & 1.38 & \textbf{0.46} & 3.91 & 5.53 & \textbf{2.95} & 1.47 & 1.32 & \textbf{0.88} \\ 
        \hline
        Mean   & 3.56 & 3.93 & \textbf{1.71} & 2.39 & 1.06 & \textbf{0.36} & 4.13 & 4.20 & \textbf{2.74} & 1.89 & 0.96 & \textbf{0.76} \\ 
        Median &3.51 & 3.85 & \textbf{1.70} & 2.31 & 1.02 & \textbf{0.33} & 4.17 & 4.11 & \textbf{2.71} & 2.00 & 0.91 & \textbf{0.74} \\ 
    \end{tabular}
    \caption{Comparison of the forecasting accuracies of our method (TNH ) with \cite{gao:shang:2018} (GSY) and the independent forecast (IF) on the Japanese mortality curves  (not their log transforms) for~$h=1,\ldots, 10$,  female and male curves. Boldface is used for the winning \emph{MAFE} and \emph{MSFE} values, which uniformly correspond to the TNH  forecast. \emph{MAFE} is multiplied by a factor 1000 and \emph{MSFE}  by a  factor 10000 for readability.}
    \label{t:mortality_forecasting2}
\end{table}

\end{document}